\documentclass[11pt,reqno]{amsart}

\usepackage{amsmath,amssymb,amsthm,mathtools,mathrsfs}
\usepackage{verbatim}
\usepackage[a4paper,margin=30mm]{geometry}
\usepackage[hidelinks]{hyperref}

\numberwithin{equation}{section}

\newtheorem{theorem}{Theorem}[section]
\newtheorem{proposition}[theorem]{Proposition}
\newtheorem{lemma}[theorem]{Lemma}
\newtheorem{corollary}[theorem]{Corollary}
\newtheorem{remark}[theorem]{Remark}
\newtheorem{claim}[theorem]{Claim}
\theoremstyle{definition}
\newtheorem{definition}[theorem]{Definition}
\theoremstyle{plain}

\def\g{\gamma}
\newcommand{\R}{\mathbb R}
\newcommand{\Hhalf}{\mathbb R^n_+}
\newcommand{\bdH}{\partial\mathbb R^n_+}

\newcommand{\M}{\mathcal M}
\newcommand{\MT}{\mathcal M_{\mathrm T}}
\newcommand{\ST}{S_{\mathrm T}}

\newcommand{\nuout}{\nu}
\newcommand{\cA}{\mathcal A}
\newcommand{\cE}{\mathcal{E}}
\newcommand{\cB}{\mathcal{B}}
\newcommand{\bD}{\mathcal D}
\newcommand{\calQ}{\mathcal Q}
\newcommand{\calL}{\mathcal L}
\newcommand{\gradS}{\nabla_{S}}
\newcommand{\dd}{\,d}
\newcommand{\OmB}{\Omega_B}

\def\ov{\overline}

\def\bs{\backslash}
\def\tl{\tilde}

\def\non{\nonumber}

\def\Hz{\mathcal{H}}

\def\p{\partial}
\def\e{\epsilon}
\def\a{\alpha}
\def\b{\beta}

\def\d{\delta}
\def\l{\lambda}

\def\vp{\varphi}

\newcommand{\T}{\text}

\newcommand{\Om}{\Omega}

\newcommand{\divg}{\operatorname{div}}
\newcommand{\spanop}{\operatorname{span}}

\begin{document}
\title[Sharp gradient stability for the Sobolev trace inequality]
{Sharp Gradient Stability for the Sobolev Trace Inequality}


\author[Xi-Nan Ma]{Xi-Nan Ma}
\address{Department of Mathematics, University of Science and Technology of China, Hefei, 230026, Anhui Province, China.}
\email{xinan@ustc.edu.cn}
\author[Yitian Zhang]{Yitian Zhang}
\address{Chengdu University of Information Technology, Shuangliu District, Chengdu, 610225, Sichuan Province, China.}
\email{zyitian@cuit.edu.cn}
\author[Yang Zhou]{Yang Zhou}
\address{School of Mathematical Sciences, University of Science and Technology of China, Hefei, 230026, Anhui Province, China.}
\email{zy19700816@mail.ustc.edu.cn}

\begin{abstract}
Let \(n\ge3\) and \(1<p<n\).  We prove a sharp quantitative stability estimate
for the critical Sobolev trace inequality on the upper half-space.  More
precisely, the Sobolev trace deficit controls the \(\max\{2,p\}\)-th
power of the gradient distance to the manifold of trace bubbles. A central part of the proof is the spectral
nondegeneracy of the trace bubbles: the first two eigenspaces of the
linearized weighted Steklov problem are exactly the amplitude, dilation,
and tangential translation modes.
\end{abstract}

\maketitle

\section{Introduction}
Sharp functional inequalities identify the optimal constant and the extremal functions for which equality occurs.  A quantitative
stability theorem asks for more: if a function nearly attains the sharp
constant, must it be close, in a natural distance, to the
manifold of extremals?  Such estimates play an important role in the
calculus of variations, geometric analysis, nonlinear elliptic
equations, and the study of critical evolution problems.  They have
motivated growing interest in quantitative stability for functional and
geometric inequalities and for critical points of the associated
variational problems; see, for instance,
\cite{BianchiEgnell1991,CianchiFuscoMaggiPratelli2009,FigalliNeumayer2019,FigalliZhang2022,CiraoloGatti2026,FLZ26,Neumayer2020,Ho2022,ZhangZhouZou2025,CiraoloFigalliMaggi2018,FigalliGlaudo2020,DengSunWei2025}.
Following this line of research, we investigate quantitative stability
for extremals of the Sobolev trace inequality on the upper half-space.
Our result extends the whole-space theorem of
\cite{FigalliZhang2022} to the trace setting and the \(p=2\) trace
theorem of \cite{Ho2022} to the full range \(1<p<n\).

\subsection{The whole-space model and history of quantitative stability}
One model example is the sharp Sobolev inequality on \(\R^n\). 
The question of quantitative stability for the Sobolev inequality was
first raised by Brezis and Lieb \cite{BrezisLieb1985}. We first
introduce some useful definitions.

Given \(n\ge3\) and \(1<p<n\), denote by \(\dot W^{1,p}(\R^n)\) the
closure of \(C_c^\infty(\R^n)\) with respect to the norm
\[
 \|u\|_{\dot W^{1,p}(\R^n)}
 :=
 \left(\int_{\R^n}|Du|^p\,dx\right)^{1/p}.
\]
The Sobolev inequality guarantees the existence of a positive constant
\(S=S(n,p)\) such that
\[
 \|Du\|_{L^p(\R^n)}
 \ge
 S\|u\|_{L^{p^*}(\R^n)},
\]
where
\[
 p^*:=\frac{np}{n-p}.
\]
We call the largest constant \(S\) satisfying this property the
\emph{optimal Sobolev constant}.

Let \(\M\) be the \((n+2)\)-dimensional manifold of all functions of the
form
\[
 v_{a,b,x_0}(x)
 :=
 \frac{a}
 {\left(1+b|x-x_0|^{\frac{p}{p-1}}\right)^{\frac{n-p}{p}}},
 \qquad
 a\in\R\setminus\{0\},\quad b>0,\quad x_0\in\R^n.
\]
It was shown in
\cite{Aubin1976,Talenti1976,CorderoNazaretVillani2004} that \(\M\)
is precisely the set of extremals for the Sobolev inequality and that
every \(v\in\M\) is a weak solution of
\begin{equation}
 \label{eq:EL}
 -\Delta_pv
 =
 S^p\|v\|_{L^{p^*}(\R^n)}^{p-p^*}|v|^{p^*-2}v,
\end{equation}
where
\(
 \Delta_pv=\operatorname{div}(|Dv|^{p-2}Dv)\).
In particular,
\[
 \|Dv\|_{L^p(\R^n)}
 =
 S\|v\|_{L^{p^*}(\R^n)}
 \qquad \forall v\in\M.
\]

In the case \(p=2\), write \(2^*:=2n/(n-2)\).
Brezis and Lieb \cite{BrezisLieb1985} raised the problem of strengthening the Sobolev inequality by adding a remainder term involving an appropriate distance
between \(u\) and \(\M\), together with a suitable decay rate.  Later, Bianchi and Egnell \cite{BianchiEgnell1991} settled this problem and
proved that there exists \(c=c(n)>0\) such that
\begin{equation}
    \label{eq:BE}
 \delta(u):=
 \frac{\|Du\|_{L^2(\R^n)}}{\|u\|_{L^{2^*}(\R^n)}}-S
 \ge
 c\inf_{v\in\M}
 \left(
  \frac{\|Du-Dv\|_{L^2(\R^n)}}{\|Du\|_{L^2(\R^n)}}
 \right)^2
 \qquad
 \forall u\in\dot W^{1,2}(\R^n)\setminus\{0\},
\end{equation}
which is optimal both in terms of the strength of the distance from
\(\M\), and in terms of the exponent \(2\) appearing on the right-hand
side. 

However, Bianchi--Egnell's method relies heavily on the Hilbert-space
structure available at \(p=2\) and does not extend directly to
\(p\ne2\).  Later, the authors of
\cite{CianchiFuscoMaggiPratelli2009} used symmetrization and optimal
transport to prove a stability estimate for every \(p\in(1,n)\), with
distance given by
\begin{equation}
 \label{eq:CFMP-distance}
 \inf_{v\in\M}
 \left(
  \frac{\|u-v\|_{L^{p^*}(\R^n)}}{\|u\|_{L^{p^*}(\R^n)}}
 \right)^\alpha
 \qquad
 \forall u\in\dot W^{1,p}(\R^n)\setminus\{0\},
\end{equation}
together with the explicit decay exponent
\[
 \alpha=\alpha_{\mathrm{CFMP}}
 :=
 \left[
  p^*\left(3+4p-\frac{3p+1}{n}\right)
 \right]^2.
\]
Although this exponent was not sharp, this was the first
stability result valid for the full range of \(p\). Building on these techniques, Figalli, Maggi, and Pratelli \cite{FigalliMaggiPratelli2013} established a sharp stability result in the special case \(p=1\).

For \(p\ne2\), the absence of a Hilbert-space structure left open the
question of whether the \(p\)-Sobolev deficit could control distance to
\(\M\) at the gradient level, the strongest natural distance one may
hope to control by \(\delta(u)\), as in the Bianchi--Egnell theorem.
 Gradient stability estimates were established first for \(p\ge2\) by Figalli and
Neumayer \cite{FigalliNeumayer2019} and then for the full range
\(1<p<n\) by Neumayer \cite{Neumayer2020}. Indeed, they proved the existence of a constant
\(c=c(n,p)>0\) such that
\begin{equation}
 \label{eq:gradient-stability}
 \delta(u)
 \ge
 c\inf_{v\in\M}
 \left(
  \frac{\|D(u-v)\|_{L^p(\R^n)}}{\|Du\|_{L^p(\R^n)}}
 \right)^\alpha
 \qquad
 \forall u\in\dot W^{1,p}(\R^n)\setminus\{0\}.
\end{equation}
However, the stability exponent $\alpha$ in the preceding gradient
estimates for $p\ne2$ was far from optimal.
The sharp gradient power
\(\max\{2,p\}\) was finally obtained by Figalli and Zhang
\cite{FigalliZhang2022}. Indeed, they carefully analyzed the Taylor
remainders
\[
 |X+Y|^p-|X|^p-p|X|^{p-2}X\cdot Y
\]
and
\[
 |a+b|^{p^*}-|a|^{p^*}-p^*|a|^{p^*-2}a b,
\]
and improved the expansion of
\(\|Dv+\varepsilon D\varphi\|_{L^p(\R^n)}^p\). Combining these delicate
estimates with a disturbed spectral gap estimate, Figalli and Zhang
proved that there exists a constant \(c=c(n,p)>0\) such that
 \eqref{eq:gradient-stability} holds with
 \(\alpha=\max\{2,p\}\).

There is a related stability theory for critical points of the
functionals associated with Sobolev-type inequalities.
 When $p=2$, instead of controlling the distance to the extremal
manifold by the deficit, one controls the distance to one or several
bubbles by the residual in the Euler--Lagrange equation, that is 
\begin{equation}
 \label{eq:intro-residual}
 \mathcal R(u)
 :=
 \bigl\|-\Delta u-|u|^{2^*-2}u\bigr\|_{\dot H^{-1}(\R^n)} .
\end{equation}
Struwe's global compactness theorem \cite{Struwe1984} gives, along a
Palais--Smale sequence,
\begin{equation}
 \label{eq:intro-Struwe}
 u_k
 =
 u_\infty+
 \sum_{j=1}^{N}U[\lambda_{j,k},z_{j,k}]
 +o_{\dot H^1}(1),
 \qquad
 \mathcal R(u_k)\longrightarrow0,
\end{equation}
with asymptotically separated bubble parameters. 
Quantitative one-bubble
and multi-bubble results, together with their dimension-dependent sharp
forms, were developed in
\cite{CiraoloFigalliMaggi2018,FigalliGlaudo2020,DengSunWei2025}.
Although the present paper concerns deficit stability rather than a
residual estimate, the two theories share two indispensable ingredients:
classification of the bubbles and the spectral gap estimate.  Their compactness mechanisms are rooted in
the Brezis--Lieb lemma, concentration--compactness, and profile
decompositions; see
\cite{BrezisLieb1983,Gerard1998,Solimini1995,TintarevFieseler2007,
PalatucciPisante2014} and the references therein.

\subsection{The trace problem and previous results}

On \(\Hhalf=\R^{n-1}\times(0,\infty)\), the critical trace inequality is
\begin{equation}
 \label{eq:intro-trace}
 \ST(n,p)
 \left(
  \int_{\bdH}|u|^{p_*}\,dy
 \right)^{1/p_*}
 \le
 \left(
  \int_{\Hhalf}|\nabla u|^p\,dx
 \right)^{1/p}.
\end{equation}
Here
\[
 p_*:=\frac{(n-1)p}{n-p},
\]
and \(u\in\dot W^{1,p}(\Hhalf)\), the completion of
\(C_c^\infty(\overline{\Hhalf})\) with respect to the seminorm
\(\|\nabla u\|_{L^p(\Hhalf)}\).
The equality cases are
\begin{equation}
 \label{eq:intro-trace-bubbles}
 \MT
 =
 \left\{
  aU_{\lambda,\xi}:\
  a\ne0,\ \lambda>0,\ \xi\in\R^{n-1}
 \right\},
 \quad
 U_{\lambda,\xi}(y,t)
 =
 \left(
  \frac{\lambda^{2/p}}
       {|y-\xi|^2+(t+\lambda)^2}
 \right)^{\frac{n-p}{2(p-1)}} .
\end{equation}
We denote by \(\MT^+\) the positive component of \(\MT\), obtained by
restricting the amplitude parameter to \(a>0\).
They solve
\begin{equation}
 \label{eq:intro-trace-EL}
 \begin{cases}
  \Delta_pU_{\lambda,\xi}=0&\text{in }\Hhalf,\\
  -|\nabla U_{\lambda,\xi}|^{p-2}\partial_tU_{\lambda,\xi}
  =
  \ST^p\|U_{\lambda,\xi}\|_{L^{p_*}(\bdH)}^{p-p_*}
  U_{\lambda,\xi}^{p_*-1}
  &\text{on }\bdH .
 \end{cases}
\end{equation}
The study of the optimal constant in \eqref{eq:intro-trace} and of
positive solutions of \eqref{eq:intro-trace-EL} goes back to Escobar's
seminal paper \cite{Escobar1988}.  Using a mass-transportation method,
Nazaret \cite{Nazaret2006} obtained the sharp constant for every
\(p\in(1,n)\) and showed that the functions in
\eqref{eq:intro-trace-bubbles} are extremals.  The full
characterization of equality cases used here is recorded in
\cite{MaggiNeumayer2017}.  For \(p=2\), the moving-spheres method was
used in \cite{HY94,O96,YYL03} to show that every positive solution of
\eqref{eq:intro-trace-EL} has the form
\eqref{eq:intro-trace-bubbles}.  For general \(1<p<n\), positive
finite-energy solutions of \eqref{eq:intro-trace-EL} were classified in
\cite{Zhou2024}.

As for stability of the Sobolev trace inequality, Ho \cite{Ho2022}
adapted the method of \cite{BianchiEgnell1991} to the trace setting.
Writing
\[
 2_*:=\frac{2(n-1)}{n-2},
\]
Ho's result is equivalently expressed, after a change of constant, in
the following normalized form for functions with nonzero trace:
\begin{equation}
 \label{eq:intro-Ho}
 \delta_{\mathrm T}(u)
 :=
 \frac{\|\nabla u\|_{L^2(\Hhalf)}}
      {\|u\|_{L^{2_*}(\bdH)}}-\ST(n,2)
 \ge
 c(n)\!
 \inf_{V\in\MT}
 \left(
  \frac{\|\nabla(u-V)\|_2}{\|\nabla u\|_2}
 \right)^2 .
\end{equation}
More recently, Zhang, Zhou, and Zou \cite{ZhangZhouZou2025} proved
Bianchi--Egnell-type stability for fractional trace inequalities.
A key ingredient is the reduction principle in
\cite[Theorem~3.3]{ZhangZhouZou2025}, which reduces the
stability of the fractional trace inequality to a stability estimate for the fractional
Sobolev inequality. For the classical trace inequality with \(p\ne2\), however,
the sharp gradient estimate remained open.

The purpose of this paper is to establish the sharp gradient stability
estimate for the classical Sobolev trace inequality throughout the full
range \(1<p<n\).  For
\(u\in\dot W^{1,p}(\Hhalf)\) with nonzero trace, define the Sobolev
trace deficit and the normalized gradient distance from the trace-bubble
manifold by
\[
 \delta_{\mathrm T}(u)
 :=
 \frac{\|\nabla u\|_{L^p(\Hhalf)}}
      {\|u\|_{L^{p_*}(\bdH)}}
 -\ST(n,p)
\]
and
\[
 d_{\mathrm T}(u,\MT)
 :=
 \inf_{v\in\MT}
 \frac{\|\nabla(u-v)\|_{L^p(\Hhalf)}}
      {\|\nabla u\|_{L^p(\Hhalf)}}.
\]
We prove the following sharp gradient stability estimate for the Sobolev trace inequality.
\begin{theorem}
 \label{thm:main}
 Let \(n\ge3\) and \(1<p<n\).  There exists a constant
 \(c=c(n,p)>0\) such that every
 \(u\in\dot W^{1,p}(\Hhalf)\) with nonzero trace satisfies
 \begin{equation}
  \label{eq:main-stability}
  \delta_{\mathrm T}(u)
  \ge
  c\,d_{\mathrm T}(u,\MT)^{\max\{2,p\}}.
 \end{equation}
 Moreover, the exponent \(\max\{2,p\}\) is optimal.
\end{theorem}
\begin{remark}
    \label{rem:sharpness}
The stability exponent \(\alpha=\max\{2,p\}\) is sharp, as in
\cite{FigalliZhang2022}.

 Fix \(v=U_{1,0}\in\MT\) and consider first \(u_i:=v(A_ix)\), where
 \(A_i\in\R^{n\times n}\) denotes the diagonal matrix
 \[
  A_i=\operatorname{diag}\left(1,\ldots,1,1+\frac1i\right).
 \]
 It is not difficult to check that \(\delta_{\mathrm T}(u_i)\) behaves as \(i^{-2}\),
 while the right-hand side of \eqref{eq:main-stability} behaves as
 \(i^{-\alpha}\), hence \eqref{eq:main-stability} cannot hold with
 \(\alpha<2\).
    
 On the other hand, fix \(\varphi\in C_c^\infty(B_1)\) such that
 \(\|D\varphi\|_{L^p(\R^n_+)}>0\), and consider
 \(\widetilde u_i:=v+\varphi(x_i+\,\cdot\,)\), where \(x_i\in\partial\R^n_+\) is
 a sequence of points escaping to infinity.  One can check that
 \[
  \|D\widetilde u_i\|_{L^p(\R^n_+)}^p
  =
  \|Dv\|_{L^p(\R^n_+)}^p
  +\|D\varphi\|_{L^p(\R^n_+)}^p+r_{i,1}
 \]
 and
 \[
  \|\widetilde u_i\|_{L^{p_*}(\partial\R^n_+)}^{p_*}
  =
  \|v\|_{L^{p_*}(\partial\R^n_+)}^{p_*}
  +\|\varphi\|_{L^{p_*}(\partial\R^n_+)}^{p_*}+r_{i,2},
 \]
 with
 \[
  |r_{i,1}|+|r_{i,2}|
  \le C\bigl(v(x_i)+|Dv(x_i)|\bigr)
  \le Cv(x_i)\longrightarrow0
  \qquad\text{as }i\to\infty.
 \]
 Hence, choosing a sequence \(\epsilon_i\to0\) such that
 \(v(x_i)\ll\epsilon_i\ll1\), the functions
 \(\widehat u_i:=v+\epsilon_i\varphi(x_i+\,\cdot\,)\) satisfy
 \[
  \|D\widehat u_i\|_{L^p(\R^n_+)}^p
  =
  \|Dv\|_{L^p(\R^n_+)}^p
  +\epsilon_i^p\|D\varphi\|_{L^p(\R^n_+)}^p+o(\epsilon_i^p)
 \]
 and
 \[
  \|\widehat u_i\|_{L^{p_*}(\partial\R^n_+)}^{p_*}
  =
  \|v\|_{L^{p_*}(\partial\R^n_+)}^{p_*}
  +\epsilon_i^{p_*}\|\varphi\|_{L^{p_*}(\partial\R^n_+)}^{p_*}
  +o(\epsilon_i^{p_*}).
 \]
 Thanks to these facts, one easily deduces that
 \(\delta_{\mathrm T}(\widehat u_i)\) behaves as \(\epsilon_i^p\), while
 the right-hand side of \eqref{eq:main-stability} behaves as
 \(\epsilon_i^\alpha\).  Thus \eqref{eq:main-stability} cannot hold with
 \(\alpha<p\).
\end{remark}

Thus the optimal power agrees with the whole-space exponent.  The
boundary geometry and the lack of radial symmetry, however, create new compactness and spectral
difficulties that prevent a direct adaptation of the Sobolev argument.
Section~\ref{sec:weighted-trace-compactness} begins by collecting the
notation used in the proof.

\subsection{Main difficulties and new ideas}

Let \(v\in\MT^+\) and write
\begin{equation}
 \label{eq:intro-modulation}
 u=v+h,
 \qquad
 \int_{\bdH}v^{p_*-2}hZ\,dy=0
 \quad\text{for every }Z\in T_v\MT .
\end{equation}
For \(X,Y\in\R^n\) and \(a\ge0\), \(b\in\R\), introduce the exact
remainders
\begin{align}
 \mathscr G_p(X,Y)
 &:=
 |X+Y|^p-|X|^p-p|X|^{p-2}X\cdot Y,
 \label{eq:intro-G}\\
 \mathscr B_q(a,b)
 &:=
 |a+b|^q-a^q-qa^{q-1}b.
 \label{eq:intro-B}
\end{align}
The remainder changes scale:
\begin{equation}
 \label{eq:intro-G-scales}
 \mathscr G_p(X,Y)
 \asymp_p
 \begin{cases}
  (|X|+|Y|)^{p-2}|Y|^2
  \asymp
  \min\{|Y|^p,|X|^{p-2}|Y|^2\},
  &1<p<2,\\[1mm]
  |X|^{p-2}|Y|^2+|Y|^p,
  &p\ge2 .
 \end{cases}
\end{equation}
The boundary remainder has two different forms:
\begin{align}
 \mathscr B_q(a,b)
 &\le
 \left(\frac{q(q-1)}2+\kappa\right)
 \frac{(a+C_\kappa|b|)^q}{a^2+b^2}\,b^2,
 &&1<q\le2,
 \label{eq:intro-B-subquadratic}\\
 \mathscr B_q(a,b)
 &\le
 \left(\frac{q(q-1)}2+\kappa\right)a^{q-2}b^2
 +C_\kappa|b|^q,
 &&q>2.
 \label{eq:intro-B-superquadratic}
\end{align}

After the normalization \(\|u\|_{L^{p_*}(\bdH)}=1\), $v=a U$ and $U:=U_{1,0}$, the
Euler--Lagrange equation cancels the first-order terms and gives the
lower bound when $\delta_{\mathrm T}(u)\ll 1$:
\begin{equation}
 \label{eq:intro-master-deficit}
C\delta_{\mathrm T}(u)\ge \int_{\Hhalf}|\nabla u|^p\,dx-\ST^p
 \ge
 \frac p2\mathcal G_v[h]
 -
 \frac p{p_*}\Lambda_v
 \int_{\bdH}\mathscr B_{p_*}(v,h)\,dy,
\end{equation}
where
\begin{equation}
 \label{eq:intro-Gv-Lambda}
 \mathcal G_v[h]
 :=
 \frac2p\int_{\Hhalf}
 \mathscr G_p(\nabla v,\nabla h)\,dx,
 \qquad
 \Lambda_v
 :=
 \ST^p\|v\|_{L^{p_*}(\bdH)}^{p-p_*}.
\end{equation}
Thus the problem reduces to comparing the two nonlinear remainders in
\eqref{eq:intro-master-deficit}.

Write $h=\varepsilon\varphi$ with
 $\|\nabla\varphi\|_{L^p(\Hhalf)}=1$.
The formal limit of \(\frac{1}{\varepsilon^2}\mathcal G_v(\varepsilon \varphi)\) is $\int_{\Hhalf}
 \mathcal A_v\nabla \varphi\cdot\nabla \varphi\,dx$, where
\begin{equation}
 \label{eq:intro-linearized-operator}
 \mathcal A_v
 =
 |\nabla v|^{p-2}
 \left(
  I+(p-2)
  \frac{\nabla v}{|\nabla v|}
  \otimes
  \frac{\nabla v}{|\nabla v|}
 \right).
\end{equation}
Also, the formal limit of \(\frac{1}{\varepsilon^2}\int_{\bdH}\mathscr B_{p_*}(v,\varepsilon\varphi)\,dy\) is \(p_*(p_*-1)/2\int_{\bdH} v^{p_*-2}\varphi^2\,dy\).

The main idea behind the proof consists of two steps.

(1) Prove the linear spectral gap estimate: for every
\(\phi\in\dot{W}^{1,2}(\Hhalf;|\nabla U|^{p-2})\) satisfying
\(\phi\perp T_U\MT\), one has
 \begin{align}
  &\int_{\Hhalf}
  \mathcal A_U\nabla\phi\cdot\nabla\phi\,dx
  \notag\\
  &\qquad\ge
  \left[
  (p_*-1)\ST^p
  +2\lambda_{\mathrm T}
  \right]\|U\|_{L^{p_*}(\bdH)}^{\,p-p_*}
  \int_{\bdH}
  U^{p_*-2}|\phi|^2\,dy.
 \label{eq:linear-spectral-gap}
 \end{align}
Here \(\phi\perp T_U\MT\) means
\[
 \int_{\bdH}U^{p_*-2}\phi Z\,dy=0
 \qquad\text{for every }Z\in T_U\MT.
\]

(2) Use compactness to show that there exists $\widetilde\lambda_{\mathrm T}>0$ such that for \(\varepsilon\) sufficiently
small,
\begin{equation}
 \label{step2}
 \mathcal G_v[\varepsilon\varphi]
 \ge
 \left(\frac2{p_*}+\widetilde\lambda_{\mathrm T}\right)\Lambda_v
 \int_{\bdH}\mathscr B_{p_*}(v,\varepsilon\varphi)\,dy
 \qquad
 \forall\,\varphi\perp T_v\MT
 \quad\text{with}\quad
 \|\nabla\varphi\|_{L^p(\Hhalf)}=1.
 \end{equation}

\noindent\textbf{The first difficulty: compactness.}
For \(p=2\), the Dirichlet energy is quadratic and its second variation
defines a fixed Hilbert structure.  For \(p\ne2\), a formal Taylor
expansion around a bubble \(U\) instead produces a weight
\(|\nabla U|^{p-2}\), while the boundary term produces
\(U^{p_*-2}\).  These weights are singular or degenerate at infinity,
depending on the range of \(p\), and the underlying half-space and its
boundary are both non-compact.  Moreover, a purely quadratic expansion
is not uniform,  especially when
\(p<2\), since the weighted norm $\dot W^{1,2}(\Hhalf;|\nabla U|^{p-2})$ cannot be controlled by
the available \(\dot W^{1,p}\) norm. When \(p_*\le2\), the boundary
remainder has the analogous obstruction.
Consequently the proof necessarily separates
\begin{equation}
 \label{eq:intro-three-regimes}
 1<p\le\frac{2n}{n+1}\ (p_*\le2),
 \qquad
 \frac{2n}{n+1}<p<2\ (p_*>2),
 \qquad
 2\le p<n .
\end{equation}

Analogous challenges in the whole-space model were overcome in
\cite{FigalliZhang2022} through delicate estimates of the remainders
\(\mathscr G_p\) and \(\mathscr B_{p_*}\).  We adopt their argument to
prove Step~(2), once Step~(1) has been established.

\noindent\textbf{The second difficulty: spectral gap estimate.}
 The spectral gap estimate \eqref{eq:linear-spectral-gap} is related to the following degenerate weighted Steklov-type eigenvalue
problem: 
\begin{equation}
 \label{eq:intro-stecklov}
 \int_{\Hhalf}\mathcal A_v\nabla\phi\cdot\nabla\psi\,dx
 =
 \mu\int_{\bdH}v^{p_*-2}\phi\psi\,dy .
\end{equation}
Both weights may degenerate or become singular, and the domain and
boundary are non-compact. We first adopt the method in \cite{FigalliZhang2022} to prove the compact embedding
\begin{equation}
 \label{eq:intro-compact-embedding}
 \dot W^{1,2}\!\left(\Hhalf;|\nabla U|^{p-2}\right)
 \hookrightarrow
 L^2\!\left(\bdH;U^{p_*-2}dy\right),
\end{equation}
which implies the spectrum is discrete: $\mu_0<\mu_1\le\mu_2\le\cdots$;
see Theorem~\ref{thm:weighted-trace-compact} and
Corollary~\ref{cor:discrete-spectrum}.

A variational argument identifies the first eigenvalue as
\(\mu_0=(p-1)\Lambda_U\), with eigenspace
\(E_0=\operatorname{span}\{U\}\), and the second eigenvalue as
\(\mu_1=(p_*-1)\Lambda_U\); see Theorem~\ref{thm:iden}.

Therefore, to prove \eqref{eq:linear-spectral-gap}, we only need to classify the second eigenspace:
\begin{equation}
 \label{eq:intro-low-spectrum}
 E_{1}=\operatorname{span}\{Z_0,Z_1,\ldots,Z_{n-1}\}.
\end{equation}
Here \(Z_0\) is the dilation mode and \(Z_1,\ldots,Z_{n-1}\) are the
tangential translation modes. 

In previously studied settings, the identification of the second
eigenspace often relies on radial symmetry of both the domain and the
weight.  Radial symmetry reduces the eigenvalue equation to ordinary
differential equations by separation of variables, after which
Sturm--Liouville theory can be used to classify the second eigenspace;
see, for example, \cite{FigalliNeumayer2019}.  In our setting, the
weight is not radially symmetric, even after the domain is transformed
to a ball as in \cite{Ho2022}.  Consequently, separation of variables
in full spherical coordinates does not directly classify the second
eigenspace of \eqref{eq:intro-stecklov}, and a different argument is
required.

We divide the proof of \eqref{eq:intro-low-spectrum} into four steps.
First, following \cite{Ho2022}, we transform the half-space \(\R^n_+\)
onto the ball
\(\Omega_B:=B\left(\frac12e_n,\frac12\right)\).
Write \[
  y=r \sigma,\,\,r=|y|,\,\,\sigma=(\sqrt{1-\tau^2}\,\theta,\tau),
  \,\,\theta\in S^{n-2},\,\,
  \tau:=\sigma_n,
\]
and define
\begin{align}\label{eq:intro-energy}
  \cE[\phi,\psi]
  :=\int_{\Omega_B} &r^{-\gamma}
  \left[
    (p-1)\phi_r\psi_r
    +\frac1{r^2}\gradS\phi\cdot\gradS\psi
  \right]\dd y,\\
\label{eq:intro-boundary}
  &\cB[\phi,\psi]
  :=\int_{\partial\Omega_B}r^{-\gamma}\phi\psi\dd S,
\end{align}
where \(\nabla_S\) is the spherical gradient on \(S^{n-1}\) and
\(\gamma=\frac{(n-1)(p-2)}{p-1}\).
Then the eigenvalue problem \eqref{eq:intro-stecklov} is equivalent to
\begin{equation}\label{eq:intro-eigen}
  \cE[\phi,\psi]=\lambda\cB[\phi,\psi]
\end{equation}
for every \(\psi\in W^{1,2}(\OmB;r^{-\g})\).
To make the change of spectral parameter explicit, set
\[
 \alpha_{\mathrm c}:=\frac{n-p}{p-1},
 \qquad
 C_0:=(p-1)\alpha_{\mathrm c}^{p-1},
 \qquad
 \lambda:=\frac{\mu-C_0}{\alpha_{\mathrm c}^{p-2}}.
\]
Thus the original first eigenvalue \(\mu_0=C_0\) becomes
\(\lambda_0=0\), with eigenfunction \(\phi_0=1\), while
\[
 \mu_1=(p_*-1)\alpha_{\mathrm c}^{p-1}
 =C_0+p\alpha_{\mathrm c}^{p-2}
\]
becomes the second eigenvalue \(\lambda_1=p\) of
\eqref{eq:intro-eigen}.
That is
\begin{equation}\label{eq:intro-gap-goal}
  \cE[\phi,\phi]\ge p\cB[\phi,\phi]
  \qquad\text{whenever}\qquad
  \cB[\phi,1]=0.
\end{equation}
To characterize the second eigenspace $E_1$, we need to prove that equality in \eqref{eq:intro-gap-goal} holds precisely on
\[
  \operatorname{span}\left\{y_1,\ldots,y_{n-1},y_n-\frac1p\right\}.
\]

Second, let
\[
\left\{Z_{k,j}: 1\le j\le \dim \mathcal{H}_k^{n-1}\right\}
\]
be an orthonormal basis of spherical harmonics of degree \(k\) on
\(S^{n-2}\) with respect to the normalized surface measure
\(d\sigma/|S^{n-2}|\).
For $\phi\in W^{1,2}(\Omega_B;r^{-\gamma})$, we can write
\begin{equation}
\phi(r,\tau,\theta)=\sum_{k=0}^\infty\sum_{j=1}^{\dim\mathcal{H}^{n-1}_k} u_{k,j}(r,\tau)Z_{k,j}(\theta)
 \label{eq:intro-expand}
\end{equation}
Thus the energy of \(\phi\) decomposes into different sectors:
\[
  \cE[\phi,\phi]=|S^{n-2}|\sum_{k,j}\cE_k[u_{k,j},u_{k,j}],
  \qquad
  \cB[\phi,\phi]=|S^{n-2}|\sum_{k,j}\cB_k[u_{k,j},u_{k,j}].
 \]
In Subsection~\ref{subsec}, we prove that every second eigenfunction
\(\phi\) satisfies
\[
 u_{1,j}(r,\tau)=c_jr\sqrt{1-\tau^2},
 \qquad
 u_{k,j}\equiv0\quad\text{for }k\ge2.
\]
Consequently,
\[
 \phi-u_0
 \in
 \operatorname{span}\{y_1,\ldots,y_{n-1}\}.
\]
It remains to classify \(u_0\); see
Theorem~\ref{ass:axisymmetric-gap}.

Third, we use a Pohozaev-type identity
(Lemma~\ref{l:pohozaev}) to reduce the classification of \(u_0\) to a
sharp one-dimensional inequality; see Proposition~\ref{p:ineq}.

Finally, in Subsection~\ref{subsec2}, we use singular
Sturm--Liouville theory to analyze the spectrum of the linearized
operator \(\calL\) associated with the one-dimensional inequality.  We
prove that \(\calL\) has exactly one negative eigenvalue and no zero
eigenvalue.  This identifies the equality cases in the sharp
one-dimensional inequality and completes the classification of \(u_0\).

\subsection{Outline of the proof}

In Section~\ref{sec:weighted-trace-compactness}, we introduce the
notation and formulate the linearized weighted Steklov problem.  We then
define the weighted energy space at the standard bubble, prove the
weighted trace inequality and compactness of the trace embedding,
establish discreteness of the Steklov spectrum, and reduce the spectral
gap estimate to the spectral nondegeneracy theorem
(Theorem~\ref{ass:spectral-nondegeneracy}).

In Section~\ref{sec:spectral-nondegeneracy-proof}, we prove
Theorem~\ref{ass:spectral-nondegeneracy} following the strategy described
in the discussion of the second difficulty.  A normalized conformal
transformation first reduces the Steklov problem on $\R_+^n$ to a
weighted equation on a ball.  We then decompose a second eigenfunction
$\phi$ into spherical-harmonic sectors and analyze the sectors of degree
$k\ge1$.  The $k=0$ sector is handled by the Pohozaev and
Sturm--Liouville arguments described above.  Combining the sector
estimates yields a strictly positive spectral gap on the orthogonal
complement of the tangent space.

In Section~\ref{sec:nonlinear-spectral-gap}, motivated by
\cite[Section~3]{FigalliZhang2022}, we work directly with the exact
Taylor remainders rather than replacing them prematurely by purely
quadratic bounds.  Combining compactness of normalized perturbations
with the linear spectral gap, we prove a uniform nonlinear spectral gap
near every trace bubble.

In Section~\ref{sec:proof-main}, we complete the proof of
Theorem~\ref{thm:main}.  Lions' concentration--compactness principle
\cite{Lions1985Part2} first reduces the global theorem to the local
regime.  As in \cite[Lemma~4.1]{FigalliZhang2022}, a modulation argument
then supplies the required tangent orthogonality.  We write
$u=v+\varepsilon\varphi$, expand
$\delta_{\mathrm T}(v+\varepsilon\varphi)$, and apply the nonlinear
spectral gap in each of the three ranges of $p$.

\section{The weighted trace space and compactness}
\label{sec:weighted-trace-compactness}

\subsection{Basic setting and the weighted trace space}

We first introduce some notation as in \cite{FigalliZhang2022}.
\begin{definition}[Weighted homogeneous Sobolev spaces]
 \label{def:weighted-spaces}
 Let \(\Om\subset\R^n\) be open.  Given \(q\ge1\) and a nonnegative
 locally integrable weight \(g_0:\Om\to[0,\infty)\), let
 \(L^q(\Om;g_0)\) be the space of measurable functions
 \(\phi:\Om\to\R\) such that
 \[
  \|\phi\|_{L^q(\Om;g_0)}
  :=\left(\int_\Om |\phi|^q g_0(x)\,dx\right)^{1/q}<\infty.
 \]
 For a nonnegative locally integrable weight \(g_1\) on \(\Om\), define
 \(\dot W^{1,q}(\Om;g_1)\) as the completion of
 \(C_c^\infty(\overline\Om)\) with respect to the homogeneous norm
 \[
  \|\phi\|_{\dot W^{1,q}(\Om;g_1)}
  :=\left(\int_\Om|D\phi|^qg_1(x)\,dx\right)^{1/q}.
 \]
 When \(g_1\equiv1\), we abbreviate
 \(\dot W^{1,q}(\Om;1)\) as \(\dot W^{1,q}(\Om)\).
\end{definition}

Let \(n\ge 3\), \(1<p<n\), and write points of the upper half-space as
\[
 x=(y,t)\in\Hhalf:=\R^{n-1}\times(0,\infty).
\]
Its boundary is identified with
\[
 \bdH=\R^{n-1}\times\{0\}\simeq\R^{n-1}.
\]
We use the two critical exponents
\begin{equation}
 \label{eq:critical-exponents}
 p^*:=\frac{np}{n-p},
 \qquad
 p_*:=\frac{(n-1)p}{n-p}.
\end{equation}
Thus \(p^*\) is the interior Sobolev exponent and \(p_*\) is the
critical trace exponent.  Notice that
\begin{equation}
 \label{eq:trace-threshold}
 p_*\le 2
 \quad\Longleftrightarrow\quad
 p\le \frac{2n}{n+1}.
\end{equation}

In the notation of Definition~\ref{def:weighted-spaces},
\(\dot{W}^{1,p}(\Hhalf)\) is equipped with the homogeneous norm
\[
 \|u\|_{\dot{W}^{1,p}(\Hhalf)}
 :=
 \|\nabla u\|_{L^p(\Hhalf)}.
\]
Every element of \(\dot{W}^{1,p}(\Hhalf)\) has a representative in
\(L^{p^*}(\Hhalf)\), and its trace belongs to
\(L^{p_*}(\bdH)\).  From now on, whenever a function appears in a
boundary integral or boundary norm, the same letter denotes its trace;
for instance, \(u\) on \(\bdH\) means \(u|_{\bdH}\).

The sharp Sobolev trace inequality is
\begin{equation}
 \label{eq:sharp-trace}
 \|\nabla u\|_{L^p(\Hhalf)}
 \ge
 \ST(n,p)\,
 \|u\|_{L^{p_*}(\bdH)}
 \qquad
 \forall u\in \dot{W}^{1,p}(\Hhalf),
\end{equation}
where
\begin{equation}
 \label{eq:sharp-trace-constant}
 \ST(n,p)
 :=
 \inf_{\substack{u\in \dot{W}^{1,p}(\Hhalf)\\ u|_{\bdH}\not\equiv0}}
 \frac{\|\nabla u\|_{L^p(\Hhalf)}}
      {\|u\|_{L^{p_*}(\bdH)}}.
\end{equation}
The quotient in \eqref{eq:sharp-trace-constant} is invariant under
multiplication by a nonzero constant, tangential translations, and
the critical dilation
\begin{equation}
 \label{eq:critical-dilation}
 u_{\lambda,\xi}(y,t)
 :=
 \lambda^{-\frac{n-p}{p}}\,
 u\!\left(\frac{y-\xi}{\lambda},\frac{t}{\lambda}\right),
 \qquad
 \lambda>0,\quad \xi\in\R^{n-1}.
\end{equation}

Set
\begin{equation}
 \label{eq:standard-bubble}
 U(y,t)
 :=
 \left(|y|^2+(t+1)^2\right)^{-\frac{n-p}{2(p-1)}}.
\end{equation}
For \(\lambda>0\) and \(\xi\in\R^{n-1}\), define
\begin{align}
 U_{\lambda,\xi}(y,t)
 &:=
 \lambda^{-\frac{n-p}{p}}\,
 U\!\left(\frac{y-\xi}{\lambda},\frac{t}{\lambda}\right)
 \label{eq:bubble-scaling}\\
 &=
 \left(
 \frac{\lambda^{2/p}}
 {|y-\xi|^2+(t+\lambda)^2}
 \right)^{\frac{n-p}{2(p-1)}}.
 \nonumber
\end{align}
The full extremal manifold is
\begin{equation}
 \label{eq:extremal-manifold}
 \MT
 :=
 \left\{
 V_{a,\lambda,\xi}:=
 aU_{\lambda,\xi}
 \,:\,
 a\in\R\setminus\{0\},\
 \lambda>0,\
 \xi\in\R^{n-1}
 \right\}.
\end{equation}
We denote its positive component by
\begin{equation}
 \label{eq:positive-extremal-manifold}
 \MT^+
 :=
 \left\{
 V_{a,\lambda,\xi}\in\MT\,:\,a>0
 \right\}.
\end{equation}
The manifold \(\MT\) is an \((n+1)\)-dimensional smooth manifold. By the classification of equality cases for the sharp
Sobolev trace inequality, \(\MT\) is precisely the family of nonzero
extremals of \eqref{eq:sharp-trace}; see
\cite{Escobar1988,Nazaret2006,MaggiNeumayer2017,Zhou2024}.
In particular,
\begin{equation}
 \label{eq:equality-on-manifold}
 \|\nabla v\|_{L^p(\Hhalf)}
 =
 \ST\,
 \|v\|_{L^{p_*}(\bdH)}
 \qquad
 \forall v\in\MT.
\end{equation}

For the standard positive bubble \(U=U_{1,0}\), introduce the
symmetry modes
\begin{equation}
 \label{eq:symmetry-modes}
 Z_{\mathrm a}:=U,
 \qquad
 Z_0:=
 \left.\frac{\partial}{\partial\lambda}
 U_{\lambda,0}\right|_{\lambda=1},
 \qquad
 Z_j:=
 \left.\frac{\partial}{\partial\xi_j}
 U_{1,\xi}\right|_{\xi=0},
 \quad 1\le j\le n-1.
\end{equation}
Hence
\begin{equation}
 \label{eq:tangent-space-standard}
 T_U\MT
 =
 \operatorname{span}
 \{Z_{\mathrm a},Z_0,Z_1,\ldots,Z_{n-1}\}.
\end{equation}
More generally,
\begin{equation}
 \label{eq:tangent-space-general}
 T_{V_{a,\lambda,\xi}}\MT
 =
 \operatorname{span}
 \left\{
 \partial_aV_{a,\lambda,\xi},
 \partial_\lambda V_{a,\lambda,\xi},
 \partial_{\xi_1}V_{a,\lambda,\xi},
 \ldots,
 \partial_{\xi_{n-1}}V_{a,\lambda,\xi}
 \right\}.
\end{equation}
There is no normal translation mode, since translation in the
\(t\)-direction does not preserve the fixed half-space and its
boundary.

Let \(v\in\MT\).  Differentiating the quotient in
\eqref{eq:sharp-trace-constant} gives the weak Euler--Lagrange equation
\begin{equation}
 \label{eq:EL-weak}
 \int_{\Hhalf}
 |\nabla v|^{p-2}\nabla v\cdot\nabla\varphi\,dx
 =
 \ST^p
 \|v\|_{L^{p_*}(\bdH)}^{\,p-p_*}
 \int_{\bdH}
 |v|^{p_*-2}v\,\varphi\,dy
\end{equation}
for every \(\varphi\in \dot{W}^{1,p}(\Hhalf)\).  Equivalently, \(v\)
satisfies
\begin{equation}
 \label{eq:EL-strong}
 \left\{
 \begin{aligned}
 &\operatorname{div}
 \bigl(|\nabla v|^{p-2}\nabla v\bigr)
 =0
 &&\text{in }\Hhalf,\\
 &|\nabla v|^{p-2}\nabla v\cdot\nuout
 =
 \ST^p
 \|v\|_{L^{p_*}(\bdH)}^{\,p-p_*}
 |v|^{p_*-2}v
 &&\text{on }\bdH,
 \end{aligned}
 \right.
\end{equation}
where \(\nuout=-e_n\) is the outward unit normal to \(\Hhalf\).
For a positive bubble, the boundary condition becomes
\begin{equation}
 \label{eq:EL-positive}
 -|\nabla v|^{p-2}\partial_t v
 =
 \ST^p
 \|v\|_{L^{p_*}(\bdH)}^{\,p-p_*}
 v^{p_*-1}
 \qquad\text{on }\bdH.
\end{equation}

It is convenient to introduce the linearized coefficient matrix
\begin{equation}
 \label{eq:linearized-matrix}
 \cA_v
 :=
 |\nabla v|^{p-2}
 \left(
 I+(p-2)
 \frac{\nabla v}{|\nabla v|}
 \otimes
 \frac{\nabla v}{|\nabla v|}
 \right).
\end{equation}
The associated generalized trace eigenvalue problem is
\begin{equation}
 \label{eq:linearized-eigenproblem}
 \int_{\Hhalf}
 \cA_v\nabla\phi\cdot\nabla\psi\,dx
 =
 \mu
 \int_{\bdH}
 |v|^{p_*-2}\phi\,\psi\,dy
\end{equation}
for all admissible test functions \(\psi\).

\begin{theorem}[Spectral nondegeneracy]
 \label{ass:spectral-nondegeneracy}
 Let \(v=V_{a,\lambda,\xi}\in\MT\) be positive.  Then
 \[
  T_v\MT=E_0\oplus E_1,
 \]
 where
 \[
  E_0:=\operatorname{span}\{v\}
 \]
 is the first eigenspace, with eigenvalue
 \[
  \mu_0=(p-1)\ST^p
  \|v\|_{L^{p_*}(\bdH)}^{p-p_*},
 \]
 and
 \[
  E_1:=\operatorname{span}\left\{
  \partial_\lambda V_{a,\lambda,\xi},
  \partial_{\xi_1}V_{a,\lambda,\xi},\ldots,
  \partial_{\xi_{n-1}}V_{a,\lambda,\xi}
  \right\}
 \]
 is the second eigenspace, with eigenvalue
 \[
  \mu_1=(p_*-1)\ST^p
  \|v\|_{L^{p_*}(\bdH)}^{p-p_*}.
 \]
\end{theorem}
The proof of Theorem~\ref{ass:spectral-nondegeneracy} is given in
Section~\ref{sec:spectral-nondegeneracy-proof}.

\subsection{Compact embedding}
Throughout this paper, \(U=U_{1,0}\) denotes the standard positive
trace bubble defined in \eqref{eq:standard-bubble}.  Since
\[
 U(y,t)
 =
 \bigl(|y|^2+(t+1)^2\bigr)^{-\frac{n-p}{2(p-1)}},
\]
a direct computation gives
\begin{equation}
 \label{eq:gradient-bubble-explicit}
 |\nabla U(y,t)|
 =
 \frac{n-p}{p-1}
 \bigl(|y|^2+(t+1)^2\bigr)^{-\frac{n-1}{2(p-1)}}.
\end{equation}
Consequently,
\begin{equation}
 \label{eq:bulk-weight-explicit}
 |\nabla U(y,t)|^{p-2}
 =
 c_{n,p}
 \bigl(|y|^2+(t+1)^2\bigr)^{
 -\frac{(n-1)(p-2)}{2(p-1)}
 },
\end{equation}
whereas
\begin{equation}
 \label{eq:boundary-weight-explicit}
 U(y,0)^{p_*-2}
 =
 (1+|y|^2)^{
 -\frac{n(p-2)+p}{2(p-1)}
 }.
\end{equation}
In particular, both weights are smooth and strictly positive on every
bounded subset of \(\overline{\Hhalf}\).  The only loss of compactness
that must be excluded occurs at infinity.

We define \(\dot{W}^{1,2}(\Hhalf;|\nabla U|^{p-2})\) as the completion of
\(C_c^\infty(\overline{\Hhalf})\) with respect to the norm
\begin{equation}
 \label{eq:weighted-H-norm}
 \|\phi\|_{\dot{W}^{1,2}(\Hhalf;|\nabla U|^{p-2})}
 :=
 \left(
 \int_{\Hhalf}
 |\nabla U|^{p-2}|\nabla\phi|^2\,dx
 \right)^{1/2}.
\end{equation}

Since the eigenvalues of
\[
 I+(p-2)
 \frac{\nabla U}{|\nabla U|}
 \otimes
 \frac{\nabla U}{|\nabla U|}
\]
are \(1\) and \(p-1\), one has
\begin{align}
 \min\{1,p-1\}
 \int_{\Hhalf}|\nabla U|^{p-2}|\nabla\phi|^2\,dx
 &\le
 \int_{\Hhalf}
 \mathcal A_U\nabla\phi\cdot\nabla\phi\,dx
 \notag\\
 &\le
 \max\{1,p-1\}
 \int_{\Hhalf}|\nabla U|^{p-2}|\nabla\phi|^2\,dx.
 \label{eq:energy-equivalence}
\end{align}
Thus the scalar weighted norm and the full linearized quadratic form
are equivalent.

For \(R>1\), introduce
\begin{align}
 \Omega_R
 &:=
 \{(y,t)\in\Hhalf: |(y,t+1)|>R\},
 \label{eq:exterior-domain}\\
 \Gamma_R
 &:=
 \{y\in\R^{n-1}:(1+|y|^2)^{1/2}>R\}.
 \label{eq:exterior-boundary}
\end{align}

The following estimate is the main ingredient in the compact embedding theorem.

\begin{lemma}[Weighted trace inequality]
 \label{lem:weighted-exterior-trace}
 There exists \(C=C(n,p)>0\) such that, for every \(R\ge2\) and every
 \(\phi\in C_c^\infty(\overline{\Hhalf})\),
 \begin{equation}
  \label{eq:weighted-exterior-trace}
  \int_{\Gamma_{2R}}
  U^{p_*-2}\,|\phi|^2\,dy
  \le
  \frac{C}{R}
  \int_{\Omega_R}
  |\nabla U|^{p-2}\,|\nabla\phi|^2\,dx.
 \end{equation}

 There exists \(C=C(n,p)>0\) such that
 \begin{equation}
  \label{eq:weighted-trace-continuous}
  \int_{\bdH}
  U^{p_*-2}|\phi|^2\,dy
  \le
  C
  \int_{\Hhalf}
  |\nabla U|^{p-2}|\nabla\phi|^2\,dx
 \end{equation}
 for every \(\phi\in C_c^\infty(\overline{\Hhalf})\). 
\end{lemma}

\begin{proof}
\textbf{(1):}
 We give a dyadic proof.  For \(k\ge0\), set
 \begin{align*}
  A_k^+
  &:=
  \left\{
  (y,t)\in\Hhalf:
  2^kR<|(y,t+1)|<2^{k+1}R
  \right\},\\
  \Gamma_k
  &:=
  \left\{
  y\in\R^{n-1}:
  2^kR<(1+|y|^2)^{1/2}<2^{k+1}R
  \right\}.
 \end{align*}
 On \(A_k^+\) and \(\Gamma_k\), respectively,
 \begin{align}
  |\nabla U|^{p-2}(x)
  &\simeq
  (2^kR)^{-\frac{(n-1)(p-2)}{p-1}},
  \label{eq:bulk-weight-annulus-comparison}\\
  U^{p_*-2}(y)
  &\simeq
  (2^kR)^{-\frac{n(p-2)+p}{p-1}},
  \label{eq:boundary-weight-annulus-comparison}
 \end{align}
 with constants depending only on \(n\) and \(p\).

 Rescaling the standard trace--Poincar\'e inequality on a fixed
 half-annulus yields
 \begin{equation}
  \label{eq:annular-trace-poincare}
  \int_{\Gamma_k}
  |\phi-(\phi)_{A_k^+}|^2\,dy
  \le
  C(2^kR)
  \int_{ A_k^+}
  |\nabla\phi|^2\,dx,
 \end{equation}
 where \((\phi)_{A_k^+}\) denotes the average of \(\phi\) on \(A_k^+\).

 Combining \eqref{eq:bulk-weight-annulus-comparison},
 \eqref{eq:boundary-weight-annulus-comparison}, and
 \eqref{eq:annular-trace-poincare}, we obtain
 \begin{equation}
  \label{eq:annular-weighted-trace}
  \int_{\Gamma_k}
  U^{p_*-2}
  |\phi-(\phi)_{A_k^+}|^2\,dy
  \le
  \frac{C}{2^kR}
  \int_{ A_k^+}
  |\nabla U|^{p-2}|\nabla\phi|^2\,dx.
 \end{equation}

 It remains to estimate the averages.  Since
 \(\phi\in C_c^\infty(\overline{\Hhalf})\), the averages vanish for
 all sufficiently large \(k\).  A telescoping argument gives
 \begin{equation}
  \label{eq:average-telescoping}
  |(\phi)_{A_k^+}|
  \le
  \sum_{j=k}^{\infty}
  |(\phi)_{A_j^+}-(\phi)_{A_{j+1}^+}|.
 \end{equation}
 We now prove the estimate for the difference of the two averages.
 Put \(\widetilde A_j^+:=A_j^+\cup A_{j+1}^+\) and \(r_j:=2^jR\).  Since
 \(|A_j^+|\simeq |A_{j+1}^+|\simeq |\widetilde A_j^+|\simeq r_j^n\), the triangle
 inequality and Cauchy--Schwarz give
 \[
  |(\phi)_{A_j^+}-(\phi)_{A_{j+1}^+}|
  \le
  C r_j^{-n/2}
  \|\phi-(\phi)_{\widetilde A_j^+}\|_{L^2(\widetilde A_j^+)}.
 \]
 Applying the Poincar\'e inequality on \(\widetilde A_j^+\)
 and using the scaling of the annulus,
 \[
  \|\phi-(\phi)_{\widetilde A_j^+}\|_{L^2(\widetilde A_j^+)}
  \le
  C r_j\|\nabla\phi\|_{L^2(\widetilde A_j^+)}.
 \]
 Therefore, by \eqref{eq:bulk-weight-annulus-comparison}
 \begin{equation}
  \label{eq:average-difference}
  |(\phi)_{A_j^+}-(\phi)_{A_{j+1}^+}|
  \le
  C(2^jR)^{-\frac{n-p}{2(p-1)}}
  \|\nabla\phi\|_{L^2(\widetilde A_j^+,|\nabla U|^{p-2})}.
 \end{equation}
 Using \eqref{eq:boundary-weight-annulus-comparison}, \eqref{eq:average-telescoping} and \eqref{eq:average-difference}, by Minkowski's inequality
 \begin{align}
  \label{eq:average-bound}
  \|(\phi)_{A_k^+}\|_{L^2(\Gamma_k,U^{p_*-2})}
  \le& \sum_{j=k}^{\infty}
  \|(\phi)_{A_j^+}-(\phi)_{A_{j+1}^+}\|_{L^2(\Gamma_k,U^{p_*-2})}\nonumber\\
  \le&
  \frac{C}{\sqrt{2^kR}}
  \sum_{j=k}^{\infty}
  2^{-\frac{n-p}{2(p-1)}(j-k)}
  \|\nabla\phi\|_{L^2(\widetilde A_j^+,|\nabla U|^{p-2})}\nonumber\\
  \le& \frac{C}{\sqrt{2^kR}}\|\nabla\phi\|_{L^2(\Omega_R,|\nabla U|^{p-2})}.
 \end{align}
 Combining \eqref{eq:annular-weighted-trace} and
 \eqref{eq:average-bound}, summing over \(k\ge1\), and using the
 bounded overlap of the enlarged annuli \(\widetilde A_k^+\), we find
 \[
  \sum_{k=1}^{\infty}
  \int_{\Gamma_k}
  U^{p_*-2}|\phi|^2\,dy
  \le
  \frac{C}{R}
  \int_{\Omega_R}
  |\nabla U|^{p-2}|\nabla\phi|^2\,dx.
 \]
 Since
 \[
  \Gamma_{2R}
  =
  \bigcup_{k=1}^{\infty}\Gamma_k,
 \]
 this proves \eqref{eq:weighted-exterior-trace}.

\noindent\textbf{(2).} Fix $R=2$ and argue as in part (1).  We obtain
\begin{equation}
  \label{eq:local-weighted-trace-preliminary1}
  \int_{\bdH\cap B_{2R}(-e_n)}
  U^{p_*-2}|\phi-(\phi)|^2\,dy
  \le
  C\int_{B_{2R}(-e_n)\cap\Hhalf}
  |\nabla U|^{p-2}|\nabla\phi|^2\,dx,
 \end{equation}
 where \((\phi)\) denotes the average of \(\phi\) on
 \(B_{2R}(-e_n)\cap\Hhalf\).  Also, we have
 \begin{align*}
  |(\phi)|
  &\le |(\phi)-(\phi)_{A_{1}^+}|
+\sum_{j=1}^{\infty}
  |(\phi)_{A_j^+}-(\phi)_{A_{j+1}^+}|,\\
  |(\phi)_{A_j^+}&-(\phi)_{A_{j+1}^+}|
  \le
  C(2^j)^{-\frac{n-p}{2(p-1)}}
  \|\nabla\phi\|_{L^2(\widetilde A_j^+,|\nabla U|^{p-2})},
 \end{align*}
and \[|(\phi)-(\phi)_{A_{1}^+}|\le C
  \|\nabla\phi\|_{L^2(B_{4R}(-e_n)\cap\Hhalf,|\nabla U|^{p-2})}.\]
Hence, by Minkowski's inequality
 \begin{align}
  \label{eq:average-bound1}
  &\|(\phi)\|_{L^2(\bdH\cap B_{2R}(-e_n),U^{p_*-2})}\nonumber\\
  \le&\|(\phi)-(\phi)_{A_{1}^+}\|_{L^2(\bdH\cap B_{4R}(-e_n),U^{p_*-2})}+ \sum_{j=1}^{\infty}
  \|(\phi)_{A_j^+}-(\phi)_{A_{j+1}^+}\|_{
  L^2(\bdH\cap B_{4R}(-e_n),U^{p_*-2})}\nonumber\\
  \le&
  C\left(
  \|\nabla\phi\|_{L^2(B_{4R}(-e_n)\cap\Hhalf,|\nabla U|^{p-2})}+
  \sum_{j=1}^{\infty}
  2^{-\frac{n-p}{2(p-1)}(j-1)}
  \|\nabla\phi\|_{L^2(\widetilde A_j^+,|\nabla U|^{p-2})}\right)\nonumber\\
  \le& C\|\nabla\phi\|_{L^2(\Hhalf,|\nabla U|^{p-2})}.
 \end{align}
Combining \eqref{eq:local-weighted-trace-preliminary1},
 \eqref{eq:average-bound1} and \eqref{eq:weighted-exterior-trace}, we obtain \eqref{eq:weighted-trace-continuous}.
 
\end{proof}

\begin{theorem}[Compact weighted trace embedding]
 \label{thm:weighted-trace-compact}
 The embedding
 \begin{equation}
  \label{eq:compact-weighted-embedding}
  \dot{W}^{1,2}(\Hhalf;|\nabla U|^{p-2})
  \hookrightarrow
  L^2(\bdH,U^{p_*-2}\,dy)
 \end{equation}
 is compact.  Equivalently, if
 \[
  \sup_k
  \int_{\Hhalf}
  |\nabla U|^{p-2}|\nabla\phi_k|^2\,dx<\infty,
 \]
 then, after passing to a subsequence, there exists
 \(\phi\in\dot{W}^{1,2}(\Hhalf;|\nabla U|^{p-2})\) such that
 \begin{equation}
  \label{eq:strong-weighted-boundary-convergence}
  \phi_k
  \longrightarrow
  \phi
  \qquad\text{strongly in }
  L^2(\bdH,U^{p_*-2}\,dy).
 \end{equation}
\end{theorem}

\begin{proof}
 Let \((\phi_k)\) be bounded in
 \(\dot{W}^{1,2}(\Hhalf;|\nabla U|^{p-2})\).  By weak compactness in
 this Hilbert space, after passing to a subsequence,
 \[
  \phi_k\rightharpoonup\phi
  \qquad\text{weakly in }
  \dot{W}^{1,2}(\Hhalf;|\nabla U|^{p-2}).
 \]
 Replacing \(\phi_k\) by \(\phi_k-\phi\), it is enough to assume
 \[
  \phi_k\rightharpoonup0
  \qquad\text{in }
  \dot{W}^{1,2}(\Hhalf;|\nabla U|^{p-2})
 \]
 and prove strong convergence of the traces.

 Fix \(R>2\).  Since \(|\nabla U|^{p-2}\) is bounded above and below by positive
 constants on
 \[
  B_{2R}(-e_n)\cap\Hhalf,
 \]
 the gradients \((\nabla\phi_k)\) are bounded in the usual \(L^2\)-space
 there.  Moreover, \eqref{eq:weighted-trace-continuous} bounds
 \((\phi_k)\) in \(L^2(\bdH\cap B_{2R}(-e_n))\), because the boundary
 weight is bounded below by a positive constant on this set.  The
 Poincar\'e inequality with a boundary term gives
 \[
  \|\phi_k\|_{L^2(B_{2R}(-e_n)\cap\Hhalf)}
  \le C_R\left(
  \|\nabla\phi_k\|_{L^2(B_{2R}(-e_n)\cap\Hhalf)}
  +\|\phi_k\|_{L^2(\bdH\cap B_{2R}(-e_n))}
  \right).
 \]
 Thus \((\phi_k)\) is bounded in the usual \(W^{1,2}\)-space on every
 bounded half-ball.  The compact trace theorem now implies
 \begin{equation}
  \label{eq:local-trace-strong}
  \phi_k\longrightarrow0
  \qquad\text{strongly in }
  L^2\bigl(
  \bdH\cap B_{2R}(-e_n)
  \bigr).
 \end{equation}
 Since \(U^{p_*-2}\) is bounded on this compact boundary region,
 \begin{equation}
  \label{eq:local-weighted-trace-strong}
  \int_{\bdH\cap B_{2R}(-e_n)}
  U^{p_*-2}|\phi_k|^2\,dy
  \longrightarrow0.
 \end{equation}

 On the other hand, Lemma~\ref{lem:weighted-exterior-trace} gives
 \begin{equation}
  \label{eq:uniform-tail-smallness}
  \sup_k
  \int_{\Gamma_{2R}}
  U^{p_*-2}|\phi_k|^2\,dy
  \le
  \frac{C}{R}
  \sup_k\|\phi_k\|_{
  \dot{W}^{1,2}(\Hhalf;|\nabla U|^{p-2})}^2.
 \end{equation}
 Therefore,
 \begin{align}
  \limsup_{k\to\infty}
  \int_{\bdH}
  U^{p_*-2}|\phi_k|^2\,dy
  &\le
  \frac{C}{R}
  \sup_k\|\phi_k\|_{
  \dot{W}^{1,2}(\Hhalf;|\nabla U|^{p-2})}^2.
  \label{eq:compactness-limsup}
 \end{align}
 Letting \(R\to\infty\) proves
 \[
  \int_{\bdH}
  U^{p_*-2}|\phi_k|^2\,dy
  \longrightarrow0.
 \]
 This is precisely the desired compactness.
\end{proof}

\subsection{Spectral gap estimate}

Consider the Rayleigh quotient
\begin{equation}
 \label{eq:trace-rayleigh-quotient}
 \frac{
 \displaystyle
 \int_{\Hhalf}
 \mathcal A_U\nabla\phi\cdot\nabla\phi\,dx
 }{
 \displaystyle
 \int_{\bdH}
 U^{p_*-2}|\phi|^2\,dy
 },
 \qquad
 \phi\in
 \left\{
 \psi\in\dot{W}^{1,2}(\Hhalf;|\nabla U|^{p-2})
 :\psi|_{\bdH}\not\equiv0
 \right\}.
\end{equation}
The Rayleigh quotient is associated with the following weak eigenvalue problem. Let $\phi\in \dot{W}^{1,2}(\Hhalf;|\nabla U|^{p-2})$ be a weak solution of
\begin{equation}
\begin{cases}
L\phi:=-\operatorname{div}\left( \mathcal A_U \nabla\phi\right)=0
&\text{in }\Hhalf,\\[1mm]
\bigl(\mathcal A_U \nabla\phi\bigr)\cdot\nu
=\mu U^{p_*-2}\phi
&\text{on }\partial\Hhalf.
\end{cases}
\label{eq:eigen}
\end{equation}

\begin{corollary}[Discreteness of the weighted trace spectrum]
 \label{cor:discrete-spectrum}
 The generalized eigenvalue problem
 \begin{equation}
  \label{eq:generalized-trace-spectrum}
  \int_{\Hhalf}
  \mathcal A_U\nabla\phi\cdot\nabla\psi\,dx
  =
  \mu
  \int_{\bdH}
  U^{p_*-2}\phi\,\psi\,dy\quad\text{for every }\psi\in \dot{W}^{1,2}(\Hhalf;|\nabla U|^{p-2})
 \end{equation}
 has a discrete sequence of positive eigenvalues
 \begin{equation}
  \label{eq:eigenvalue-sequence}
  0<\mu_0\le\mu_1\le\cdots\le\mu_k\le\cdots,
  \qquad
  \mu_k\longrightarrow+\infty,
 \end{equation}
 counted with multiplicity. The traces of the corresponding
 eigenfunctions form an orthonormal basis of \(L^2(\bdH,U^{p_*-2})\).
\end{corollary}

\begin{proof}
 The energy equivalence \eqref{eq:energy-equivalence} shows that the
 bilinear form corresponding to $L$
 \[
  (\phi,\psi)_{\dot{W}^{1,2}(\Hhalf;|\nabla U|^{p-2})}
  :=
 \int_{\Hhalf}
  \mathcal A_U\nabla\phi\cdot\nabla\psi\,dx
 \]
 is coercive and defines an equivalent Hilbert norm on
 \(\dot{W}^{1,2}(\Hhalf;|\nabla U|^{p-2})\).  Lemma~\ref{lem:weighted-exterior-trace}
implies that the weighted boundary pairing
 \[
  (\phi,\psi)_{L^2(\bdH,U^{p_*-2})}:=\int_{\bdH}
  U^{p_*-2}\phi\,\psi\,dy
 \]
 is continuous on
 \(\dot{W}^{1,2}(\Hhalf;|\nabla U|^{p-2})\).  
 Therefore, by the Riesz representation theorem, for every
 $f\in L^2(\bdH,U^{p_*-2})$, there exists a unique solution
 $\phi\in \dot{W}^{1,2}(\Hhalf;|\nabla U|^{p-2})$ of 
 \[(\phi,\psi)_{\dot{W}^{1,2}(\Hhalf;|\nabla U|^{p-2})}=(f,\psi)_{L^2(\bdH,U^{p_*-2})}.\]
 
 Let us write $\phi=L^{-1}f$.
 By Theorem~\ref{thm:weighted-trace-compact}, the
 operator $L^{-1}:L^2(\bdH,U^{p_*-2})\rightarrow \dot{W}^{1,2}(\Hhalf;|\nabla U|^{p-2})\Subset L^2(\bdH,U^{p_*-2})$ is compact, positive, and self-adjoint.  Its nonzero eigenvalues
 are the reciprocals of the generalized eigenvalues in
 \eqref{eq:generalized-trace-spectrum}.  The spectral theorem for
 compact self-adjoint operators then gives the asserted discreteness
 and the basis statement; see, for
 example, \cite{Borthwick2020,Kato1995}.
\end{proof}

\begin{corollary}[Spectral gap estimate]
 \label{cor:linear-spectral-gap}
 By Theorem~\ref{ass:spectral-nondegeneracy}, there exists
 \(\lambda_{\mathrm T}=\lambda_{\mathrm T}(n,p)>0\) such that every
 \(\phi\in\dot{W}^{1,2}(\Hhalf;|\nabla U|^{p-2})\) satisfying
 \begin{equation}
  \label{eq:weighted-orthogonality-standard}
  \int_{\bdH}
  U^{p_*-2}\phi\,Z\,dy=0
  \qquad
  \forall Z\in T_U\MT
 \end{equation}
 obeys
 \begin{align}
  &\int_{\Hhalf}
  \mathcal A_U\nabla\phi\cdot\nabla\phi\,dx
  \notag\\
  &\qquad\ge
  \left[
  (p_*-1)\ST^p
  +2\lambda_{\mathrm T}
  \right]\|U\|_{L^{p_*}(\bdH)}^{\,p-p_*}
  \int_{\bdH}
  U^{p_*-2}|\phi|^2\,dy.
  \label{eq:linear-spectral-gap-standard}
 \end{align}
 By scaling, tangential translation, and multiplication, the
 corresponding estimate holds at every positive
 \(v\in\MT\).
\end{corollary}

\begin{proof}
 By Corollary~\ref{cor:discrete-spectrum}, the spectrum is discrete. Theorem~\ref{ass:spectral-nondegeneracy} identifies the second eigenvalue as
 \[
  (p_*-1)\ST^p
  \|U\|_{L^{p_*}(\bdH)}^{\,p-p_*}
  \]
  and identifies the first two eigenspaces. The next eigenvalue is therefore strictly
  larger. Denoting half of the spectral separation, after factoring out
  \(\|U\|_{L^{p_*}(\bdH)}^{p-p_*}\), by
  \(\lambda_{\mathrm T}>0\) gives
  \eqref{eq:linear-spectral-gap-standard}.  The invariances of the
 quotient and of the bubble family yield the uniform form for general
 \(v\in\MT\).
\end{proof}

\section{Proof of spectral nondegeneracy}
\label{sec:spectral-nondegeneracy-proof}

This section proves Theorem~\ref{ass:spectral-nondegeneracy}.  We retain
the local notation introduced below to streamline the explicit
calculation.  The identifications are
\[
 \Om=\Hhalf,
 \qquad
 S=\ST,
 \qquad
 \mathcal M=\MT^+,
\]
where \(\MT^+\) denotes the positive component of \(\MT\).  Thus the
constant \(S\) has exactly the same gradient-side normalization as
\(\ST\).  At the end of the section, we translate the classification
back into the notation of
Section~\ref{sec:weighted-trace-compactness}.

\subsection{Reduction to the weighted Steklov problem}

For the explicit spectral calculation, we use the following local
notation.  Let $n\ge 3$ and $p\in(1,n)$, and denote by $B(x,r)$ the ball
centered at $x$ with radius $r$.  Set
\[
  \Om=\R^n_+ := \{(x,t):x\in\R^{n-1},\ t>0\},
  \qquad
  p_*:=\frac{p(n-1)}{n-p}.
\]
Let $\{e_i\}_{i=1}^n$ be the standard basis of $\R^n$, and denote the Euclidean metric on $\Om$ by
\[
  g_e=dx^2+dt^2.
\]
In this notation, the sharp Sobolev trace inequality reads
\begin{equation}
  S\|u\|_{L^{p_*}(\partial\Om)}\le \|Du\|_{L^p(\Om)}
  \label{eq:sobolev-trace}
\end{equation}
for all $u\in \dot W^{1,p}(\Om)$, with the homogeneous space understood
as in Definition~\ref{def:weighted-spaces}, under the norm
\[
  \|u\|_{\dot W^{1,p}(\Om)}:=\|Du\|_{L^p(\Om)}.
\]
For $a>0$, $\lambda>0$, and $x_0\in\R^{n-1}$, set
\[
  v_{a,\lambda,x_0}(x,t)
  :=a\lambda^{\frac{n-p}{p(p-1)}}
  \bigl(|x-x_0|^2+(t+\lambda)^2\bigr)^{-\frac{n-p}{2(p-1)}}
\]
and
\[
  \mathcal M:=\{v_{a,\lambda,x_0}:a>0,\ \lambda>0,\ x_0\in\R^{n-1}\}\subset \dot W^{1,p}(\Om).
\]
It is known that $\mathcal M$ is the set of positive functions for which equality holds in the Sobolev trace inequality above. In addition,
\begin{equation}
\begin{cases}
\Delta_p v_{a,\lambda,x_0}=0 &\text{in }\Om,\\[2mm]
|Dv_{a,\lambda,x_0}|^{p-2}\dfrac{\partial v_{a,\lambda,x_0}}{\partial t}
=-S^p\|v_{a,\lambda,x_0}\|_{L^{p_*}(\partial\Om)}^{p-p_*}|v_{a,\lambda,x_0}|^{p_*-2}v_{a,\lambda,x_0}
&\text{on }\partial\Om.
\end{cases}
\label{eq2}
\end{equation}

Define
\[
  v:=v_{a,1,0},
  \qquad
  U:=v_{1,1,0}.
\]
It follows from \eqref{eq2} that $S^p\|U\|_{L^{p_*}(\partial\Om)}^{p-p_*}=(\frac{n-p}{p-1})^{p-1}$.
A direct differentiation of the bubble parameters gives
\[
  T_v\mathcal M
  =\spanop\{v,\ \partial_\lambda|_{\lambda=1}v_{a,\lambda,0},\ \partial_1v,\ldots,\partial_{n-1}v\}\subset \dot W^{1,2}(\Om;|D v|^{p-2})\cap \dot W^{1,p}(\Om).
\]
Recall the desired spectral gap estimate:
there exists $\lambda=\lambda(n,p)>0$ such that, for every
\[
  \phi\in \dot W^{1,2}(\Om;|D v|^{p-2})\cap (T_v\mathcal M)^\perp,
\]
one has
\begin{align}
 &\int_\Om |Dv|^{p-2}|D\phi|^2
 +(p-2)|Dv|^{p-4}(Dv\cdot D\phi)^2\,dxdt
 \nonumber\\
 &\qquad\ge
 \bigl((p_*-1)S^p+2\lambda\bigr)
 \|v\|_{L^{p_*}(\partial\Om)}^{p-p_*}
 \int_{\partial\Om}v^{p_*-2}\phi^2\,dx.
 \label{eq:stability-orthogonal}
\end{align}

\begin{remark}\label{rem1}
Theorem~\ref{thm:weighted-trace-compact} gives the compact embedding
\[
        \dot W^{1,2}(\Om;|D v|^{p-2})
        \hookrightarrow L^2(\p\Om;v^{p_*-2}).
\]
We say that
$\phi\in \dot W^{1,p}(\Om)$ or
$\phi\in \dot W^{1,2}(\Om;|D v|^{p-2})$ is perpendicular to
$T_v\mathcal M$ if
    \[\int_{\p\Om}\phi\,\xi\, v^{p_*-2} d x=0\qquad\T{for any }\xi\in T_v\mathcal M.\]
\end{remark}

It remains to solve the eigenvalue problem
\begin{equation}
\begin{cases}
\divg\bigl(|DU|^{p-2}D\phi+(p-2)|DU|^{p-4}(DU\cdot D\phi)DU\bigr)=0
&\text{in }\Om,\\[1mm]
\bigl(|DU|^{p-2}D\phi+(p-2)|DU|^{p-4}(DU\cdot D\phi)DU\bigr)\cdot\nu
=\mu U^{p_*-2}\phi
&\text{on }\partial\Om.
\end{cases}
\label{eq:eigen-halfspace}
\end{equation}
Indeed, it suffices to consider the first two eigenspaces.

\begin{theorem}[Identification of the first two eigenvalues]\label{thm:iden}
The first eigenpair is
    \[
  \mu_0=(p-1)\left(\frac{n-p}{p-1}\right)^{p-1},
  \qquad
  \phi_0=U.
\]
The second eigenvalue is
\[
  \mu_1=(p_*-1)\left(\frac{n-p}{p-1}\right)^{p-1},
\]
and its eigenspace contains
\[
  \spanop\{\partial_\lambda|_{\lambda=1}v_{1,\lambda,0},\ \partial_1U,\ldots,\partial_{n-1}U\}.
\]
\end{theorem}
\begin{proof}
The positive function $\phi_0:=U$ satisfies \eqref{eq:eigen-halfspace} with
\[
  \mu_0=(p-1)\left(\frac{n-p}{p-1}\right)^{p-1}.
\]
We identify its variational position directly.  Denote by
$Q_U[\cdot,\cdot]$ and $B_U[\cdot,\cdot]$ the bulk and boundary
bilinear forms in \eqref{eq:eigen-halfspace}, respectively.  Since $U$
is a minimizer of the Sobolev trace inequality
\eqref{eq:sobolev-trace}, its second variation gives
\begin{align}
    0\le &\frac{d^2}{d \e^2}\bigg|_{\e=0}\left(\int_{\R^n_+}|D (U+\e \vp)|^p \,dx dt-S^p \bigg(\int_{\R^{n-1}}|U+\e \vp|^{p_*} \,dx\bigg)^\frac{p}{p_*}\right)\non\\
    =&p \int_\Om \left[
 |DU|^{p-2}|D\vp|^2
 +(p-2)|DU|^{p-4}(DU\cdot D\vp)^2\right]\,dx\,dt\non\\
-&S^p p (p_*-1)\|U\|_{L^{p_*}(\p\Om)}^{p-p_*}\int_{\p\Om}U^{p_*-2}\vp^2 d x-S^p p (p-p_*)\|U\|_{L^{p_*}(\p\Om)}^{p-2 p_*}\bigg(\int_{\p\Om}U^{p_*-1}\vp d x\bigg)^2.\label{eq5}
\end{align}
If $\vp\perp \phi_0$ in
$L^2(\partial\Om;U^{p_*-2})$, the final term vanishes and hence
\begin{equation}
\int_\Om |DU|^{p-2}|D\vp|^2
+(p-2)|DU|^{p-4}(DU\cdot D\vp)^2\,dxdt\ge S^p (p_*-1)\|U\|_{L^{p_*}(\p\Om)}^{p-p_*}\int_{\p\Om}U^{p_*-2}\vp^2 d x,\label{eq6}
\end{equation}
Using
\[
  S^p\|U\|_{L^{p_*}(\partial\Om)}^{p-p_*}
  =\left(\frac{n-p}{p-1}\right)^{p-1},
\]
define
\[
  \mu_1:=(p_*-1)\left(\frac{n-p}{p-1}\right)^{p-1}.
\]
Since $p_*>p$, one has $\mu_1>\mu_0$.  For arbitrary $\phi$, write
$\phi=cU+\psi$ with $B_U[\psi,U]=0$.  The eigenvalue equation for $U$
and \eqref{eq6} give
\[
  Q_U[\phi,\phi]
  =c^2\mu_0B_U[U,U]+Q_U[\psi,\psi]
  \ge \mu_0B_U[\phi,\phi],
\]
with equality only when $\psi=0$.  Thus $(\mu_0,U)$ is the first
eigenpair and its eigenspace is one-dimensional.

The functions in
\[
  \spanop\{\partial_\lambda|_{\lambda=1}v_{1,\lambda,0},\ 
  \partial_1U,\ldots,\partial_{n-1}U\}
\]
attain equality in \eqref{eq6}.
The min--max principle therefore identifies $\mu_1$ as the second
eigenvalue, and the displayed span is contained in its eigenspace; see
\cite[Theorem~5.15]{Borthwick2020}.
    
\end{proof}

Consider the conformal map
\[
  f:B\left(\frac12e_n,\frac12\right)\to\R^n_+,
  \qquad
  y=(y',y_n)\mapsto |y|^{-2}(y',y_n)-e_n,
\]
and its inverse
\[
  F=f^{-1}:\R^n_+\to B\left(\frac12e_n,\frac12\right),
  \qquad
  (x,t)\mapsto \bigl(|x|^2+(t+1)^2\bigr)^{-1}(x,t+1).
\]
Let $dy^2$ denote the Euclidean metric on $B(\frac12e_n,\frac12)$, and define
\[
  W(y):=U(f(y))=|y|^{\frac{n-p}{p-1}}.
\]
A direct calculation gives
\[
  F^*(dy^2)=\bigl(|x|^2+(t+1)^2\bigr)^{-2}(dx^2+dt^2)
  =U^{\frac{4(p-1)}{n-p}}g_e.
\]
Introduce the conformal exponent
\[
  \beta_{\mathrm c}:=-\frac{2(p-2)}{n-p},
\]
and the conformal metric
\[
  g_0:=W^{2\beta_{\mathrm c}}dy^2
  \qquad\text{on }B\left(\frac12e_n,\frac12\right).
\]
We have
\[
  g:=F^*g_0=U^{2\beta_{\mathrm c}+\frac{4(p-1)}{n-p}}g_e.
\]
Define
\[
  X(\phi)
  :=|\nabla_g U|^{p-2}\nabla_g\phi
  +(p-2)|\nabla_gU|^{p-4}(\nabla_gU\cdot\nabla_g\phi)\nabla_gU,
\]
where $\nabla_g$ is the Riemannian gradient on $(\Om,g)$ and $\cdot$ denotes the inner product induced by $g$. Then
\[
  X(\phi\varphi)=X(\phi)\varphi+\phi X(\varphi).
\]
With this choice of \(\beta_{\mathrm c}\), one obtains
\begin{align}
&\int_\Om |DU|^{p-2}|D\phi|^2
+(p-2)|DU|^{p-4}(DU\cdot D\phi)^2\,dxdt
\nonumber\\
&\quad=
\int_\Om X(\phi)\cdot\nabla_g\phi\,
U^{-\beta_{\mathrm c}(n-p)-2(p-1)}\,dV_g
\nonumber\\
&\quad=
\int_\Om X(U^{-1}\phi)\cdot\nabla_g(U^{-1}\phi)
+\phi^2\divg(X(U^{-1}))U^{-1}\,dV_g
\nonumber\\
&\qquad-
\int_{\partial\Om}\phi^2U^{-1}g(X(U^{-1}),\nu)\,d\sigma_g.
\label{eq:energy-identity}
\end{align}
A direct calculation gives
\[
  \divg(X(U^{-1}))=0,
  \qquad
  Ug(X(U^{-1}),\nu)
  =-(p-1)\left(\frac{n-p}{p-1}\right)^{p-1}
  U^{\frac{(n-1)(p-2)}{n-p}}.
\]
Set
\[
  \varphi:=U^{-1}\phi.
\]
Then
\begin{align}
&\int_\Om |DU|^{p-2}|D\phi|^2
+(p-2)|DU|^{p-4}(DU\cdot D\phi)^2\,dxdt
\nonumber\\
&\quad=
\int_\Om |\nabla_gU|^{p-2}|\nabla_g\varphi|^2
+(p-2)|\nabla_gU|^{p-4}(\nabla_gU\cdot\nabla_g\varphi)^2\,dV_g
\nonumber\\
&\qquad+(p-1)\left(\frac{n-p}{p-1}\right)^{p-1}
\int_{\partial\Om}U^{\frac{(n-1)(p-2)}{n-p}}\varphi^2\,d\sigma_g.
\label{eq:energy-u-transform}
\end{align}
On the other hand,
\begin{equation}
  \int_{\partial\Om}U^{p_*-2}\phi^2\,dx
  =\int_{\partial\Om}U^{\frac{(n-1)(p-2)}{n-p}}\varphi^2\,d\sigma_g.
\label{eq:boundary-norm-transform}
\end{equation}
Denote
\[
  \OmB:=B\left(\frac12e_n,\frac12\right),
  \qquad
  \widetilde\varphi:=\varphi\circ f.
\]
Since $g=F^*g_0$, we have
\begin{align}
&\int_\Om |DU|^{p-2}|D\phi|^2
+(p-2)|DU|^{p-4}(DU\cdot D\phi)^2\,dxdt
\nonumber\\
&\quad=
\int_{\OmB}|\nabla_{g_0}W|^{p-2}|\nabla_{g_0}\widetilde\varphi|^2
+(p-2)|\nabla_{g_0}W|^{p-4}
(\nabla_{g_0}W\cdot\nabla_{g_0}\widetilde\varphi)^2\,dV_{g_0}
\nonumber\\
&\qquad+(p-1)\left(\frac{n-p}{p-1}\right)^{p-1}
\int_{\partial\OmB}W^{\frac{(n-1)(p-2)}{n-p}}\widetilde\varphi^2\,d\sigma_{g_0},
\label{eq:ball-energy}
\end{align}
and
\begin{equation}
  \int_{\partial\Om}U^{p_*-2}\phi^2\,dx
  =\int_{\partial\OmB}W^{\frac{(n-1)(p-2)}{n-p}}\widetilde\varphi^2\,d\sigma_{g_0}.
\label{eq:ball-boundary-norm}
\end{equation}
For notational simplicity, we relabel $\widetilde\varphi$ as $\varphi$
from this point onward.
Therefore, the eigenvalue problem \eqref{eq:eigen-halfspace} is equivalent to the following eigenvalue problem on the ball:
\begin{equation}
\begin{cases}
\divg\bigl(|\nabla_{g_0}W|^{p-2}\nabla_{g_0}\varphi
 +(p-2)|\nabla_{g_0}W|^{p-4}(\nabla_{g_0}W\cdot\nabla_{g_0}\varphi)\nabla_{g_0}W\bigr)=0
&\text{in }\OmB,\\[1mm]
g_0\bigl(|\nabla_{g_0}W|^{p-2}\nabla_{g_0}\varphi
 +(p-2)|\nabla_{g_0}W|^{p-4}(\nabla_{g_0}W\cdot\nabla_{g_0}\varphi)\nabla_{g_0}W,\nu\bigr)
&\\
\qquad+(p-1)\left(\frac{n-p}{p-1}\right)^{p-1}W^{\frac{(n-1)(p-2)}{n-p}}\varphi
=\widetilde\mu W^{\frac{(n-1)(p-2)}{n-p}}\varphi
&\text{on }\partial\OmB.
\end{cases}
\label{eq:ball-eigenproblem-metric}
\end{equation}
Here
\[
  g_0=W^{2\beta_{\mathrm c}}dy^2,
  \qquad
  W(y)=|y|^{\frac{n-p}{p-1}},
  \qquad
  \beta_{\mathrm c}=-\frac{2(p-2)}{n-p}.
\]
Equivalently, in Euclidean notation, \eqref{eq:ball-eigenproblem-metric} becomes
\begin{equation}
\begin{cases}
\divg\Bigl[W^{-2(p-2)}\bigl(|DW|^{p-2}D\varphi
 +(p-2)|DW|^{p-4}(DW\cdot D\varphi)DW\bigr)\Bigr]=0
&\text{in }\OmB,\\[1mm]
W^{-2(p-2)}\bigl(|DW|^{p-2}D\varphi
 +(p-2)|DW|^{p-4}(DW\cdot D\varphi)DW\bigr)\cdot\nu
&\\[1mm]
\qquad+(p-1)\left(\frac{n-p}{p-1}\right)^{p-1}W^{-\frac{(n-1)(p-2)}{n-p}}\varphi
=\widetilde\mu W^{-\frac{(n-1)(p-2)}{n-p}}\varphi
&\text{on }\partial\OmB.
\end{cases}
\label{eq:transformed-eigenproblem}
\end{equation}

\subsection{Analysis of the transformed eigenvalue problem}

\subsubsection{Notation}

Let $\mathcal P_k^n$ denote the space of real homogeneous polynomials of degree $k$ on $\R^n$, and define the space of spherical harmonics of degree $k$ by
\[
  \mathcal H_k^n
  :=\{P|_{S^{n-1}}:\Delta P=0,\ P\in\mathcal P_k^n\}.
\]
It is known that \cite{DX13}
\[
  \dim\mathcal P_k^n=\binom{k+n-1}{k},
  \qquad
  \dim\mathcal H_k^n=\dim\mathcal P_k^n-\dim\mathcal P_{k-2}^n.
\]
Let
\[
  1<p<n,
  \qquad
  \OmB=B\left(\frac12e_n,\frac12\right)\subset\R^n,
  \qquad
  W(y)=|y|^\alpha,
  \qquad
  \alpha:=\frac{n-p}{p-1}.
\]
We write
\[
  r:=|y|,
  \qquad
  \sigma:=\frac{y}{|y|}\in S^{n-1},
  \qquad
  \tau:=\sigma_n.
\]
The boundary of $\OmB$ satisfies
\[
  \partial\OmB=\left\{y:\bigg|y-\frac12e_n\bigg|=\frac12\right\}
  \Longleftrightarrow |y|^2=y_n.
\]
Therefore, on $\partial\OmB$,
\[
  r=\tau.
\]
The Euclidean outward unit normal to $\partial\OmB$ is
\[
  \nu=2y-e_n.
\]
Set
\[
  \mathcal A_\varphi
  :=
  W^{-2(p-2)}\bigl(|DW|^{p-2}D\varphi
  +(p-2)|DW|^{p-4}(DW\cdot D\varphi)DW\bigr).
\]
Then \eqref{eq:transformed-eigenproblem} can be rewritten as follows:
\(\varphi\in W^{1,2}(\OmB;r^{-\gamma})\) is a weak solution of
\begin{equation}
\begin{cases}
L\vp:=\divg \mathcal A_\varphi=0
&\text{in }\OmB,\\[1mm]
\mathcal A_\varphi\cdot\nu
+C_0W^{-m}\varphi
=\widetilde\mu W^{-m}\varphi
&\text{on }\partial\OmB,
\end{cases}
\label{eq:model-eigenproblem}
\end{equation}
where
\[
  C_0:=(p-1)\alpha^{p-1},
  \qquad
  m:=\frac{(n-1)(p-2)}{n-p}.
\]
Since
\[
  W^{-m}=r^{-\gamma},
  \qquad
  \gamma:=\alpha m=\frac{(n-1)(p-2)}{p-1},
\]
this is the problem analyzed below.

\begin{remark}\label{remark2.1}
After the transformation, the first eigenpair is
\[
  \widetilde\mu_0=\mu_0=(p-1)\left(\frac{n-p}{p-1}\right)^{p-1},
  \qquad
  \varphi_0=1,
\]
and the second eigenvalue is
\[
  \widetilde\mu_1=\mu_1=(p_*-1)\left(\frac{n-p}{p-1}\right)^{p-1}.
\]
Moreover,
\[
  \spanop\left\{y_1,\ldots,y_{n-1},y_n-\frac1p\right\}
\]
is contained in the second eigenspace.

In addition, \eqref{eq6} becomes: 
if $\vp\perp \vp_0$ in $L^2(\partial\OmB;r^{-\g})$, then we have 
\begin{equation}
\int_{\OmB} \big(|D\vp|^2
+(p-2)(D r\cdot D\vp)^2\big)r^{-\g}\,d y\ge \frac{\tl\mu_1-C_0}{\a^{p-2}}\int_{\p\OmB}r^{-\g}\vp^2 d \sigma,\label{eq15}
\end{equation}
where $d \sigma$ is the area form of $\p\OmB$, and the functions in
\[
        \spanop\left\{y_1,\ldots,y_{n-1},y_n-\frac1p\right\}
\]
attain equality in \eqref{eq15}.

\end{remark}

\subsubsection{Transformed normalized problem and variational characterization}

Recall that
\[
\begin{gathered}
  1<p<n, \qquad
  \alpha:=\frac{n-p}{p-1}, \qquad
  \gamma:=\frac{(n-1)(p-2)}{p-1},\\
  \beta:=\frac{n-3}{2}, \qquad
  C_0:=(p-1)\alpha^{p-1}.
\end{gathered}
\]
From now on, we set
\[
  \OmB:=B\!\left(\frac12 e_n,\frac12\right)\subset\R^n.
\]
Use polar variables centered at the origin:
\[
  y=r\sigma,\qquad r=|y|,\qquad \sigma\in S^{n-1},\qquad \tau:=\sigma_n.
\]
Next write
\[
\sigma=\left(\sqrt{1-\tau^2}\,\xi,\tau\right),
\qquad \xi\in S^{n-2}.
\]

First consider the interior equation of \eqref{eq:model-eigenproblem}. Since
\[
DW=\alpha r^{\alpha-1}\sigma,
\qquad
|DW|=\alpha r^{\alpha-1},
\]
and
\[
D\varphi=\varphi_r\sigma+\frac{1}{r}\nabla_S\varphi,
\]
we get
\[
\mathcal A_\varphi
=
\alpha^{p-2}r^{-\gamma}
\left(
(p-1)\varphi_r\sigma+\frac{1}{r}\nabla_S\varphi
\right).
\]
Therefore, the interior equation is equivalent to
\[
\operatorname{div}
\left[
r^{-\gamma}
\left(
(p-1)\varphi_r\sigma+\frac{1}{r}\nabla_S\varphi
\right)
\right]
=0.
\]
Using the polar-coordinate divergence formula gives
\[
(p-1)\varphi_{rr}
+\frac{n-1}{r}\varphi_r
+\frac{1}{r^2}\Delta_S\varphi
=0.
\label{eq16}
\]

On \(\partial\OmB\), we have \(r=\tau\), \(\nu=2r\sigma-e_n\), and hence
\[
\sigma\cdot\nu=2r-\tau=\tau.
\]
Moreover,
\[
\nu-(\nu\cdot\sigma)\sigma=\tau\sigma-e_n=-\nabla_S\tau.
\]
Therefore
\[
\left((p-1)\varphi_r\sigma+\frac{1}{r}\nabla_S\varphi\right)\cdot\nu
=
(p-1)r\varphi_r
-\frac{1}{r}\nabla_S\tau\cdot\nabla_S\varphi.
\]
After cancellation of the common factor \(r^{-\gamma}\), the boundary condition in \eqref{eq:model-eigenproblem} becomes
\[
\alpha^{p-2}
\left(
(p-1)r\varphi_r
-\frac{1}{r}\nabla_S\tau\cdot\nabla_S\varphi
\right)
+C_0\varphi
=
\tilde{\mu}\varphi
\qquad \text{on } r=\tau.
\]

Note that the boundary of \(\OmB\) is
\[
  \partial\OmB=\{y:|y|^2=y_n\}=\{r=\tau\,\},
\]
and the interior is described by
\[
 \mathcal{D}:=\{(r,\tau): 0<\tau<1,\,\, 0<r<\tau\}.
\]
After subtracting the first eigenvalue contribution, the transformed Steklov parameter is
\[
  \lambda:=\frac{\widetilde\mu-C_0}{\alpha^{p-2}}.
\]
The normalized transformed problem \eqref{eq:model-eigenproblem} can be rewritten as follows.
For
\[
        \phi\in W^{1,2}(\OmB;r^{-\g}),
\]
it is
\begin{equation}\label{eq:normalized-problem}
\begin{cases}
\displaystyle
\operatorname{div}\left[
 r^{-\gamma}\left((p-1)\phi_r\sigma+\frac1r\gradS\phi\right)
\right]=0,
&\text{in }\OmB,\\[2mm]
\displaystyle
(p-1)r\phi_r-\frac1r\gradS\tau\cdot\gradS\phi=\lambda\phi,
&\text{on }r=\tau.
\end{cases}
\end{equation}
The first normalized eigenvalue is \(0\), with eigenfunction \(1\). The target second normalized eigenvalue is
\[
  \lambda_1=p.
\]
In the original transformed parameter this corresponds to
\[
  \widetilde\mu_1=C_0+p\alpha^{p-2}
  =n\alpha^{p-2}
  =(p_*-1)\alpha^{p-1},
  \qquad p_*:=\frac{p(n-1)}{n-p}.
\]

For $\phi,\psi\in W^{1,2}(\OmB;r^{-\g})$, define
\begin{equation}\label{eq:global-energy}
  \cE[\phi,\psi]
  :=\int_{\OmB} r^{-\gamma}
  \left[
    (p-1)\phi_r\psi_r
    +\frac1{r^2}\gradS\phi\cdot\gradS\psi
  \right]\dd y,
\end{equation}
and
\begin{equation}\label{eq:global-boundary}
  \cB[\phi,\psi]
  :=\int_{\partial\OmB}r^{-\gamma}\phi\psi\dd S.
\end{equation}
Then \eqref{eq:normalized-problem} is equivalent to
\begin{equation}\label{eq:weak-global}
  \cE[\phi,\psi]=\lambda\cB[\phi,\psi]
\end{equation}
for every \(\psi\in W^{1,2}(\OmB;r^{-\g})\). 
\begin{remark}
The desired spectral characterization at the transformed level is
\begin{equation}\label{eq:global-gap-goal}
  \cE[\phi,\phi]\ge p\cB[\phi,\phi]
  \qquad\text{whenever}\qquad
  \cB[\phi,1]=0.
\end{equation}
The equality case should be
\[
  \operatorname{span}\left\{y_1,\ldots,y_{n-1},y_n-\frac1p\right\}.
\]
\end{remark}

We consider a sector decomposition as in \cite[Section 5]{FLZ26}.
Write
\[
  \sigma=(\sqrt{1-\tau^2}\,\theta,\tau),
  \qquad \theta\in S^{n-2}.
\]
Then
\[
  |\gradS\phi|^2
  =(1-\tau^2)\phi_\tau^2
  +\frac1{1-\tau^2}|\nabla_\theta\phi|^2.
\]
Let
\[
\left\{Z_{k,j}: 1\le j\le \dim \mathcal{H}_k^{n-1}\right\}
\]
be an orthonormal basis of spherical harmonics of degree \(k\) on \(S^{n-2}\)
with respect to the normalized surface measure
\(d\sigma/|S^{n-2}|\), so that
\[
  -\Delta_{S^{n-2}}Z_{k,j}=\kappa_k Z_{k,j},
  \qquad \kappa_k:=k(k+n-3).
\]

For $\phi\in W^{1,2}(\OmB;r^{-\g})$, we can write
\begin{equation}
  \phi(r,\tau,\theta)=\sum_{k=0}^\infty\sum_{j=1}^{\dim \Hz^{n-1}_k} u_{k,j}(r,\tau)Z_{k,j}(\theta),
 \label{eq:expand}
\end{equation}
where the expansion is understood in the angular \(L^2(S^{n-2})\)-sense for a.e. \((r,\tau)\in \mathcal{D}\), and
\[
u_{k,j}(r,\tau)
=
\left\langle \phi(r,\tau,\cdot),Z_{k,j}\right\rangle_{S^{n-2}}
=
\frac{1}{|S^{n-2}|}
\int_{S^{n-2}}
\phi(r,\tau,\theta)Z_{k,j}(\theta)\,d\sigma(\theta).
\]
Therefore, we have $u_{k,j}(y):=u_{k,j}(r,\tau)\in W^{1,2}(\OmB;r^{-\g})$.

Let
\[
  w(\tau):=(1-\tau^2)^{\frac{n-3}{2}}.
\]
Since
\[
  n-2-\gamma=\alpha,
\]
the boundary form in the \(k\)-sector is
\begin{equation}\label{eq:sector-boundary}
  \cB_k[u,v]
  :=\int_0^1 \tau^\alpha w(\tau)u(\tau,\tau)v(\tau,\tau)\dd\tau.
\end{equation}
The \(k\)-sector energy is
\begin{equation}\label{eq:sector-energy}
\begin{aligned}
  \cE_k[u,v]
  :=\int_0^1\int_0^\tau r^{\a+1}w(\tau)
  \bigg[&
    (p-1)u_rv_r
    +\frac{1-\tau^2}{r^2}u_\tau v_\tau \\
    &+\frac{\kappa_k}{r^2(1-\tau^2)}uv
  \bigg]\dd r\dd\tau.
\end{aligned}
\end{equation}
Thus
\[
  \cE[\phi,\phi]=|S^{n-2}|\sum_{k,j}\cE_k[u_{k,j},u_{k,j}],
  \qquad
  \cB[\phi,\phi]=|S^{n-2}|\sum_{k,j}\cB_k[u_{k,j},u_{k,j}].
\]

\begin{remark}\label{regular}
    For $\phi\in W^{1,2}(\OmB;r^{-\g})$, we have
    \(u_{k,j}(\tau,\tau)\in L^2([0,1];\tau^{\alpha}(1-\tau^2)^{\beta})\) and
    \[
        u_{k,j}(r,\tau)\in W^{1,2}(\mathcal{D};\cE_k)
        :=\{u\in W^{1,2}_{loc}(\mathcal{D}):\cE_k[u,u]<\infty\}.
    \]
    For $k\ge 1$ and \(u\in W^{1,2}(\mathcal{D};\cE_k)\), we have
    \(u\in L^2(\mathcal{D};r^{\alpha-1}(1-\tau^2)^{\beta-1})\),
    \(u_r\in L^2(\mathcal{D};r^{\alpha+1}(1-\tau^2)^{\beta})\), and
    \(u_\tau\in L^2(\mathcal{D};r^{\alpha-1}(1-\tau^2)^{\beta+1})\). In addition, $W^{1,2}(\mathcal{D};\cE_k)=W^{1,2}(\mathcal{D};\cE_1)$ for $k\ge1$, because \(\kappa_1=n-2>0\).
\end{remark}

Let
\[
        A_{\mathrm{ax}}
        :=\{y=(y',y_n)\in\overline{\OmB}:y'=0\}.
\]
Here and below,
\(C_c^\infty(\overline{\OmB}\setminus A_{\mathrm{ax}})\) denotes the
restrictions to \(\overline{\OmB}\) of functions that are smooth in a
neighborhood of \(\overline{\OmB}\) and whose supports have positive
distance from \(A_{\mathrm{ax}}\).

\begin{claim}\label{claim:k1-density}
The space
\[
        C^\infty_{c}(\overline{\OmB}\setminus A_{\mathrm{ax}})
        \cap W^{1,2}(\mathcal D;\cE_1)
\]
is dense in \(W^{1,2}(\mathcal D;\cE_1)\) with respect to the norm
\[
        \|v\|_{\cE_1}^2:=\cE_1[v,v].
\]
\end{claim}

\begin{proof}[Proof of the claim]
Let \(v\in W^{1,2}(\mathcal D;\cE_1)\). Choose
\(\chi_\varepsilon,\psi_\varepsilon\in C^\infty([0,1])\) such that
\[
\begin{gathered}
        \chi_\varepsilon=0\text{ on }(0,\varepsilon),\qquad
        \chi_\varepsilon=1\text{ on }(2\varepsilon,1],\qquad
        |\chi_\varepsilon'|\le C/\varepsilon,\\
        \psi_\varepsilon=1\text{ on }[0,1-2\varepsilon],\qquad
        \psi_\varepsilon=0\text{ on }[1-\varepsilon,1],\qquad
        |\psi_\varepsilon'|\le C/\varepsilon.
\end{gathered}
\]
Set \(v_\varepsilon=\chi_\varepsilon(r)\psi_\varepsilon(\tau)v\). The derivative terms
created by the cutoffs are controlled by the angular-potential term in \(\cE_1\):
\[
\begin{aligned}
\int_{\{\varepsilon<r<2\varepsilon\}}
r^{\alpha+1}(1-\tau^2)^\beta|\chi_\varepsilon'|^2v^2\dd r\dd\tau
&\le
C\int_{\{r<2\varepsilon\}}
r^{\alpha-1}(1-\tau^2)^{\beta-1}v^2\dd r\dd\tau,
\\
\int_{\{1-2\varepsilon<\tau<1-\varepsilon\}}
r^{\alpha-1}(1-\tau^2)^{\beta+1}|\psi_\varepsilon'|^2v^2\dd r\dd\tau
&\le
C\int_{\{\tau>1-2\varepsilon\}}
r^{\alpha-1}(1-\tau^2)^{\beta-1}v^2\dd r\dd\tau.
\end{aligned}
\]
Both right-hand sides tend to zero by absolute continuity of the integral. The remaining
energy terms converge by dominated convergence. Hence \(v_\varepsilon\to v\) in the
\(\|\cdot\|_{\cE_1}\)-norm. After this cutoff, the support stays away from the singular
sets \(r=0\) and \(\tau=1\); on such truncated Lipschitz subdomains all weights are smooth
and bounded above and below. Standard \(H^1\)-mollification therefore gives approximants
in \(C^\infty_{c}(\overline{\OmB}\setminus A_{\mathrm{ax}})\cap W^{1,2}(\mathcal D;\cE_1)\).
\end{proof}

\begin{claim}\label{claim:k1-trace}
If \(u\in W^{1,2}(\mathcal D;\cE_1)\), then the boundary trace
\(u(\tau,\tau)\) belongs to
\[
        L^2([0,1];\tau^\alpha(1-\tau^2)^\beta).
\]
Moreover, there exists $C=C(n,p)>0$ such that
\[
        \int_0^1 |u(\tau,\tau)|^2\tau^\alpha(1-\tau^2)^\beta\dd\tau
        \le C\cE_1[u,u].
\]
\end{claim}

\begin{proof}
By Claim~\ref{claim:k1-density}, we only need to consider
$u\in C^\infty_{c}(\overline{\OmB}\setminus A_{\mathrm{ax}})
        \cap W^{1,2}(\mathcal D;\cE_1)$.
First, we use a one-dimensional trace inequality. For every
\(T\in(0,1)\) and every \(f\in W^{1,2}(\frac{T}{2},T)\),
\begin{equation}\label{eq:one-dimensional-trace-k1}
        T^\alpha |f(T)|^2
        \le
        C_\alpha\int_{\frac{T}{2}}^T
        \left(r^{\alpha+1}|f'(r)|^2+r^{\alpha-1}|f(r)|^2\right)\dd r.
\end{equation}
Indeed, after the scaling \(r=Ts\), this follows from the usual \(H^1\)-trace
inequality on \((1/2,1)\).

Apply \eqref{eq:one-dimensional-trace-k1} with \(T=\tau\) and
\(f(r)=u(r,\tau)\). Multiplying by \(w(\tau)=(1-\tau^2)^\beta\) and integrating
in \(\tau\), we obtain
\[
\begin{aligned}
        \int_0^1 |u(\tau,\tau)|^2\tau^\alpha w(\tau)\dd\tau
        &\le
        C\int_0^1\int_0^\tau
        r^{\alpha+1}w(\tau)|u_r|^2\dd r\dd\tau\\
        &\quad+
        C\int_0^1\int_0^\tau
        r^{\alpha-1}w(\tau)|u|^2\dd r\dd\tau.
\end{aligned}
\]
Since \(0<1-\tau^2\le1\), the
right-hand side is controlled by \(\cE_1[u,u]\). This proves the estimate and hence the claim.
\end{proof}

In the \(k\)-sector, the strong equation is
\begin{equation}\label{eq:sector-pde}
  (p-1)u_{rr}
  +\frac{n-1}{r}u_r
  +\frac1{r^2}\left[
    (1-\tau^2)u_{\tau\tau}
    -(n-1)\tau u_\tau
    -\frac{\kappa_k}{1-\tau^2}u
  \right]=0,
\end{equation}
and the boundary condition on \(r=\tau\) is
\begin{equation}\label{eq:sector-boundary-condition}
  (p-1)\tau u_r
  -\frac{1-\tau^2}{\tau}u_\tau
  =\lambda u.
\end{equation}

\subsubsection{Explicit eigenfunctions at \texorpdfstring{\(\lambda=p\)}{lambda=p}}

\paragraph{The axisymmetric mode}
\label{subsubsec:axisymmetric-mode}

Let
\[
  h_0(r,\tau):=r\tau-\frac1p.
\]
Since \(r\tau=y_n\), this is
\[
  h_0=y_n-\frac1p.
\]
It solves the interior equation in the \(k=0\) sector. On \(r=\tau\),
\[
  h_0(\tau,\tau)=\tau^2-\frac1p.
\]
Also
\[
  (h_0)_r=\tau,
  \qquad
  (h_0)_\tau=r.
\]
Therefore
\[
\begin{aligned}
  (p-1)\tau(h_0)_r
  -\frac{1-\tau^2}{\tau}(h_0)_\tau
  &=(p-1)\tau^2-(1-\tau^2)\\
  &=p\tau^2-1\\
  &=p\left(\tau^2-\frac1p\right)\\
  &=p h_0(\tau,\tau).
\end{aligned}
\]
Hence \(h_0\) is a \(k=0\) eigenfunction with eigenvalue \(p\).

It is also \(\cB_0\)-orthogonal to constants. Indeed,
\[
  \cB_0[h_0,1]
  =\int_0^1\left(\tau^2-\frac1p\right)
  \tau^\alpha(1-\tau^2)^{\frac{n-3}{2}}\dd\tau.
\]
The beta-function identity gives
\[
\frac{
\int_0^1\tau^{\alpha+2}(1-\tau^2)^{\frac{n-3}{2}}\dd\tau
}{
\int_0^1\tau^\alpha(1-\tau^2)^{\frac{n-3}{2}}\dd\tau
}
=\frac{\alpha+1}{\alpha+n}=\frac1p,
\]
so \(\cB_0[h_0,1]=0\).

\paragraph{The transverse modes}

For \(1\le i\le n-1\),
\[
  y_i=r\sqrt{1-\tau^2}\,\theta_i.
\]
Thus the common \(k=1\) amplitude is
\[
  h_1(r,\tau):=r\sqrt{1-\tau^2}.
\]
Here \(\kappa_1=n-2\). The functions \(h_1Z_{1,j}\) solve the interior equation. On \(r=\tau\),
\[
  h_1(\tau,\tau)=\tau\sqrt{1-\tau^2}.
\]
Also
\[
  (h_1)_r=\sqrt{1-\tau^2},
  \qquad
  (h_1)_\tau=-\frac{r\tau}{\sqrt{1-\tau^2}}.
\]
Therefore
\[
\begin{aligned}
  (p-1)\tau(h_1)_r
  -\frac{1-\tau^2}{\tau}(h_1)_\tau
  &=(p-1)\tau\sqrt{1-\tau^2}
    +\tau\sqrt{1-\tau^2}\\
  &=p\tau\sqrt{1-\tau^2}\\
  &=p h_1(\tau,\tau).
\end{aligned}
\]
Thus each \(y_i\), \(1\le i\le n-1\), is a transformed eigenfunction with eigenvalue \(p\).

\subsubsection{The sectors for
\texorpdfstring{$k\ge 1$}{k >= 1}}\label{subsec}

\begin{lemma}[Ground-state identity in the transverse sector]\label{lem:k1-ground-state}
For every
\[
  u=u(r,\tau)\in W^{1,2}(\mathcal{D};\cE_1)
\]
in the \(k=1\) sector, write \(u=h_1\eta\). Then
\begin{equation}\label{eq:k1-ground-state-identity}
  \cE_1[u,u]-p\cB_1[u,u]
  =\int_0^1\int_0^\tau r^{\a+1}w(\tau)h_1^2
  \left[
    (p-1)\eta_r^2
    +\frac{1-\tau^2}{r^2}\eta_\tau^2
  \right]\dd r\dd\tau.
\end{equation}
Consequently,
\[
  \cE_1[u,u]\ge p\cB_1[u,u],
\]
with equality if and only if \(u\) is a constant multiple of \(h_1\).
\end{lemma}

\begin{proof}
We first prove the identity for
\[
        u(y):=u(r,\tau)\in
        C^\infty_{c}(\overline{\OmB}\setminus A_{\mathrm{ax}}).
\]
Since \(h_1(r,\tau)=r\sqrt{1-\tau^2}\) is positive in \(\mathcal D\), set
\(\eta:=u/h_1\). Notice first that
\begin{equation}\label{eq:k1-eta-derivatives}
        h_1\eta_r=u_r-\frac{u}{r},
        \qquad
        h_1\eta_\tau=u_\tau+\frac{\tau}{1-\tau^2}u.
\end{equation}
Since \(\kappa_1=n-2>0\), the definition of \(\cE_1\) and
\eqref{eq:k1-eta-derivatives} imply
\begin{align}
&\int_{\mathcal D}
 r^{\alpha+3}(1-\tau^2)^{\beta+1}|\eta_r|^2\dd r\dd\tau
 \le C\cE_1[u,u],
\label{eq:k1-eta-r-integrability}\\
&\int_{\mathcal D}
 r^{\alpha+1}(1-\tau^2)^{\beta+2}|\eta_\tau|^2\dd r\dd\tau
 \le C\cE_1[u,u].
\label{eq:k1-eta-tau-integrability}
\end{align}

The function \(h_1\) satisfies
\(
  \cE_1[h_1,\zeta]=p\cB_1[h_1,\zeta]
\)
for every \(\zeta\in W^{1,2}(\mathcal D;\cE_1)\).

Note that the support of \(\eta\) stays away from both \(r=0\) and
\(\tau=1\). Consequently
\(h_1\eta^2\in W^{1,2}(\mathcal D;\cE_1)\), so it is an
admissible test function in the equation for \(h_1\). Expanding
\(\cE_1[h_1\eta,h_1\eta]\) and testing with
\(h_1\eta^2\) gives
\begin{equation}\label{eq:k1-ground-state-cutoff}
\begin{aligned}
  &\cE_1[u,u]
  -p\cB_1[u,u]\\
  &\quad=
  \int_0^1\int_0^\tau r^{\a+1}(1-\tau^2)^\beta h_1^2
  \left[
    (p-1)|\eta_r|^2
    +\frac{1-\tau^2}{r^2}|\eta_\tau|^2
  \right]\dd r\dd\tau.
\end{aligned}
\end{equation}

We next consider general $u\in W^{1,2}(\mathcal D;\cE_1)$.
Claim~\ref{claim:k1-density} gives
$u_\varepsilon\in C^\infty_{c}(\overline{\OmB}\setminus A_{\mathrm{ax}})$ satisfying
\[
        u_\varepsilon\longrightarrow u
        \qquad\text{in }W^{1,2}(\mathcal D;\cE_1).
\]
By Claim~\ref{claim:k1-trace}, the boundary traces converge in
\(L^2([0,1];\tau^\alpha(1-\tau^2)^\beta)\). It remains to treat the
right-hand side of \eqref{eq:k1-ground-state-cutoff}. Denote
$\eta_\varepsilon:=u_\varepsilon/h_1$. Since
$\eta_\varepsilon-\eta=(u_\varepsilon-u)/h_1$, the estimates
\eqref{eq:k1-eta-r-integrability} and
\eqref{eq:k1-eta-tau-integrability}, applied to
$u_\varepsilon-u$, give
\[
 \int_{\mathcal D}
 r^{\alpha+3}(1-\tau^2)^{\beta+1}
 |(\eta_\varepsilon-\eta)_r|^2\dd r\dd\tau
 \le C\cE_1[u_\varepsilon-u,u_\varepsilon-u]\longrightarrow0,
\]
and
\[
 \int_{\mathcal D}
 r^{\alpha+1}(1-\tau^2)^{\beta+2}
 |(\eta_\varepsilon-\eta)_\tau|^2\dd r\dd\tau
 \le C\cE_1[u_\varepsilon-u,u_\varepsilon-u]\longrightarrow0.
\]

Letting \(\varepsilon\downarrow0\) in
\eqref{eq:k1-ground-state-cutoff} with $u=u_\varepsilon$ proves
\eqref{eq:k1-ground-state-identity} for every \(u\in W^{1,2}(\mathcal D;\cE_1)\).

The right-hand side vanishes exactly when \(\eta\) is constant in the
connected domain \(\mathcal D\), which proves the equality statement.
\end{proof}

\begin{lemma}[No equality at \(p\) for \(k\ge2\)]\label{lem:higher-sectors}
For every \(k\ge2\) and every $u(r,\tau)\in W^{1,2}(\mathcal{D};\cE_k)$ in the \(k\)-sector,
\[
  \cE_k[u,u]\ge p\cB_k[u,u].
\]
Moreover, if \(u\not\equiv0\), then equality cannot hold. Hence no nonzero \(k\ge2\) sector component can belong to the eigenspace \(\lambda=p\).
\end{lemma}

\begin{proof}
For \(k\ge2\),
\[
  \kappa_k=k(k+n-3)>n-2=\kappa_1.
\]
For the same amplitude \(u\), the boundary form is independent of \(k\), while
\[
  \cE_k[u,u]
  =\cE_1[u,u]
  +(\kappa_k-\kappa_1)
  \int_0^1\int_0^\tau r^{\a+1}w(\tau)
  \frac{u^2}{r^2(1-\tau^2)}\dd r\dd\tau.
\]
By Lemma~\ref{lem:k1-ground-state},
\[
  \cE_1[u,u]\ge p\cB_1[u,u]=p\cB_k[u,u].
\]
Therefore \(\cE_k[u,u]\ge p\cB_k[u,u]\). If equality held and \(u\not\equiv0\), then the positive angular-potential remainder would have to vanish:
\[
  \int_0^1\int_0^\tau r^{\a+1}w(\tau)
  \frac{u^2}{r^2(1-\tau^2)}\dd r\dd\tau=0,
\]
which forces \(u\equiv0\), a contradiction.
\end{proof}

\subsubsection{The axisymmetric sector \texorpdfstring{$k=0$}{k = 0}}\label{subsec2}

The sector decomposition reduces the only delicate part to the following sharp two-variable inequality, which is related to a two-variable eigenproblem.

\begin{theorem}
  \label{ass:axisymmetric-gap}
For every axisymmetric $u(y):=u(r,\tau)\in W^{1,2}(\OmB;r^{-\g})$ in the \(k=0\) sector satisfying
\begin{equation}\label{eq:k0-orthogonality}
  \cB_0[u,1]
  =\int_0^1 u(\tau,\tau)\tau^\alpha(1-\tau^2)^{\frac{n-3}{2}}\dd\tau
  =0,
\end{equation}
one has
\begin{equation}\label{eq:k0-gap}
\begin{aligned}
  &\int_0^1\int_0^\tau r^{\a+1}(1-\tau^2)^{\frac{n-3}{2}}
  \left[
    (p-1)u_r^2+\frac{1-\tau^2}{r^2}u_\tau^2
  \right]\dd r\dd\tau \\
  &\hspace{32mm}\ge
  p\int_0^1\tau^\alpha(1-\tau^2)^{\frac{n-3}{2}}u(\tau,\tau)^2\dd\tau.
\end{aligned}
\end{equation}
Equality holds if and only if
\[
  u=c\left(r\tau-\frac1p\right).
\]
\end{theorem}

\begin{remark}\label{remark2.6}
Note that $\frac{\tl\mu_1-C_0}{\a^{p-2}}=p$.
It follows from Remark~\ref{remark2.1} and \eqref{eq15} that
if
\[
        u\perp 1
        \quad\text{in }L^2(\partial\OmB;r^{-\g}),
\]
then
\begin{equation}
\int_{\OmB} \big(|D u|^2
+(p-2)(D r\cdot D u)^2\big)r^{-\g}\,d y\ge p\int_{\p\OmB}r^{-\g}u^2 d \sigma,\label{eq}
\end{equation}
where $d \sigma$ is the area form of $\p\OmB$. Taking $u=u(r,\tau)$ under the condition
\eqref{eq:k0-orthogonality}, we obtain \eqref{eq:k0-gap}. Therefore, we only need to
characterize the equality case.
\end{remark}

Recall that
\[
  1<p<n,\qquad
  \alpha:=\frac{n-p}{p-1},\qquad
  \beta=\frac{n-3}{2},
\]
and
\[
  \bD:=\{(r,\tau):0<r<\tau<1\}.
\]
Recall that for an axisymmetric function $u=u(r,\tau)$ 
\[
  \cE_0[u,v]
  :=\int_0^1\int_0^\tau
  r^{\alpha+1}(1-\tau^2)^\beta
  \left[
    (p-1)u_rv_r+\frac{1-\tau^2}{r^2}u_\tau v_\tau
  \right]\dd r\dd\tau
\]
and
\[
  \cB_0[u,v]
  :=\int_0^1
  u(\tau,\tau)v(\tau,\tau)
  \tau^\alpha(1-\tau^2)^\beta\dd\tau.
\]
The $k=0$ sector eigenvalue problem is the weak problem
\begin{equation}\label{eq:weak}
  \cE_0[u,\zeta]=\lambda\cB_0[u,\zeta]
  \qquad\text{for every axisymmetric }\zeta(y):=\zeta(r,\tau)\in W^{1,2}(\OmB;r^{-\g}).
\end{equation}
Equivalently, in strong form,
\begin{equation}\label{eq:interior}
  (p-1)u_{rr}+\frac{n-1}{r}u_r
  +\frac1{r^2}\left((1-\tau^2)u_{\tau\tau}-(n-1)\tau u_\tau\right)=0
\end{equation}
in $\bD$, with boundary condition on $r=\tau$,
\begin{equation}\label{eq:steklov-bc}
  (p-1)\tau u_r-\frac{1-\tau^2}{\tau}u_\tau=\lambda u.
\end{equation}

Note that if $u$ attains equality in \eqref{eq:k0-gap}, then $u$
satisfies \eqref{eq:k0-orthogonality} and \eqref{eq:weak} with $\l=p$.
Since $u(y)=u(r,\tau)\in W^{1,2}(\OmB;r^{-\g})$, we have
$u_r\in L^2(\mathcal{D};r^{\alpha+1}(1-\tau^2)^\beta)$,
$u_\tau\in L^2(\mathcal{D};r^{\alpha-1}(1-\tau^2)^{\beta+1})$, and
the boundary trace
$b(\tau):=u(\tau,\tau)\in
L^2([0,1];\tau^{\alpha}(1-\tau^2)^{\beta})$.
Moreover, elliptic regularity yields
$u\in C^\infty_{loc}(\ov\OmB\bs\{0\})$; see
\cite{ADN1959,LionsMagenes1972}.

The proof of the Pohozaev-type identity also requires the following
regularity lemma.

\begin{lemma}[Regularity of the dilation derivative]\label{lem:dilation-regularity}
Let \(u(y)=u(r,\tau)\in W^{1,2}(\OmB;r^{-\g})\) be an axisymmetric weak solution of
\eqref{eq:weak}. Then
\[
        r u_r\in W^{1,2}(\OmB;r^{-\g}).
\]
\end{lemma}

\begin{proof}
Write
\[
        X:=y\cdot\nabla=r\partial_r.
\]
We prove \(Xu\in W^{1,2}(\OmB;r^{-\gamma})\). Since \(r\le1\) in \(\OmB\),
\[
        |Xu|\le r|\nabla u|\le |\nabla u|.
\]
The sector energy \(\cE_0[u,u]\) controls \(\int_{\OmB}r^{-\gamma}|\nabla u|^2\dd y\), and hence
\[
        \int_{\OmB}r^{-\gamma}|Xu|^2\dd y<\infty.
\]

It remains to prove \(\nabla Xu\in L^2(\OmB;r^{-\gamma})\). Since \(u\) is smooth on
\(\overline{\OmB}\setminus\{0\}\), only the singular point \(0\) has to be considered. Let
\[
        A_\rho:=\OmB\cap\{\rho<r<2\rho\},
        \qquad
        \widetilde A_\rho:=\OmB\cap\{\rho/2<r<4\rho\},
        \qquad 0<\rho<\rho_0:=\frac{1}{8}.
\]
In Cartesian variables, the equation associated with \eqref{eq:weak} has the divergence form
\begin{equation*}
\begin{cases}
\displaystyle
\operatorname{div}\!\left(r^{-\gamma}
        \bigl(I+(p-2)e_r\otimes e_r\bigr)\nabla u\right)=0,
&\text{in }\OmB,\\[2mm]
\displaystyle
\left(r^{-\gamma}
        \bigl(I+(p-2)e_r\otimes e_r\bigr)\nabla u\right)\cdot\nu=\lambda r^{-\gamma}u.
&\text{on }\p\OmB,
\end{cases}
\end{equation*}
On each annulus
\(\widetilde A_\rho\), after the rescaling \(y=\rho z\), the coefficient matrix becomes, up to
the harmless scalar factor \(\rho^{-\gamma}\),
\[
        |z|^{-\gamma}\bigl(I+(p-2)e_z\otimes e_z\bigr).
\]
Then, $\tl u(z):=u(\rho z)$ satisfies
\begin{equation*}
\begin{cases}
\displaystyle
\operatorname{div}\!\left(|z|^{-\gamma}\bigl(I+(p-2)e_z\otimes e_z\bigr)\nabla\tl u\right)=0,
&\text{in }\frac{1}{\rho}\OmB,\\[2mm]
\displaystyle
\left(|z|^{-\gamma}\bigl(I+(p-2)e_z\otimes e_z\bigr)\nabla\tl u\right)\cdot\nu=\lambda\rho |z|^{-\gamma}\tl u.
&\text{on }\frac{1}{\rho}(\p\OmB),
\end{cases}
\end{equation*}
The coefficient matrix is smooth and uniformly elliptic on the fixed annulus
\(\{1/2<|z|<4\}\), and the rescaled boundary $\frac{1}{\rho}(\p\OmB\cap\p \widetilde A_\rho)$ has uniformly bounded \(C^2\)-norm. The rescaled Robin coefficient
\(\lambda\rho\) is uniformly bounded. Therefore the standard interior and boundary
\(W^{2,2}\) estimates for uniformly elliptic equations with conormal boundary
conditions \cite[Chapter 2, Section 4]{LionsMagenes1972} give a
constant \(C\), independent of \(\rho\), such that
\begin{equation}\label{eq:dyadic-H2}
        \int_{A_\rho} r^{-\gamma}r^2|D^2u|^2\dd y
        \le
        C\left(
        \int_{\widetilde A_\rho}r^{-\gamma}|\nabla u|^2\dd y
        +
        \int_{\partial\OmB\cap\widetilde A_\rho}r^{-\gamma}u^2\dd S
        \right).
\end{equation}
Here the boundary term is finite because \(u\) is an eigenfunction:
\[
        \cE_0[u,u]=\lambda\cB_0[u,u]<\infty.
\]
Summing \eqref{eq:dyadic-H2} over the dyadic annuli \(\rho=2^{-j}\rho_0\), and using finite
overlap, yields
\[
        \int_{\OmB\cap\{r<\rho_0\}}r^{-\gamma}r^2|D^2u|^2\dd y<\infty.
\]
Since $u\in C^\infty_{loc}(\ov\OmB\bs\{0\})$, we obtain
\[
        \int_{\OmB}r^{-\gamma}r^2|D^2u|^2\dd y<\infty.
\]

Finally, note that
\(\nabla(Xu)=\nabla u+D^2u\, y\),
so we have
\[
        |\nabla(Xu)|^2\le 2|\nabla u|^2+2r^2|D^2u|^2.
\]
The two terms on the right-hand side are integrable with weight \(r^{-\gamma}\). Therefore
\[
        \int_{\OmB}r^{-\gamma}|\nabla(Xu)|^2\dd y<\infty.
\]
Together with \(Xu\in L^2(\OmB;r^{-\gamma})\), this proves
Lemma~\ref{lem:dilation-regularity}.
\end{proof}

We now prove the following Pohozaev-type identity, which reduces the
two-variable eigenproblem to the equality case of a one-dimensional
inequality.

\begin{lemma}[Pohozaev-type identity]\label{l:pohozaev}
Suppose that an axisymmetric function
\[
        u(y):=u(r,\tau)\in W^{1,2}(\OmB;r^{-\g})
\]
satisfies the eigenproblem \eqref{eq:weak} with eigenvalue \(\l=p\), and set
\[
  b(\tau):=u(\tau,\tau),
  \qquad
  M(\tau):=1+(p-2)\tau^2.
\]
Then we have
\[\int_0^1 b(\tau)\tau^\alpha(1-\tau^2)^{\frac{n-3}{2}}\dd\tau
  =0,\]
and
\begin{align}
    \int_0^1
b'(\tau)^2\frac{
  (p-1)\tau^2(1-\tau^2)
}{M(\tau)}
\tau^\alpha(1-\tau^2)^\beta\dd\tau 
=p\int_0^1
b(\tau)^2
  N(\tau)
\tau^\alpha(1-\tau^2)^\beta\dd\tau,\label{eq:pohozaev}
\end{align}
where
\begin{align*}
    N(\tau)=-\frac{1}{M(\tau)}+\frac{2(n-1)\tau^2}{M(\tau)}+\frac{2(p-1)\tau^2}{M(\tau)^2}.
\end{align*}
\end{lemma}
\begin{proof}
By testing the weak equation with the first eigenfunction $u_0\equiv1$, we obtain
\[\int_0^1 b(\tau)\tau^\alpha(1-\tau^2)^{\frac{n-3}{2}}\dd\tau
  =0.\]

The Pohozaev-type identity \eqref{eq:pohozaev} is obtained by testing the eigenvalue equation with the dilation derivative
\[
  \zeta=r u_r.
\]
By Lemma~\ref{lem:dilation-regularity}, this is an admissible test
function.
We write the computation for the test function $ru_r$.
\begin{remark}\label{remark2.8}
Note that Lemma~\ref{lem:dilation-regularity} implies
$f(\tau):=\tau u_r(\tau,\tau)
\in L^2([0,1];\tau^{\alpha}(1-\tau^2)^{\beta})$.
    Then by boundary condition \eqref{eq:steklov-bc}, we have $U(\tau):=u_r(\tau,\tau)\in L^2([0,1];\tau^{\alpha+2}(1-\tau^2)^{\beta})$, $V(\tau):=u_\tau(\tau,\tau)\in L^2([0,1];\tau^{\alpha-2}(1-\tau^2)^{\beta+2})$ and $b'=U+V\in L^2([0,1];\tau^{\alpha+2}(1-\tau^2)^{\beta+2})$.
\end{remark}

From \eqref{eq:weak},
\begin{equation}\label{eq:test-start}
  \cE_0[u,ru_r]
  =\lambda\int_0^1
  b(\tau)\,\tau u_r(\tau,\tau)
  \tau^\alpha(1-\tau^2)^\beta\dd\tau.
\end{equation}
We compute the left-hand side. First,
\[
  (ru_r)_r=u_r+ru_{rr},
  \qquad
  (ru_r)_\tau=ru_{r\tau}.
\]
Hence
\begin{align*}
  \cE_0[u,ru_r]
  &=\int_0^1\int_0^\tau
  r^{\alpha+1}(1-\tau^2)^\beta
  \left[
    (p-1)u_r(u_r+ru_{rr})
    +\frac{1-\tau^2}{r^2}u_\tau(ru_{r\tau})
  \right]\dd r\dd\tau.
\end{align*}
We first justify the integration by parts at \(r=0\). By
Lemma~\ref{lem:dilation-regularity} and Fubini's theorem, for a.e. \(\tau\in(0,1)\) the functions
\[
  r\longmapsto r^{\alpha+2}u_r(r,\tau)^2,
  \qquad
  r\longmapsto r^\alpha u_\tau(r,\tau)^2
\]
belong to \(W^{1,1}(0,\tau)\). Moreover, finite energy gives
\[
  \int_0^\tau \frac{r^{\alpha+2}u_r(r,\tau)^2}{r}\dd r<\infty,
  \qquad
  \int_0^\tau \frac{r^\alpha u_\tau(r,\tau)^2}{r}\dd r<\infty.
\]
Therefore, each of the two \(W^{1,1}\)-functions has a limit at \(r=0\),
and that limit must be zero. Thus there is no boundary contribution at
\(r=0\). Integrating by parts in the \(r\) variable gives
\begin{align*}
  &\int_0^1 w(\tau)\int_0^\tau r^{\alpha+1}u_r(u_r+ru_{rr})\dd r \dd \tau\\
  &=\int_0^1 w(\tau)\int_0^\tau r^{\alpha+1}u_r^2\dd r\dd \tau
    +\frac12\int_0^1 w(\tau)\int_0^\tau r^{\alpha+2}\partial_r(u_r^2)\dd r \dd \tau \\
  &=\frac12\int_0^1\tau^{\alpha+2}(1-\tau^2)^\beta u_r(\tau,\tau)^2\dd \tau
    -\frac\alpha2\int_0^1w(\tau)\int_0^\tau r^{\alpha+1}u_r^2\dd r\dd \tau,
\end{align*}
and
\begin{align*}
  &\int_0^1(1-\tau^2)^{\beta+1}\int_0^\tau r^{\alpha-1}u_\tau(ru_{r\tau})\dd r\dd\tau
  =\frac12\int_0^1(1-\tau^2)^{\beta+1}\int_0^\tau r^\alpha\partial_r(u_\tau^2)\dd r \dd\tau \\
  &=\frac12\int_0^1\tau^\alpha(1-\tau^2)^{\beta+1} u_\tau(\tau,\tau)^2\dd\tau
    -\frac\alpha2\int_0^1(1-\tau^2)^{\beta+1}\int_0^\tau r^{\alpha-1}u_\tau^2\dd r\dd\tau.
\end{align*}
Therefore
\begin{equation}\label{eq:lhs-rellich}
\begin{aligned}
  \cE_0[u,ru_r]
  &=\frac12\int_0^1
  \left[(p-1)\tau^2u_r(\tau,\tau)^2
        +(1-\tau^2)u_\tau(\tau,\tau)^2\right]
  \tau^\alpha(1-\tau^2)^\beta\dd\tau  \\
  &\qquad -\frac\alpha2\cE_0[u,u].
\end{aligned}
\end{equation}
Taking $\zeta=u$ in \eqref{eq:weak} gives
\[
  \cE_0[u,u]=\lambda\cB_0[u,u]
  =\lambda\int_0^1 b(\tau)^2\tau^\alpha(1-\tau^2)^\beta\dd\tau.
\]
Combining this with \eqref{eq:test-start} and \eqref{eq:lhs-rellich}, we obtain
\begin{equation}\label{eq:boundary-rellich-raw}
\begin{aligned}
&\int_0^1
\left\{
\frac12\left[(p-1)\tau^2U(\tau)^2+(1-\tau^2)V(\tau)^2\right]
-\lambda\tau b(\tau)U(\tau)
\right\}
\tau^\alpha(1-\tau^2)^\beta\dd\tau  \\
&\hspace{35mm}
=\frac{\alpha\lambda}{2}
\int_0^1 b(\tau)^2\tau^\alpha(1-\tau^2)^\beta\dd\tau,
\end{aligned}
\end{equation}
where
\[
  U(\tau):=u_r(\tau,\tau),
  \qquad
  V(\tau):=u_\tau(\tau,\tau).
\]
It remains to express $U$ and $V$ in terms of the boundary trace $b$.
Since
\[
  b'(\tau)=u_r(\tau,\tau)+u_\tau(\tau,\tau)=U(\tau)+V(\tau),
\]
and the Steklov condition \eqref{eq:steklov-bc} gives
\[
  (p-1)\tau U(\tau)-\frac{1-\tau^2}{\tau}V(\tau)=\lambda b(\tau),
\]
we solve the resulting linear system to get
\begin{equation}\label{eq:UV}
  U(\tau)=\frac{\lambda\tau b(\tau)+(1-\tau^2)b'(\tau)}{M(\tau)},
  \qquad
  V(\tau)=\frac{(p-1)\tau^2b'(\tau)-\lambda\tau b(\tau)}{M(\tau)}.
\end{equation}
Substituting \eqref{eq:UV} into the integrand in \eqref{eq:boundary-rellich-raw} yields the elementary identity
\begin{align*}
&\frac12\left[(p-1)\tau^2U^2+(1-\tau^2)V^2\right]-\lambda\tau bU \\
&\qquad =
\frac{
  (p-1)\tau^2(1-\tau^2)b'^2
  -2\lambda\tau(1-\tau^2)bb'
  -\lambda^2\tau^2b^2
}{2M(\tau)}.
\end{align*}
Together with \eqref{eq:boundary-rellich-raw}, this proves 
\begin{equation}\label{eq:rellich-identity}
\begin{aligned}
&\int_0^1
\frac{
  (p-1)\tau^2(1-\tau^2)b'(\tau)^2
  -2\lambda\tau(1-\tau^2)b(\tau)b'(\tau)
  -\lambda^2\tau^2b(\tau)^2
}{M(\tau)}
\tau^\alpha(1-\tau^2)^\beta\dd\tau \\
&\hspace{35mm}
=\alpha\lambda
\int_0^1 b(\tau)^2\tau^\alpha(1-\tau^2)^\beta\dd\tau.
\end{aligned}
\end{equation}
Integrating by parts gives
\begin{align}
    &\int_0^1
\frac{
  -2\tau(1-\tau^2)b(\tau)b'(\tau)
}{M(\tau)}
\tau^\alpha(1-\tau^2)^\beta\dd\tau
=\int_0^1
b(\tau)^2\left(\frac{\tau^{\alpha+1}(1-\tau^2)^{\beta+1}
}{M(\tau)}\right)'
\dd\tau\non\\
=&\int_0^1
b(\tau)^2\left(\frac{(\a+1)(1-\tau^2)-2(\b+1)\tau^2
}{M(\tau)}-\frac{2(p-2)\tau^2(1-\tau^2)}{M(\tau)^2}\right)\tau^{\alpha}(1-\tau^2)^{\beta}
\dd\tau,\label{eq37}
\end{align}
Indeed, to justify the absence of endpoint terms in this integration by parts,
set
\[
  G(\tau):=
  \frac{b(\tau)^2\tau^{\alpha+1}(1-\tau^2)^{\beta+1}}{M(\tau)}.
\]
The weighted estimates for \(b\) and \(b'\) in
Remark~\ref{remark2.8}, together with Cauchy--Schwarz, imply that
\(G\in W^{1,1}(0,1)\).
Thus \(G\) has finite limits at
both endpoints, and then
the weighted integrability of \(b\) forces both endpoint limits of \(G\)
to be zero. Hence the integration by parts produces no term at either
\(\tau=0\) or \(\tau=1\).

Substituting \eqref{eq37} into \eqref{eq:rellich-identity}, we obtain
\begin{align}
    \int_0^1
b'(\tau)^2\frac{
  (p-1)\tau^{\a+2}(1-\tau^2)^{\b+1}
}{M(\tau)}
\dd\tau 
=\l\int_0^1
b(\tau)^2
  N(\tau)
\tau^\alpha(1-\tau^2)^\beta\dd\tau,\label{eq38}
\end{align}
where
\begin{align*}
    N(\tau)=-\frac{1}{M(\tau)}+\frac{2(n-1)\tau^2}{M(\tau)}+\frac{2(p-1)\tau^2}{M(\tau)^2}.
\end{align*}

\end{proof}

Set
\begin{align*}
    &P(t):=\frac{(p-1)t^{\a+2}(1-t^2)^{\b+1}}{1+(p-2)t^2}>0\qquad\T{in }(0,1),\\
    &\rho(t):=t^{\a}(1-t^2)^{\b}>0\qquad\T{in }(0,1).
\end{align*}
Define the weighted Sobolev space
\[
 V:=\left\{
 f\in W^{1,2}_{\mathrm{loc}}(0,1):
 \int_0^1\rho(t)|f(t)|^2\dd t
 +\int_0^1P(t)|f'(t)|^2\dd t<\infty
 \right\}.
\]
For $f\in V$, define the quadratic form
\begin{equation}\label{eq:Qdef}
       \calQ[f]
        :=
        \int_0^1 P(t)f'(t)^2\dd t
        -p\int_0^1 N(t)\rho(t)f(t)^2\dd t.
\end{equation}

Indeed, Lemma~\ref{l:pohozaev} reduces
Theorem~\ref{ass:axisymmetric-gap} to the following sharp
one-dimensional inequality.
\begin{proposition}\label{p:ineq}  
If $f\in V$ satisfies
\begin{equation}\label{eq:meanzero}
        \int_0^1 f(t)\rho(t)\dd t=0,
\end{equation}
then
\begin{equation}\label{eq:Qnonnegative}
        \calQ[f]\ge 0.
\end{equation}
Moreover,
\begin{equation}\label{eq:equalitycase}
        \calQ[f]=0
        \quad\Longleftrightarrow\quad
        f(t)=C\left(t^2-\frac1p\right)
\end{equation}
for some constant $C\in\R$.
\end{proposition}

Now we consider the variation of \eqref{eq:Qdef} and obtain a one-dimensional
eigenproblem. Suppose that
\[
        f\in V
\]
is an eigenfunction with eigenvalue \(\delta\), that is,
  \begin{equation}\label{eq:ode}
\begin{cases}
\displaystyle
\frac{1}{\rho}\left[
(P f')'-Q f
\right]+\d f=0,
&\text{in }(0,1),\\[2mm]
\displaystyle
\lim_{t\downarrow0}P(t)f'(t)=0,
\qquad
\lim_{t\uparrow1}P(t)f'(t)=0
\end{cases}
\end{equation}
where
\begin{align*}
    Q(t):=-p N(t) \rho(t)\qquad\T{in }(0,1).
\end{align*}
Denote by \(\calL\) the formal differential expression
\[
        \calL f:=-\frac{1}{\rho}\left[(P f')'-Q f\right].
\]
More precisely, let
\[
        H:=L^2([0,1];\rho),\quad  \langle f,g\rangle_H:=\int_0^1 f g \rho \dd t.
\]
With the norm
\[
        \|f\|_V^2:=\|f\|_H^2+\int_0^1P(t)|f'(t)|^2\dd t,
\]
the space \(V\) is complete. Hence, $V$ and $H$ are Hilbert spaces.
The symmetric quadratic form defined on $V$
\[
        \calQ[f,g]
        :=
        \int_0^1 P(t)f'(t)g'(t)\dd t
        -p\int_0^1 N(t)\rho(t)f(t)g(t)\dd t
\]
is lower semibounded on $H$. Indeed, for some constant \(C>0\),
\[
        \calQ[f]\ge
        \int_0^1P(t)|f'(t)|^2\dd t
        -C\|f\|_H^2,
\]
since $M\ge \min\{1,p-1\}>0$ and $N$ is bounded on $[0,1]$. Therefore, \(\calQ\) is a closed lower semibounded form; see Kato \cite[Theorem 1.11 in Chapter VI]{Kato1995}.
By the first representation theorem, \(\calQ\) defines a self-adjoint lower-semibounded operator, still denoted by \(\calL\), through
\[
        \calQ[f,g]=\langle \mathcal L f,g\rangle_H
        \qquad
        f\in \operatorname{Dom}(\calL)\subset V,\ g\in V,
\]
see Kato \cite[Theorem 2.1 in Chapter VI]{Kato1995}. It is straightforward to verify that if $u\in \operatorname{Dom}(\calL)$, then for any cut-off function $\eta\in C^2_c([0,1])$, we have $\eta u\in \operatorname{Dom}(\calL)$ and 
\begin{equation}
\calL(\eta u)
=\eta \calL(u)
-\frac{2P\eta'}{\rho}u'
-\frac{(P\eta')'}{\rho}u.
\label{eq-tool}
\end{equation}

\begin{lemma}[Essential spectrum lower bound]\label{l:embed}
Let \(\calL\) be the self-adjoint operator associated with the form \(\calQ\) on
\(H=L^2([0,1];\rho)\). Then
\[
        \sigma_{\mathrm{ess}}(\calL)\subset[\Lambda_*,\infty),
        \qquad
        \Lambda_*:=p+\frac{(n-1)^2}{4(p-1)}.
\]
Consequently, the spectrum of \(\calL\) in \((-\infty,\Lambda_*)\) consists only of isolated eigenvalues of finite multiplicity, and these eigenvalues can accumulate only at \(\Lambda_*\).
\end{lemma}

\begin{proof}
We use the IMS/Weyl-sequence argument as in \cite[Chapter 3]{CFKS87} to get a lower bound on the bottom of the essential spectrum. The point is that the full embedding
\(V\hookrightarrow H\) is not compact at \(t=0\); nevertheless the loss of compactness occurs only above the threshold \(\Lambda_*\).

\medskip
\noindent\textbf{Step 1: local compactness away from \(t=0\).}
Fix \(\varepsilon>0\). On \([\varepsilon,1/2]\), the weights \(\rho\) and \(P\) are bounded above and below by positive constants, so Rellich's theorem gives compactness into \(L^2\).

Near \(t=1\), set
\[
        y=\sqrt{1-t^2},\qquad F(y)=f(t)=f(\sqrt{1-y^2}).
\]
Since \(1-t^2=y^2\) and \(dt=-(y/t)\dd y\), as \(y\downarrow0\) one has
\[
        \rho(t)\dd t\simeq y^{2\beta+1}\dd y=y^{n-2}\dd y.
\]
Moreover,
\[
        f'(t)=F_y(y)\frac{dy}{dt}=-F_y(y)\frac{t}{y},
\]
and therefore
\[
        P(t)|f'(t)|^2\dd t
        \simeq
        y^{2\beta+2}\frac{|F_y(y)|^2}{y^2}y\dd y
        =
        y^{n-2}|F_y(y)|^2\dd y.
\]
Thus near \(t=1\) the space \(W^{1,2}([\frac{1}{2},1],\rho,P)\) is equivalent to
\[
        W^{1,2}\!\left([0,\frac{\sqrt{3}}{2}],y^{n-2},y^{n-2}\right),
\]
which corresponds to the radial \(H^1(B^{n-1}_{\sqrt{3}/2})\) space. This radial
\(H^1\) space compactly embeds into \(L^2(B^{n-1}_{\sqrt{3}/2})\). Hence bounded
subsets of \(V\) are precompact in \(H\) away from \(t=0\).

\medskip
\noindent\textbf{Step 2: the tail estimate at \(t=0\).}
Let
\[
        c:=\frac{\alpha+1}{2}=\frac{n-1}{2(p-1)}.
\]
Put
\[
        t=e^{-s},\qquad F(s):=e^{-cs}f(e^{-s}),
\]
so \(f(e^{-s})=e^{cs}F(s)\). Let \(f\in V\) vanish almost everywhere on
\((\varepsilon,1)\), and set \(S=-\log\varepsilon\). Then
\[
        \int_0^\varepsilon |f(t)|^2\rho(t)\dd t
        =
        \int_S^\infty |F(s)|^2w(s)\dd s,
        \qquad
        w(s):=(1-e^{-2s})^\beta.
\]
Also
\[
        f'(e^{-s})=-e^{(c+1)s}\bigl(F'(s)+cF(s)\bigr),
\]
and hence
\[
        \int_0^\varepsilon P(t)|f'(t)|^2\dd t
        =
        \int_S^\infty a(s)|F'(s)+cF(s)|^2\dd s,
\]
where
\[
        a(s):=\frac{(p-1)(1-e^{-2s})^{\beta+1}}{1+(p-2)e^{-2s}}.
\]
Finally,
\[
        -p\int_0^\varepsilon N(t)\rho(t)|f(t)|^2\dd t
        =
        \int_S^\infty V_0(s)|F(s)|^2w(s)\dd s,
        \qquad
        V_0(s):=-pN(e^{-s}).
\]
We now justify the Sobolev regularity and endpoint behavior needed below.
For \(S\) sufficiently large, \(a\) and \(w\) are bounded above and below
by positive constants on \([S,\infty)\).  The preceding identities and
\(f\in V\) therefore give
\[
        F\in L^2(S,\infty),
        \qquad
        F'+cF\in L^2(S,\infty),
\]
so \(F\in W^{1,2}(S,\infty)\).  Moreover, the weights defining \(V\)
are bounded above and below by positive constants locally in $(0,1)$. Thus \(f\in H^1_{loc}(0,1)\subset C^\a_{loc}(0,1)\), and the fact that \(f=0\) on \((\varepsilon,1)\)
implies \(f(\varepsilon)=0\).  Hence \(F(S)=0\).  Finally, the continuous
representative of a function in \(W^{1,2}(S,\infty)\) tends to zero at
infinity; indeed, for any $t\ge 0$
\[
      \bigl||F(s+t)|^2-|F(s)|^2\bigr|
        \le
        2\|F\|_{L^2(s,\infty)}
        \|F'\|_{L^2(s,\infty)}
        \longrightarrow0
        \qquad\text{as }s\to\infty.
\]
Consequently, $\lim_{s\rightarrow\infty}F^2(s)=0$ and
\[
        \int_S^\infty 2cF'F\dd s
        =0.
\]
Since \(a(s)\to p-1\), \(V_0(s)\to p\) as $s\rightarrow+\infty$, and \(0<w(s)\le1\), for every \(0<\eta<p-1\) we may choose \(S\) sufficiently large so that, on \([S,\infty)\),
\[
        a(s)\ge p-1-\frac{\eta}{1+c^2},
        \qquad
        V_0(s)\ge p-\frac{\eta}{1+c^2}.
\]
Therefore
\[
        \calQ[f]
        \ge
        \left(p-1-\frac{\eta}{1+c^2}\right)
        \int_S^\infty |F'+cF|^2\dd s
        +\left(p-\frac{\eta}{1+c^2}\right)
        \int_S^\infty |F|^2w(s)\dd s.
\]
For the function \(F\) obtained above,
\[
        \int_S^\infty |F'+cF|^2\dd s
        =
        \int_S^\infty |F'|^2\dd s
        +c^2\int_S^\infty |F|^2\dd s
        \ge
        c^2\int_S^\infty |F|^2w(s)\dd s,
\]
where the cross term vanishes by the endpoint behavior just proved and
the inequality follows from \(w\le1\).
Consequently, for every \(0<\eta<p-1\) there exists
\(\varepsilon>0\) such that
\[
        f=0\ \text{a.e. on }(\varepsilon,1)
        \quad\Longrightarrow\quad
        \calQ[f]\ge(\Lambda_*-\eta)\|f\|_H^2,
\]
where
\[
        \Lambda_*=p+(p-1)c^2=p+\frac{(n-1)^2}{4(p-1)}.
\]

\medskip
\noindent\textbf{Step 3: IMS localization and essential spectrum estimate.}
Choose \(\chi_0,\chi_1\in C^2([0,1])\) such that
\[
        \chi_0^2+\chi_1^2=1,\qquad
        \chi_0=1\ \text{on }(0,\varepsilon/2),\qquad
        \chi_0=0\ \text{on }(\varepsilon,1).
\]
Then \(\chi_1\) is supported away from \(t=0\). Direct expansion of the derivative term gives
\[
        \calQ[f]
        =
        \calQ[\chi_0f]+\calQ[\chi_1f]-R_\varepsilon[f],
\]
where
\[
        R_\varepsilon[f]
        :=
        \int_0^1P(t)\bigl((\chi_0')^2+(\chi_1')^2\bigr)|f(t)|^2\dd t.
\]
The coefficient in \(R_\varepsilon\) is supported in
\([\varepsilon/2,\varepsilon]\), a compact interval away from the
singular endpoint.

Assume, to the contrary, that \(\lambda<\Lambda_*\) belongs to \(\sigma_{\mathrm{ess}}(\calL)\). Choose
\[
        0<\eta<\min\left\{p-1,\frac{\Lambda_*-\lambda}{4}\right\}.
\]
By Weyl's criterion \cite[Theorem 5.13]{Borthwick2020}, there are \(f_j\in\operatorname{Dom}(\calL)\) such that
\[
        \|f_j\|_H=1,\qquad
        f_j\rightharpoonup0\text{ in }H,\qquad
        (\calL-\lambda)f_j\to0\text{ in }H.
\]
Then
\[
        \calQ[f_j]
        =
        \langle\calL f_j,f_j\rangle_H
        =
        \lambda+o(1).
\]
In particular, \(\calQ[f_j]\) is bounded, and the lower semiboundedness estimate for \(\calQ\) gives a uniform bound for
\(\int_0^1P|f_j'|^2\).

By local compactness away from \(t=0\) and weak convergence to zero,
\[
        \chi_1f_j\to0\qquad\text{strongly in }H.
\]
Thus \(R_\varepsilon[f_j]\to0\). Since \(N\) is bounded,
\[
        \calQ[\chi_1f_j]\ge -C\|\chi_1f_j\|_H^2=o(1).
\]
Since \(\chi_0f_j\in V\) and \(\chi_0f_j=0\) on
\((\varepsilon,1)\), the tail estimate applies to \(\chi_0f_j\).
Together with the IMS identity, this gives
\[
        \calQ[f_j]\ge(\Lambda_*-\eta)\|\chi_0f_j\|_H^2+o(1).
\]
Moreover,
\[
        \|\chi_0f_j\|_H^2+\|\chi_1f_j\|_H^2=1,
\]
and hence \(\|\chi_0f_j\|_H\to1\). Therefore
\[
        \lambda=\lim_{j\to\infty}\calQ[f_j]\ge\Lambda_*-\eta,
\]
which contradicts the choice of \(\eta\). Thus
\[
        \sigma_{\mathrm{ess}}(\calL)\subset[\Lambda_*,\infty).
\]
The final assertion follows from the Min-Max principle for lower semibounded self-adjoint operators; see \cite[Theorem 5.15]{Borthwick2020}.
\end{proof}
\begin{lemma}[Sturm--Liouville Theory]\label{l:sl}
The following properties hold for the isolated eigenvalues of the singular Sturm--Liouville problem
\eqref{eq:ode} below the threshold \(\Lambda_*\):
\begin{enumerate}
    \item If \(f_1\) and \(f_2\) are two eigenfunctions corresponding to the same eigenvalue
    \(\delta<\Lambda_*\), then
    \[
    f_1 = c f_2.
    \]
    In other words, each eigenspace of \(\calL\) below \(\Lambda_*\) is one-dimensional.

    \item If the eigenvalues of \(\calL\) in \((-\infty,\Lambda_*)\) are ordered increasingly as
    \[
      \delta_1<\delta_2<\delta_3<\cdots<\Lambda_*,
    \]
    and \(e_i\) is an eigenfunction associated with \(\delta_i\), then \(e_i\) has exactly \(i-1\) interior zeros in \((0,1)\).
\end{enumerate}
\end{lemma}

\begin{proof}[Proof of Lemma~\ref{l:sl}]
The argument follows the proof of Lemma~6.4 in Figalli--Neumayer
\cite{FigalliNeumayer2019}, with the endpoint at infinity replaced by
the singular endpoint \(t=0\).  We first record the regularity and
endpoint facts needed below.

Let \(f\in V\) be an eigenfunction. In the weak sense it satisfies
\[
  (P f')'=(Q-\delta\rho)f.
\]
If \(I\Subset(0,1)\), then \(P,\rho,Q\) are smooth on \(I\) and \(P\ge c_I>0\). 
The standard elliptic regularity theory implies $f\in C^\infty_{\rm loc}(0,1)$.

\begin{claim}\label{claim2}
     For any two functions \(f,g\in\operatorname{Dom}(\calL)\),
\begin{equation}\label{eq:lagrange-bracket-endpoint}
  \lim_{t\downarrow0}P(t)(f(t)g'(t)-g(t)f'(t))=0,
  \qquad
  \lim_{t\uparrow1}P(t)(f(t)g'(t)-g(t)f'(t))=0.
\end{equation}
\end{claim}
\begin{proof}[Proof of Claim~\ref{claim2}]
     For functions \(u,v\in\operatorname{Dom}(\calL)\), set
\[
  [u,v](t):=P(t)(u(t)v'(t)-v(t)u'(t)).
\]
Green's identity on \((a,b)\Subset(0,1)\) gives
\begin{equation}\label{eq:green-lagrange-identity}
  [u,v](b)-[u,v](a)
  =\int_a^b\bigl((\calL u)v-u(\calL v)\bigr)\rho\,dt.
\end{equation}
The right-hand side has finite one-sided limits as \(a\downarrow0\) and \(b\uparrow1\) whenever \(u,v\in\operatorname{Dom}(\calL)\), by Cauchy's inequality in \(L^2([0,1];\rho)\). Thus the endpoint limits of \([u,v](t)\) exist.

Now let \(f,g\in\operatorname{Dom}(\calL)\) and choose \(\chi_0\in C^\infty([0,1])\) such that \(\chi_0=1\) near \(0\) and \(\chi_0=0\) near \(1\). Since \(\calL\) is self-adjoint, applying \eqref{eq:green-lagrange-identity} to \(f\) and \(\chi_0g\) yields
\[
  0=\langle \calL f,\chi_0g\rangle_H-\langle f,\calL(\chi_0g)\rangle_H
  =[f,\chi_0g](1)-[f,\chi_0g](0).
\]
The bracket at \(1\) is zero because \(\chi_0\) vanishes near \(1\), while \([f,\chi_0g](0)=[f,g](0)\) because \(\chi_0=1\) and \(\chi_0'=0\) near \(0\). Hence \([f,g](0)=0\). Repeating the same argument with a cutoff \(\chi_1\) equal to \(1\) near \(1\) and equal to \(0\) near \(0\) gives \([f,g](1)=0\). This proves \eqref{eq:lagrange-bracket-endpoint}; see also \cite[Chapters~6 and~10]{Zettl2005} for the general Friedrichs endpoint theory.
\end{proof}

We will use the following fact: If two solutions of \eqref{eq:ode} have the same Cauchy data at some point \(t_0\in(0,1)\), then they coincide on \((0,1)\). Indeed, on every compact subinterval of \((0,1)\), the equation is a nondegenerate second-order linear ODE with smooth coefficients.

\noindent\textbf{Proof of (1)}. Let \(f_1\) and \(f_2\) be two eigenfunctions associated with the same eigenvalue \(\delta\), and set
\[
  W(t):=f_1(t)f_2'(t)-f_2(t)f_1'(t).
\]
A direct computation using
\[
  (P f_i')'=(Q-\delta\rho)f_i,
  \qquad i=1,2,
\]
gives
\[
  (P W)'=0\qquad\text{in }(0,1).
\]
Thus \(P W\) is constant. Applying Claim~\ref{claim2} to \(f_1\)
and \(f_2\) gives \(\lim_{t\downarrow0}P(t)W(t)=0\). Hence
\(P(t)W(t)=0\) for every \(t\in(0,1)\). Since \(P(t)>0\) in
\((0,1)\), we have \(W\equiv0\). Fixing any \(t_0\in(0,1)\),
there are constants \(c_1,c_2\), not both zero, such that
\[
  c_1f_1(t_0)+c_2f_2(t_0)=0,
  \qquad
  c_1f_1'(t_0)+c_2f_2'(t_0)=0.
\]
The linear combination \(c_1f_1+c_2f_2\) therefore has zero Cauchy data at \(t_0\), so uniqueness gives \(c_1f_1+c_2f_2\equiv0\). Thus the eigenspace is one-dimensional.

\noindent\textbf{Proof of (2)}. Let \(f_1\) and \(f_2\) be eigenfunctions corresponding to \(\delta_1<\delta_2\). For the same Wronskian \(W=f_1f_2'-f_2f_1'\), we now have
\begin{equation}\label{eq:sturm-wronskian-comparison}
  (PW)'=(\delta_1-\delta_2)\rho f_1f_2.
\end{equation}
We claim that between two consecutive zeros of \(f_1\) there is a zero of \(f_2\). Suppose, to the contrary, that \(f_1\) has consecutive zeros \(a<b\) and that \(f_2\) has no zero in \((a,b)\). Changing signs if necessary, assume that \(f_1>0\) and \(f_2\ge0\) on \((a,b)\). Integrating \eqref{eq:sturm-wronskian-comparison} from \(a\) to \(b\) gives
\[
  0>(\delta_1-\delta_2)\int_a^b\rho f_1f_2\,dt
  =P(b)W(b)-P(a)W(a).
\]
Since \(f_1(a)=f_1(b)=0\), \(f_1'(a)>0\), and \(f_1'(b)<0\), the right-hand side equals
\[
  -P(b)f_1'(b)f_2(b)+P(a)f_1'(a)f_2(a)\ge0,
\]
a contradiction.

The same argument also gives zeros before the first zero and after the last zero of \(f_1\). If \(a\) is the first zero of \(f_1\) and \(f_2\) has no zero in \((0,a)\), choose signs so that \(f_1,f_2\ge0\) on \((0,a)\). Since \(P W\to0\) at \(0\), integrating from \(0\) to \(a\) gives
\[
  0>(\delta_1-\delta_2)\int_0^a\rho f_1f_2\,dt
  =P(a)W(a)=-P(a)f_1'(a)f_2(a)\ge0,
\]
again a contradiction. If \(b\) is the last zero of \(f_1\) and \(f_2\) has no zero in \((b,1)\), choose signs so that \(f_1,f_2\ge0\) on \((b,1)\). Using \(P W\to0\) at \(1\),
\[
  0>(\delta_1-\delta_2)\int_b^1\rho f_1f_2\,dt
  =-P(b)W(b)=P(b)f_1'(b)f_2(b)\ge0,
\]
where \(f_1'(b)>0\) after the final sign choice. This is impossible. Therefore an eigenfunction with the larger eigenvalue has at least one zero in each interval before, between, and after the zeros of an eigenfunction with the smaller eigenvalue.

The first isolated eigenfunction below the essential spectrum has a fixed
sign in the interior. Indeed, by the variational characterization for
isolated eigenvalues below the essential spectrum (see \cite[Theorem 5.15]{Borthwick2020}) and simplicity, the first eigenfunction does not change sign in the interior.  A second isolated eigenfunction, if present, is
orthogonal to the positive first eigenfunction in \(L^2((0,1);\rho)\).
It must therefore change sign and hence has at least one interior zero.
This supplies the base case for the comparison argument above, which then
shows inductively that the \(i\)-th eigenfunction below \(\Lambda_*\) has
at least \(i-1\) interior zeros. On the other hand, the variational characterization for
isolated eigenvalues below the essential spectrum implies the Courant nodal-domain
theorem below the essential spectrum; that is, the \(i\)-th eigenfunction has at most \(i-1\)
interior zeros; see \cite{Teschl2014,CH89}. Therefore, the proof is complete.

\end{proof}

\begin{lemma}[Moser-index lemma]\label{lem:moser-index}
The operator \(\calL\), associated with the form $\calQ$ whose domain is
\(
        V,
\)
has exactly one negative eigenvalue and has no zero eigenvalue.
\end{lemma}

\begin{proof}
The proof uses two explicit eigenfunctions and the Sturm--Liouville
Lemma~\ref{l:sl}.

\medskip
\noindent\textbf{Step 1: the first eigenfunction.}

Set
\[
  L_1 f=-(P f')'-p N\rho f,\qquad
  a(\tau):=\tau^{\frac{p}{p-1}}.
\]
We claim that
\begin{equation}\label{eq:a-eigenfunction}
  L_1 a=-\frac{np}{p-1}\rho a.
\end{equation}
Indeed, for a pure power \(f(\tau)=\tau^m\), one computes
\begin{align}\label{eq:power-computation}
\frac{1}{\rho\tau^m}(P f')'
&=m(p-1)\left[
\frac{(\alpha+m+1)(1-\tau^2)-(n-1)\tau^2}{M(\tau)}
-\frac{2(p-2)\tau^2(1-\tau^2)}{M(\tau)^2}
\right].
\end{align}
Taking \(m=p/(p-1)\), using \(\alpha=(n-p)/(p-1)\), and inserting the definition of
\(N\), the identity
\[
  -\frac{1}{\rho a}(Pa')'-p N=-\frac{np}{p-1}
\]
follows by direct simplification.  This proves \eqref{eq:a-eigenfunction}.

Moreover \(a(\tau)>0\) for \(0<\tau<1\).  Since \(a\) has no
interior zero, Sturm--Liouville Lemma~\ref{l:sl} implies that \(a\)
is the first eigenfunction of \eqref{eq:ode}.  Thus the
first eigenvalue is
\[
  \d_1=-\frac{np}{p-1}<0.
\]
\medskip
\noindent\textbf{Step 2: the second eigenfunction.}
Define
\[
  \Delta:=\sqrt{n^2+p^2-2p},
\]
\[
  m_*:=\frac{\Delta-n}{p-1},
\]
\[
  c_*:=\frac{p^2-2p-n-(p-1)\Delta}{n+1},
\]
and
\[
  g(\tau):=\tau^{m_*}(1+c_*\tau^2).
\]
We claim that
\begin{equation}\label{eq:g-eigenfunction}
  L_1 g=\d_2\rho g,
\end{equation}
where
\begin{equation}\label{eq:mu-one}
  \d_2:=\frac{(n+1)(\Delta-n)+p}{p-1}.
\end{equation}

We first verify the differential equation.  Let
\[
  x:=\tau^2,
  \qquad
  f(\tau):=\tau^m(1+cx).
\]
Then
\[
  f'(\tau)=\tau^{m-1}\left[m+c(m+2)x\right].
\]
A direct computation gives
\begin{align}\label{eq:general-mc-computation}
\frac{1}{\rho\tau^m}(Pf')'
&=(p-1)\Bigg\{
\frac{\left[(\alpha+m+1)(1-x)-(n-1)x\right]
\left[m+c(m+2)x\right]}{M}
\notag\\
&\hspace{1.5cm}
+\frac{2c(m+2)x(1-x)}{M}
-
\frac{2(p-2)x(1-x)\left[m+c(m+2)x\right]}{M^2}
\Bigg\},
\end{align}
where now \(M=1+(p-2)x\).  Therefore \(L_1 f=\d\rho f\) is equivalent to the polynomial
identity
\begin{equation}\label{eq:poly-identity}
  C_0+C_1x+C_2x^2+C_3x^3\equiv 0,
\end{equation}
after multiplying by \(M^2\), where
\begin{align*}
C_0&=-(p-1)m^2-(n-1)m-\d+p,
\\
C_1&=-c m^2p+c m^2-cmn-4cmp+5cm-c\d-2cn-3cp+6c
\\
&\quad -m^2p^2+4m^2p-3m^2+2mn+2mp^2-6mp+2m-2\d p+4\d
-2np-p^2+2p,
\\
C_2&=-cm^2p^2+4cm^2p-3cm^2+2cmn-2cmp^2+10cmp-10cm
-2c\d p+4c\d
\\
&\quad -2cnp+4cn-cp^2+6cp-8c
+m^2p^2-3m^2p+2m^2+mn p^2-2mn p
\\
&\quad -3mp^2+8mp-4m-\d p^2+4\d p-4\d-2np^2+4np+2p^2-4p,
\\
C_3&=c(p-2)\left((p-1)m^2+(np+p-2)m-(p-2)\d\right).
\end{align*}
Now insert
\[
  m=m_*,\qquad c=c_*,\qquad \d=\d_2.
\]
Equivalently, use the identities
\begin{equation}\label{eq:mstar-relations}
  (p-1)^2m_*^2+2n(p-1)m_*-p(p-2)=0,
\end{equation}
\begin{equation}\label{eq:mustar-relation}
  \d_2=p-(p-1)m_*^2-(n-1)m_*,
\end{equation}
and
\begin{equation}\label{eq:cstar-relation}
  c_*=\frac{p^2-2p-n-(p-1)(n+(p-1)m_*)}{n+1}.
\end{equation}
The coefficient \(C_0\) vanishes by \eqref{eq:mustar-relation}; the coefficient \(C_3\)
vanishes by \eqref{eq:mstar-relations} and \eqref{eq:mustar-relation}; and substituting
\eqref{eq:mstar-relations}--\eqref{eq:cstar-relation} into \(C_1\) and \(C_2\) gives zero as
well.  Hence \eqref{eq:poly-identity} holds, proving \eqref{eq:g-eigenfunction}.

Next we check the sign of \(\d_2\).  Since
\[
  \d_2>0
  \quad\Longleftrightarrow\quad
  \Delta>n-\frac{p}{n+1},
\]
it is enough to square both sides.  We obtain
\begin{align*}
\Delta^2-\left(n-\frac{p}{n+1}\right)^2
&=n^2+p^2-2p-
\left(n^2-\frac{2np}{n+1}+\frac{p^2}{(n+1)^2}\right)\\
&=\frac{p\left(pn(n+2)-2(n+1)\right)}{(n+1)^2}>0,
\end{align*}
because \(p>1\) and \(n\ge 3\).  Thus \(\d_2>0\).
Finally, \(g\) has exactly one zero in \((0,1)\).  Indeed,
\[
  c_*+1
  =\frac{(p-1)(p-1-\Delta)}{n+1}.
\]
But
\[
  \Delta^2-(p-1)^2=n^2-1>0,
\]
so \(\Delta>p-1\), and hence \(c_*+1<0\).  Therefore \(c_*<-1\), and
\[
  1+c_*\tau^2
\]
has exactly one zero in \((0,1)\), namely \(\tau_*=(-1/c_*)^{1/2}\).  Since \(\tau^{m_*}>0\)
on \((0,1)\), the function \(g\) has exactly one interior zero.

\medskip
\noindent\textbf{Step 3: admissibility at the singular endpoints.}
Near \(\tau=0\), we have \(g(\tau)\sim \tau^{m_*}\).  The form-domain condition reduces to
\[
  \alpha+2m_*>-1.
\]
But
\[
  \alpha+2m_*+1=\frac{2\Delta-n-1}{p-1}>0.
\]
The last inequality follows from
\[
  \Delta^2-\frac{(n+1)^2}{4}
  =\frac{3n^2-2n-1+4p(p-2)}{4}>0
\]
for \(n\ge3\) and \(p>1\).  Thus \(g\in L^2((0,1);\rho)\) and
\(\int_0^1  g'(\tau)^2 \tau^{\a+2}(1-\tau^2)^{\b+1} \dd \tau<\infty\), since near $\tau=0$
\[
  g'(\tau)^2 \tau^{\a+2}(1-\tau^2)^{\b+1}=O\left(\tau^{\alpha+2m_*}\right).
\]

Near \(\tau=1\), the weight $\tau^{\a+2}(1-\tau^2)^{\b+1}$ contains the factor
\((1-\tau^2)^{\beta+1}\), with \(\beta=(n-3)/2\ge0\).  Hence the natural boundary term also
vanishes at \(\tau=1\).  The same endpoint checks are simpler for
\(a(\tau)=\tau^{p/(p-1)}\).  Therefore both explicit functions belong to the natural operator
form domain \(V\).

\medskip
\noindent\textbf{Step 4: Sturm oscillation and the index conclusion.}
The eigenvalue of \(a\) is negative, hence it lies below
\(\Lambda_*\). Since \(a\) has no interior zero, Lemma~\ref{l:sl}
identifies it as the first isolated eigenfunction below \(\Lambda_*\),
with eigenvalue
\[
  \d_1=-\frac{np}{p-1}<0.
\]
We next check that the explicit eigenvalue \(\d_2\) in \eqref{eq:mu-one} is also below \(\Lambda_*\). Multiplying the inequality \(\d_2<\Lambda_*\) by \(p-1\) and using
\(\Delta^2=n^2+p^2-2p\), it is equivalent to
\[
        (n+1)\Delta<\Delta^2+\frac{(n+1)^2}{4},
\]
namely
\[
        0<\left(\Delta-\frac{n+1}{2}\right)^2.
\]
The strict inequality holds because
\[
        \Delta^2-\frac{(n+1)^2}{4}
        =
        \frac{3n^2-2n-1+4p^2-8p}{4}>0
\]
for \(n\ge3\) and \(1<p<n\). Hence \(g\) is a discrete eigenfunction below the essential spectrum threshold.

The function \(g\) has exactly one interior zero. By
Lemma~\ref{l:sl}, it is therefore the second eigenfunction below
\(\Lambda_*\). Its eigenvalue is the positive number \(\d_2\) given
in \eqref{eq:mu-one}. Consequently
\[
        \d_1<0<\d_2<\Lambda_*,
        \qquad
        \sigma_{\mathrm{ess}}(\calL)\subset[\Lambda_*,\infty).
\]
Thus the operator has exactly one negative eigenvalue and no zero eigenvalue.
\end{proof}

We now prove Proposition~\ref{p:ineq}.
\begin{proof}[Proof of Proposition~\ref{p:ineq}]
As in Subsection~\ref{subsubsec:axisymmetric-mode}, we know that
$h_0(r,\tau):=r\tau-\frac{1}{p}$ satisfies
\[\int_0^1 h_0(\tau,\tau)\rho(\tau)\dd\tau
  =0.\]
Denote $b_0(\tau):=h_0(\tau,\tau)=\tau^2-\frac{1}{p}$.
Recall that $L_1 f=-(P f')'-p N\rho f$.
\begin{claim}\label{claim1}
    $L_1 b_0=-\rho$.
\end{claim}
\begin{proof}[Proof of Claim~\ref{claim1}]
Since
\[
        b_0'(t)=2t,
        \qquad
        \frac{P(t)}{\rho(t)}=\frac{(p-1)t^2(1-t^2)}{M(t)},
\]
we compute
\begin{align*}
        \frac{(P b_0')'}{\rho}
        &=\left(\frac{P}{\rho}b_0'\right)'
          +\frac{\rho'}{\rho}\frac{P}{\rho}b_0' \\
        &=2(p-1)
          \left[
             \frac{(\alpha+3)t^2-(\alpha+n+2)t^4}{M(t)}
             -\frac{2(p-2)t^4(1-t^2)}{M(t)^2}
          \right].
\end{align*}
Here we used
\[
        \frac{\rho'}{\rho}
        =\frac{\alpha}{t}-\frac{2\beta t}{1-t^2},
        \qquad 2\beta=n-3.
\]
Therefore
\begin{align*}
        \frac{(P b_0')'}{\rho}+p N b_0-1
        &=2(p-1)
          \left[
             \frac{(\alpha+3)t^2-(\alpha+n+2)t^4}{M}
             -\frac{2(p-2)t^4(1-t^2)}{M^2}
          \right] \\
        &\quad
          +p\left(t^2-\frac1p\right)
          \left[
             -\frac1M+\frac{2(n-1)t^2}{M}
             +\frac{2(p-1)t^2}{M^2}
          \right]-1.
\end{align*}
Substituting $\alpha=(n-p)/(p-1)$ and putting everything over the common denominator $M^2$
shows that the numerator is identically zero.  Hence
\[
        \frac{(P b_0')'}{\rho}+p N b_0=1,
\]
which is exactly Claim~\ref{claim1}.
\end{proof}

By Lemma~\ref{lem:moser-index}, \(\calL\) has exactly one negative
eigenvalue
\(\d_1<0\), whose normalized eigenfunction is \(e_1=Ct^{p/(p-1)}\), and
\(0\notin\sigma(\calL)\). Set
\[
        H_+:=\{e_1\}^{\perp}\subset H.
\]
Since \(e_1\) is an eigenfunction of the self-adjoint operator \(\calL\), we have \(\calL(H_+)\subset H_+\). By Lemmas~\ref{l:embed}
and~\ref{lem:moser-index}, the
restriction of \(\calL\) to \(H_+\), denoted by \(\calL_+\), is self-adjoint and
strictly positive on \(H_+\), and \(\calL_+^{-1}\) is bounded. Therefore by the
spectral theorem, the fractional powers
\[
\calL_+^{1/2},\qquad \calL_+^{-1/2}
\]
are well-defined, and \(\calL_+^{-1/2}\) is bounded.

Write the orthogonal decompositions
\[
        f=xe_1+y,\qquad
        1=ae_1+h_+,\qquad
        y,h_+\in H_+.
\]
Here \(a=\langle1,e_1\rangle_H\ne0\), since \(e_1>0\) in \((0,1)\). The constraint \(\langle f,1\rangle_H=0\) gives
\[
        ax+\langle y,h_+\rangle_H=0.
\]
Since \(\calL b_0=-1\), we have \(\calL^{-1}1=-b_0\). Together with
\(\langle b_0,1\rangle_H=0\), this gives
\[
        0=\langle 1,\calL^{-1}1\rangle_H
        =
        \frac{a^2}{\d_1}
        +\langle h_+,\calL_+^{-1}h_+\rangle_H.
\]
Thus
\begin{equation}\label{eq:positive-resolvent-identity}
        \langle h_+,\calL_+^{-1}h_+\rangle_H
        =
        -\frac{a^2}{\d_1}.
\end{equation}
By Cauchy's inequality,
\[
        |\langle y,h_+\rangle_H|^2
        =
        |\langle \calL_+^{1/2}y,\calL_+^{-1/2}h_+\rangle_H|^2
        \le
        \calQ[y]\,
        \langle h_+,\calL_+^{-1}h_+\rangle_H.
\]
Using the constraint and \eqref{eq:positive-resolvent-identity}, we obtain
\[
        a^2x^2
        \le
        -\frac{a^2}{\d_1}\calQ[y].
\]
Since \(\d_1<0\), this implies
\[
        \calQ[f]=\d_1x^2+\calQ[y]\ge0.
\]

Equality holds exactly when equality holds in the Cauchy inequality above. Hence \(y=C\calL_+^{-1}h_+\), and the constraint then gives
\[
        f=C\calL^{-1}1=-Cb_0.
\]
Thus equality holds only for multiples of \(b_0\). Conversely, by
Lemma~\ref{l:pohozaev}, \(b_0\) attains equality in
\eqref{eq:Qnonnegative}. This proves the equality statement.
\end{proof}

We now combine Lemma~\ref{l:pohozaev} and
Proposition~\ref{p:ineq} to prove
Theorem~\ref{ass:axisymmetric-gap}.

\begin{proof}[Proof of Theorem~\ref{ass:axisymmetric-gap}]
By Remark~\ref{remark2.6}, we only need to characterize the equality
case.

Recall that $h_0(r,\tau)=r\tau-\frac{1}{p}$ and $b_0(\tau)=h_0(\tau,\tau)=\tau^2-\frac{1}{p}$.

As observed before Lemma~\ref{l:pohozaev}, if $u$ attains equality in
\eqref{eq:k0-gap}, then $u$ satisfies
\eqref{eq:k0-orthogonality} and \eqref{eq:weak} with $\l=p$.

Therefore, using Lemma~\ref{l:pohozaev} and
Proposition~\ref{p:ineq}, there exists a constant $C$ such that
$u(\tau,\tau)=C b_0(\tau)$, that is, $u=C h_0$ on
$\p\OmB=\{r=\tau\}$.
Hence $w:=u-C h_0$ solves the same eigenvalue equation with $\lambda=p$ and has zero boundary
trace.  Testing with $w$ gives
\[
        \cE_0[w,w]=p\cB_0[w,w]=0.
\]
Thus $w$ is constant, and its zero trace forces $w\equiv0$.  Therefore
\[
        u(r,\tau)=C\left(r\tau-\frac1p\right).
\]
This is the desired equality characterization.
    
\end{proof}

\subsubsection{Sectorial classification}

\begin{theorem}[Sectorial classification]\label{thm:sectorial-classification}
If $\phi\in W^{1,2}(\OmB;r^{-\g})$ satisfies \(\cB[\phi,1]=0\), then
\[
  \cE[\phi,\phi]\ge p\cB[\phi,\phi],
\]
with equality if and only if
\[
  \phi\in\operatorname{span}\left\{
    y_1,\ldots,y_{n-1},y_n-\frac1p
  \right\}.
\]
\end{theorem}

\begin{proof}
Decompose as \eqref{eq:expand}
\[
  \phi(r,\tau,\theta)=\sum_{k=0}^\infty
  \sum_{j=1}^{\dim\mathcal H_k^{n-1}}
  u_{k,j}(r,\tau)Z_{k,j}(\theta).
\]
The forms \(\cE\) and \(\cB\) are diagonal in this decomposition. Hence
\[
  \cE[\phi,\phi]=|S^{n-2}|\sum_{k,j}\cE_k[u_{k,j},u_{k,j}],
  \qquad
  \cB[\phi,\phi]=|S^{n-2}|\sum_{k,j}\cB_k[u_{k,j},u_{k,j}].
\]
If \(\cB[\phi,1]=0\), then the \(k=0\) component satisfies
\[
  \cB_0[u_{0,1},1]=0,
\]
while all \(k\ge1\) components are automatically orthogonal to constants by angular orthogonality.

For the \(k=0\) component, Theorem~\ref{ass:axisymmetric-gap} gives
\[
  \cE_0[u_{0,1},u_{0,1}]
  \ge p\cB_0[u_{0,1},u_{0,1}],
\]
with equality only for
\[
  u_{0,1}=c\left(r\tau-\frac1p\right).
\]
For each \(k=1\) component, Lemma~\ref{lem:k1-ground-state} gives
\[
  \cE_1[u_{1,j},u_{1,j}]
  \ge p\cB_1[u_{1,j},u_{1,j}],
\]
with equality only for
\[
  u_{1,j}=c_j r\sqrt{1-\tau^2}.
\]
For every \(k\ge2\), Lemma~\ref{lem:higher-sectors} gives
\[
  \cE_k[u_{k,j},u_{k,j}]
  \ge p\cB_k[u_{k,j},u_{k,j}],
\]
and equality is impossible unless \(u_{k,j}\equiv0\).

Summing over all sectors proves
\[
  \cE[\phi,\phi]\ge p\cB[\phi,\phi].
\]
The equality case forces all \(k\ge2\) components to vanish, each \(k=1\) component to be a multiple of \(r\sqrt{1-\tau^2}Z_{1,j}\), and the \(k=0\) component to be a multiple of \(r\tau-1/p\). Since the degree-one harmonics on \(S^{n-2}\) are spanned by \(\theta_1,\ldots,\theta_{n-1}\), this equality space is exactly
\[
  \operatorname{span}\left\{
    y_1,\ldots,y_{n-1},y_n-\frac1p
  \right\}.
\]
The explicit computations above show that all these functions are genuine eigenfunctions with eigenvalue \(p\). Thus the min--max principle gives \(\lambda_1=p\), and the equality classification gives the full second eigenspace.
\end{proof}

We now return to the original half-space variables.
Let
\[
  F(x,t)=\frac{(x,t+1)}{|x|^2+(t+1)^2},
  \qquad
  D:=|x|^2+(t+1)^2,
\]
and
\[
  U(x,t)=D^{-\alpha/2}.
\]
The original eigenfunction corresponding to a transformed eigenfunction \(\phi\) is
\[
  \varphi(x,t)=U(x,t)\phi(F(x,t)).
\]
For \(1\le i\le n-1\),
\[
  U\,y_i\circ F
  =U\frac{x_i}{D}
  =-\frac1\alpha\partial_iU.
\]
Also,
\[
  U\left(y_n-\frac1p\right)\circ F
  =U\left(\frac{t+1}{D}-\frac1p\right).
\]
Since
\[
  v_{1,\lambda,0}(x,t)
  =\lambda^{\alpha/p}\left(|x|^2+(t+\lambda)^2\right)^{-\alpha/2},
\]
one has
\[
  \left.\partial_\lambda v_{1,\lambda,0}\right|_{\lambda=1}
  =\alpha U\left(\frac1p-\frac{t+1}{D}\right).
\]
Therefore
\[
  U\left(y_n-\frac1p\right)\circ F
  =-\frac1\alpha\left.\partial_\lambda v_{1,\lambda,0}\right|_{\lambda=1}.
\]
Consequently, the transformed classification is equivalent to 
\begin{align*}
  E_{1}
  =&\operatorname{span}\left\{
    \partial_1U,\ldots,\partial_{n-1}U,
    \left.\partial_\lambda v_{1,\lambda,0}\right|_{\lambda=1}
  \right\}\\
  =&\operatorname{span}\left\{ Z_0,Z_1,\cdots,Z_{n-1} \right\}.
\end{align*}

Together with Theorem \ref{thm:iden}, we finish the proof of Theorem~\ref{ass:spectral-nondegeneracy}.

\section{Nonlinear spectral gap near the trace-bubble manifold}
\label{sec:nonlinear-spectral-gap}

In this section we use the linear spectral gap estimate from
Corollary~\ref{cor:linear-spectral-gap} to prove a disturbed nonlinear spectral gap estimate that
will appear in the Taylor expansion of $\delta(v+\varepsilon\varphi)$.
This is the main statement needed in the proof of
Theorem~\ref{thm:main}.

We first work at the standard positive bubble \(U=U_{1,0}\).  The
corresponding statements at an arbitrary positive \(v\in\MT\) follow
from the invariances of the trace quotient.

\subsection{The exact nonlinear gradient remainder}

For \(X,Y\in\R^n\), define
\begin{equation}
 \label{eq:gradient-remainder-pointwise}
 \mathscr G_p(X,Y)
 :=
 |X+Y|^p-|X|^p-p|X|^{p-2}X\cdot Y.
\end{equation}
By the fundamental theorem of calculus,
\begin{align}
 \mathscr G_p(X,Y)
 &=
 p\int_0^1(1-s)
 \Big[
 |X+sY|^{p-2}|Y|^2
 \notag\\
 &\hspace{35mm}
 +(p-2)|X+sY|^{p-4}
 \bigl((X+sY)\cdot Y\bigr)^2
 \Big]\,ds.
 \label{eq:gradient-remainder-integral}
\end{align}
For \(h\in \dot{W}^{1,p}(\Hhalf)\), set
\begin{equation}
 \label{eq:nonlinear-gradient-form}
 \mathcal G_U[h]
 :=
 \frac{2}{p}
 \int_{\Hhalf}
 \mathscr G_p(\nabla U,\nabla h)\,dx.
\end{equation}
For a positive bubble \(v\in\MT\) we use the analogous notation
\begin{equation}
 \label{eq:general-bubble-gradient-form}
 \mathcal G_v[h]
 :=
 \frac{2}{p}
 \int_{\Hhalf}
 \mathscr G_p(\nabla v,\nabla h)\,dx.
\end{equation}
Equivalently,
\begin{align}
 \mathcal G_U[h]
 &=
 2\int_0^1(1-s)
 \int_{\Hhalf}
 \Big[
 |\nabla U+s\nabla h|^{p-2}|\nabla h|^2
 \notag\\
 &\hspace{31mm}
 +(p-2)|\nabla U+s\nabla h|^{p-4}
 \bigl((\nabla U+s\nabla h)\cdot\nabla h\bigr)^2
 \Big]\,dx\,ds.
 \label{eq:nonlinear-gradient-form-integral}
\end{align}

For every fixed smooth compactly supported \(h\),
\begin{equation}
 \label{eq:nonlinear-form-linearization}
 \lim_{\varepsilon\to0}
 \frac{\mathcal G_U[\varepsilon h]}{\varepsilon^2}
 =
 \int_{\Hhalf}
 \mathcal A_U\nabla h\cdot\nabla h\,dx.
\end{equation}

The following pointwise estimates can be proved as in \cite[Section 2]{FigalliZhang2022}, and will be used repeatedly.

\begin{lemma}[Estimate of the gradient remainder]
 \label{lem:gradient-remainder-coercivity}
 There exist constants \(c_p,C_p>0\), depending only on \(p\), such
 that
 \begin{equation}
  \label{eq:gradient-remainder-subquadratic}
  c_p
  (|X|+|Y|)^{p-2}|Y|^2
  \le
  \mathscr G_p(X,Y)
  \le
  C_p
  (|X|+|Y|)^{p-2}|Y|^2.
 \end{equation}
If \(1<p<2\), then
 \begin{equation}
  \label{eq:gradient-remainder-min-form}
  \mathscr G_p(X,Y)
  \simeq_p
  \min
  \left\{
  |Y|^p,\,
  |X|^{p-2}|Y|^2
  \right\}.
 \end{equation}
 If \(p\ge2\), then
 \begin{equation}
  \label{eq:gradient-remainder-superquadratic}
  \mathscr G_p(X,Y)
  \ge
  c_p
  \left(
  |X|^{p-2}|Y|^2+|Y|^p
  \right).
 \end{equation}
\end{lemma}

\begin{proof}
 They follow from
 \eqref{eq:gradient-remainder-integral} by separating the regions
 \(|Y|\le |X|/2\) and \(|Y|>|X|/2\).  In the first region,
 \(|X+sY|\simeq|X|\), while in the second region the homogeneity of
 \(\mathscr G_p\) gives a lower bound proportional to \(|Y|^p\).
\end{proof}

The boundary counterpart of the preceding gradient estimates is the
following scalar lemma.

\begin{lemma}[Boundary Taylor remainder]
 \label{lem:boundary-taylor-remainder}
 For \(a>0\), \(b\in\R\), and \(q>1\), set
 \begin{equation}
  \label{eq:boundary-remainder-pointwise}
  \mathscr B_q(a,b)
  :=
  |a+b|^q-a^q-qa^{q-1}b.
 \end{equation}
 If \(1<q\le2\), then, for every \(\kappa>0\), there exists
  \(C_\kappa=C(q,\kappa)\ge1\) such that
 \begin{equation}
  \label{eq:boundary-remainder-q-less-two}
  \mathscr B_q(a,b)
  \le
  \left(
  \frac{q(q-1)}{2}+\kappa
  \right)
  \frac{(a+C_\kappa|b|)^q}{a^2+b^2}\,b^2.
 \end{equation}
 If \(q>2\), then, for every \(\kappa>0\), there exists
  \(C_\kappa=C(q,\kappa)>0\) such that
 \begin{equation}
  \label{eq:boundary-remainder-q-greater-two}
  \mathscr B_q(a,b)
  \le
  \left(
  \frac{q(q-1)}{2}+\kappa
  \right)a^{q-2}b^2
  +
   C_\kappa|b|^q.
 \end{equation}
\end{lemma}

\begin{proof}
 These are precisely the scalar estimates of
 \cite[Lemma~2.4]{FigalliZhang2022}, with the Sobolev exponent there
 replaced by \(q\). 
\end{proof}

We next extend the weighted notion of orthogonality as in
\cite[Remark~3.7]{FigalliZhang2022}.  If \(v\in\MT^+\) and
\(Z\in T_v\MT\), then \(|Z|\le C_vv\) on \(\bdH\), and hence
\[
 v^{p_*-2}Z\in L^{p_*'}(\bdH).
\]
Therefore, for every \(\psi\in L^{p_*}(\bdH)\), we define
\(\psi\perp T_v\MT\) by the duality conditions
\begin{equation}
 \label{eq:orthogonality-general-duality}
 \int_{\bdH}v^{p_*-2}\psi Z\,dy=0
 \qquad
 \forall Z\in T_v\MT.
\end{equation}
The pairing is well defined by H\"older's inequality.  When
\(\psi\in L^2(\bdH,v^{p_*-2}\,dy)\), this agrees with the usual
weighted Hilbert-space orthogonality.  In particular,
the trace inequality gives this notion of orthogonality for every
\(h\in\dot{W}^{1,p}(\Hhalf)\), even when
\(h\notin L^2(\bdH,v^{p_*-2}\,dy)\).

At the standard bubble \(U\), condition
\eqref{eq:orthogonality-general-duality} becomes
\begin{equation}
 \label{eq:nonlinear-gap-orthogonality}
 \int_{\bdH}
 U^{p_*-2}h\,Z\,dy=0
 \qquad
 \forall Z\in T_U\MT.
\end{equation}
In view of \eqref{eq:tangent-space-standard}, this is equivalent to
\begin{equation}
 \label{eq:orthogonality-expanded}
 \int_{\bdH}
 U^{p_*-2}h\,U\,dy=0,
 \qquad
 \int_{\bdH}
 U^{p_*-2}h\,Z_j\,dy=0,
 \quad 0\le j\le n-1.
\end{equation}

\subsection{A compactness lemma for normalized perturbations}

We adopt the idea in \cite[Lemma 3.4]{FigalliZhang2022} to prove the following compactness lemma, which will be used in the
contradiction argument of Proposition \ref{prop:nonlinear-spectral-gap}.

\begin{lemma}[Compactness of normalized perturbations]
 \label{lem:normalized-perturbation-compactness}
 Assume \(1<p<2\).  Let \(h_k\in \dot{W}^{1,p}(\Hhalf)\) satisfy
 \begin{equation}
  \label{eq:hk-small-gradient}
  \|\nabla h_k\|_{L^p(\Hhalf)}\longrightarrow0
 \end{equation}
 and
 \begin{equation}
  \label{eq:hk-orthogonal}
  \int_{\bdH}
  U^{p_*-2}h_k\,Z\,dy=0
  \qquad
  \forall Z\in T_U\MT.
 \end{equation}
 Define
 \begin{equation}
  \label{eq:epsilon-k-definition}
  \varepsilon_k^2
  :=
  \int_{\Hhalf}
  \bigl(|\nabla U|+|\nabla h_k|\bigr)^{p-2}
  |\nabla h_k|^2\,dx
 \end{equation}
 and assume \(\varepsilon_k>0\).  Set
 \[
  \phi_k:=\frac{h_k}{\varepsilon_k}.
 \]
 Then we have $\varepsilon_k\rightarrow0$ and, after passing to a subsequence, there exists
 \(\phi\in\dot{W}^{1,p}(\Hhalf)\cap L^2(\bdH,U^{p_*-2}\,dy)\) such that
 \begin{align}
  \phi_k
  \rightharpoonup\phi
  \qquad\text{weakly in }\dot{W}^{1,p}(\Hhalf),
  \label{eq:phi-k-weak}
 \end{align}
 and, after taking a further subsequence, \(\phi_k\to\phi\) almost
 everywhere in \(\Hhalf\) and on \(\bdH\).  Moreover,
 \begin{equation}
  \label{eq:limit-orthogonality}
  \int_{\bdH}
  U^{p_*-2}\phi\,Z\,dy=0
  \qquad
  \forall Z\in T_U\MT.
 \end{equation}

 If \(p_*>2\), then
 \begin{equation}
  \label{eq:phi-k-strong-trace-weighted}
  \phi_k\longrightarrow\phi
  \qquad\text{strongly in }
  L^2(\bdH,U^{p_*-2}\,dy).
 \end{equation}

 If \(p_*\le2\), then, for every fixed \(C_0\ge1\),
 \begin{equation}
  \label{eq:nonlinear-boundary-limsup}
  \lim_{k\to\infty}
  \frac{1}{\varepsilon_k^2}
  \int_{\bdH}
  \frac{(U+C_0|h_k|)^{p_*}}{U^2+|h_k|^2}
  |h_k|^2\,dy
  =
  \int_{\bdH}
  U^{p_*-2}|\phi|^2\,dy.
 \end{equation}
\end{lemma}

\begin{proof}
 For the energy estimates and the nonlinear boundary convergence, we
 may replace \(h_k\) by \(|h_k|\), since
 \(|\nabla|h_k||=|\nabla h_k|\) almost everywhere and all the relevant
 boundary integrands depend only on \(|h_k|\).  Thus we may assume
 \(h_k\ge0\), and hence \(\phi_k\ge0\), in those parts of the proof.
 The passage to the orthogonality condition is carried out for the
 original signed sequence.
 By the definition of \(\varepsilon_k\),
 \begin{equation}
  \label{eq:normalized-mixed-energy}
  \int_{\Hhalf}
  \bigl(|\nabla U|+\varepsilon_k|\nabla\phi_k|\bigr)^{p-2}
  |\nabla\phi_k|^2\,dx
  =1.
 \end{equation}

 First, \(\varepsilon_k\to0\).  Indeed, since \(p<2\),
 \[
  \varepsilon_k^2
  \le
  \int_{\Hhalf}|\nabla h_k|^p\,dx
  \longrightarrow0.
 \]

 We next prove that \((\phi_k)\) is bounded in
 \(\dot{W}^{1,p}(\Hhalf)\).  H\"older's inequality and
 \eqref{eq:normalized-mixed-energy} give
 \begin{align}
  \int_{\Hhalf}|\nabla\phi_k|^p\,dx
  &\le
  \left(
  \int_{\Hhalf}
  \bigl(|\nabla U|+\varepsilon_k|\nabla\phi_k|\bigr)^{p-2}
  |\nabla\phi_k|^2\,dx
  \right)^{p/2}
  \notag\\
  &\quad\times
  \left(
  \int_{\Hhalf}
  \bigl(|\nabla U|+\varepsilon_k|\nabla\phi_k|\bigr)^p\,dx
  \right)^{1-p/2}
  \notag\\
  &\le
  C
  \left(
  1+\varepsilon_k^p
  \int_{\Hhalf}|\nabla\phi_k|^p\,dx
  \right)^{1-p/2}.
  \label{eq:phi-k-W1p-absorption}
 \end{align}
 Since \(1-p/2<1\) and \(\varepsilon_k\le1\) for large \(k\),
 an elementary absorption argument yields
 \begin{equation}
  \label{eq:phi-k-W1p-bound}
  \sup_k\|\nabla\phi_k\|_{L^p(\Hhalf)}<\infty.
 \end{equation}
 Reflexivity gives \eqref{eq:phi-k-weak} after passing to a
 subsequence.  The trace operator is continuous into
 \(L^{p_*}(\bdH)\),
 and the local Rellich theorem and local compact trace embedding give,
 after a further subsequence, almost-everywhere convergence in
 \(\Hhalf\) and on \(\bdH\).

 For every \(Z\in T_U\MT\), the explicit tangent modes satisfy
 \(|Z|\le C U\) on \(\bdH\).  Hence
 \(U^{p_*-2}Z\in L^{p_*'}(\bdH)\).  Dividing
 \eqref{eq:hk-orthogonal} by \(\varepsilon_k\) and passing to the
 weak trace limit gives \eqref{eq:limit-orthogonality}.

 If \(p_*>2\), then the local compact trace convergence and H\"older's
 inequality on the tails imply strong convergence in the weighted
 \(L^2\)-space.  Indeed, for \(R>1\),
 \begin{align*}
  \int_{\{\rho>R\}\cap\bdH}
  U^{p_*-2}|\phi_k-\phi|^2\,dy
  &\le
  \left(
  \int_{\{\rho>R\}\cap\bdH}U^{p_*}\,dy
  \right)^{(p_*-2)/p_*}
  \|\phi_k-\phi\|_{L^{p_*}(\bdH)}^2,
 \end{align*}
 where \(\rho(x):=|x+e_n|\).  The first factor tends to zero as
 \(R\to\infty\), uniformly in \(k\), while on \(\{\rho\le R\}\cap\bdH\)
 the convergence is strong in \(L^2\).  This proves
 \eqref{eq:phi-k-strong-trace-weighted}.

 It remains to prove the nonlinear boundary convergence when
 \(p_*\le2\).  Set
 \[
  \alpha:=\frac{n-p}{p-1},
  \qquad
  \beta:=\alpha(2-p_*)\ge0,
  \qquad
  \rho(x):=|x+e_n|.
 \]
 Then
 \[
  U\asymp\rho^{-\alpha},
  \qquad
  |\nabla U|\asymp\rho^{-\alpha-1},
  \qquad
  U^{p_*-2}\asymp\rho^\beta,
 \]
 and the identities
 \begin{equation}
  \label{eq:trace-alpha-identities}
  \alpha(p_*-p)=p,
  \qquad
  \beta+p-1-\alpha(2-p)=-1
 \end{equation}
 will be used below.

We shall use the following Hardy--trace type estimate: for any
\(\xi\in\dot{W}^{1,p}(\Hhalf,\rho^{\beta+p-1})\),
 \begin{equation}
  \label{eq:weighted-hardy-trace}
  \int_{\Hhalf}\rho^{\beta-1}|\xi|^p\,dx
  +
  \int_{\bdH}\rho^\beta|\xi|^p\,dy
  \le
  C
  \int_{\Hhalf}\rho^{\beta+p-1}|\nabla\xi|^p\,dx.
 \end{equation}
 To prove it, first apply integration by parts to
 \(X(x):=\rho^{\beta-1}(x+e_n)\).  Since
 \(\operatorname{div}X=(n+\beta-1)\rho^{\beta-1}\), H\"older's
 inequality gives the interior Hardy estimate.  Then integrate
 \(-\partial_t(\rho^\beta|\xi|^p)\) in the vertical direction and use
 the interior estimate once more to obtain the boundary term.

 \medskip
 \noindent\emph{Step 1: small amplitudes.}
 We first prove the convergence under the additional assumptions
 \begin{equation}
  \label{eq:small-amplitude-regime}
  0\le\varepsilon_k\phi_k\le\zeta U
  \quad\text{in }\Hhalf,
  \qquad
  \sup_k
  \int_{\Hhalf}
  \bigl(|\nabla U|+\varepsilon_k|\nabla\phi_k|\bigr)^{p-2}
  |\nabla\phi_k|^2\,dx
  <\infty,
  \end{equation}
 where \(\zeta>0\) will be chosen sufficiently small.  Since
 \(0\le\varepsilon_k\phi_k/U\le\zeta\),
 \begin{equation}
  \label{eq:small-amplitude-comparison}
  \left[
  \left(1+\frac{\varepsilon_k\phi_k}{U}\right)^{(p-2)/p}
  \phi_k^{2/p}
  \right]^p
  \asymp\phi_k^2.
 \end{equation}
 A direct differentiation gives
 \begin{equation}
  \label{eq:small-amplitude-gradient-preliminary}
  \left|
  \nabla\left[
  \left(1+\frac{\varepsilon_k\phi_k}{U}\right)^{(p-2)/p}
  \phi_k^{2/p}
  \right]
  \right|^p
  \le
  C\phi_k^2
  \left(
  \frac{\varepsilon_k|\nabla\phi_k|}{U}
  +
  \frac{\varepsilon_k\phi_k}{U}\frac{|\nabla U|}{U}
  \right)^p
  +
  C\phi_k^{2-p}|\nabla\phi_k|^p.
 \end{equation}
 Using \(|\nabla U|/U\asymp\rho^{-1}\), we get
 \begin{align}
  \rho^{\beta+p-1}
  \left|
  \nabla\left[
  \left(1+\frac{\varepsilon_k\phi_k}{U}\right)^{(p-2)/p}
  \phi_k^{2/p}
  \right]
  \right|^p
  &\le
  C\zeta^p\rho^{\beta-1}\phi_k^2
  +
  C\rho^{\beta+p-1}
  \phi_k^{2-p}|\nabla\phi_k|^p.
  \label{eq:small-amplitude-gradient-reduced}
 \end{align}

 \begin{claim}
 For every \(\delta>0\), choosing
 \(\zeta=\zeta(\delta)>0\) sufficiently small gives
 \begin{align}
  \rho^{\beta+p-1}\phi_k^{2-p}|\nabla\phi_k|^p
  &\le
  \delta\rho^{\beta-1}\phi_k^2
  \notag\\
  &\quad+
  C_\delta\rho^{-1}
  \bigl(|\nabla U|+\varepsilon_k|\nabla\phi_k|\bigr)^{p-2}
  |\nabla\phi_k|^2.
  \label{eq:key-weighted-numerical-estimate}
 \end{align}
 \end{claim}

 \begin{proof}
 On \(\{\varepsilon_k|\nabla\phi_k|<|\nabla U|\}\), we have
 \[
  |\nabla U|+\varepsilon_k|\nabla\phi_k|
  \asymp |\nabla U|.
 \]
 Moreover, \(|\nabla U|\asymp\rho^{-\alpha-1}\) and
 \eqref{eq:trace-alpha-identities} imply
 \begin{align*}
  &\rho^{\beta+p-1}\phi_k^{2-p}|\nabla\phi_k|^p\\
  &\quad\le C
  \bigl(\rho^{\beta-1}\phi_k^2\bigr)^{1-p/2}
  \bigl(\rho^{-1}|\nabla U|^{p-2}
  |\nabla\phi_k|^2\bigr)^{p/2}.
 \end{align*}
 Indeed, equality of the powers of \(\rho\) is precisely equivalent to
 \(\beta+p=\alpha(2-p)\), which is the second identity in
 \eqref{eq:trace-alpha-identities}.  Young's inequality therefore gives
 \begin{align*}
  \rho^{\beta+p-1}\phi_k^{2-p}|\nabla\phi_k|^p
  &\le
  \delta\rho^{\beta-1}\phi_k^2\\
  &\quad+
  C_\delta\rho^{-1}
  \bigl(|\nabla U|+\varepsilon_k|\nabla\phi_k|\bigr)^{p-2}
  |\nabla\phi_k|^2.
 \end{align*}

 On \(\{\varepsilon_k|\nabla\phi_k|\ge|\nabla U|\}\), we instead have
 \[
  \bigl(|\nabla U|+\varepsilon_k|\nabla\phi_k|\bigr)^{p-2}
  |\nabla\phi_k|^2
  \asymp
  \varepsilon_k^{p-2}|\nabla\phi_k|^p.
 \]
 Since \(\varepsilon_k\phi_k\le\zeta U\), it follows that
 \begin{align*}
  &\rho^{\beta+p-1}\phi_k^{2-p}|\nabla\phi_k|^p\\
  &\quad\le
  C\zeta^{2-p}
  \rho^{\beta+p-1}U^{2-p}\varepsilon_k^{p-2}
  |\nabla\phi_k|^p\\
  &\quad\le
  C\zeta^{2-p}\rho^{-1}
  \bigl(|\nabla U|+\varepsilon_k|\nabla\phi_k|\bigr)^{p-2}
  |\nabla\phi_k|^2,
 \end{align*}
 where the last inequality follows from
 \(U\asymp\rho^{-\alpha}\) and
 \(\beta+p=\alpha(2-p)\).  Combining the two regions proves
 \eqref{eq:key-weighted-numerical-estimate}.
 \end{proof}

 Combining \eqref{eq:small-amplitude-gradient-reduced} and
 \eqref{eq:key-weighted-numerical-estimate}, and choosing
 \(\delta\) and then \(\zeta\) sufficiently small, we obtain
 \begin{equation}
  \label{eq:small-amplitude-gradient-final}
  \rho^{\beta+p-1}
  \left|
  \nabla\left[
  \left(1+\frac{\varepsilon_k\phi_k}{U}\right)^{(p-2)/p}
  \phi_k^{2/p}
  \right]
  \right|^p
  \le
  c_{\mathrm{abs}}\rho^{\beta-1}\phi_k^2
  +
  C\rho^{-1}
  \bigl(|\nabla U|+\varepsilon_k|\nabla\phi_k|\bigr)^{p-2}
  |\nabla\phi_k|^2,
 \end{equation}
 where \(c_{\mathrm{abs}}>0\) is chosen small enough to be absorbed below.  Applying
 \eqref{eq:weighted-hardy-trace} to
 \[
  \left(1+\frac{\varepsilon_k\phi_k}{U}\right)^{(p-2)/p}
  \phi_k^{2/p},
 \]
 and using \eqref{eq:small-amplitude-comparison} and
 \eqref{eq:small-amplitude-gradient-final}, yields
 \begin{equation}
  \label{eq:small-regime-weighted-bound}
  \int_{\Hhalf}\rho^{\beta-1}\phi_k^2\,dx
  +
  \int_{\bdH}\rho^\beta\phi_k^2\,dy
  \le
  C
  \int_{\Hhalf}
  \bigl(|\nabla U|+\varepsilon_k|\nabla\phi_k|\bigr)^{p-2}
  |\nabla\phi_k|^2\,dx
  \le C.
 \end{equation}

 We next prove the uniform tail bound
 \begin{equation}
  \label{eq:small-regime-tail-vanishing}
  \lim_{R\to\infty}\sup_k
  \left[
  \int_{\{\rho>R\}\cap\Hhalf}\rho^{\beta-1}\phi_k^2\,dx
  +
  \int_{\{\rho>R\}\cap\bdH}\rho^\beta\phi_k^2\,dy
  \right]
  =0.
 \end{equation}
 For \(R>e\), choose a logarithmic cutoff \(\eta_R\) such that
 \[
  0\le\eta_R\le1,
  \qquad
  \eta_R=0\ \text{on }\{\rho\le R\},
  \qquad
  \eta_R=1\ \text{on }\{\rho\ge R^2\},
 \]
 and
 \[
  |\nabla\eta_R|
  \le\frac{C}{\rho\log R}
  \mathbf 1_{\{R<\rho<R^2\}}.
 \]
 Apply \eqref{eq:weighted-hardy-trace}, first with an additional outer
 compact cutoff and then by approximation, to
 \[
  \eta_R
  \left(1+\frac{\varepsilon_k\phi_k}{U}\right)^{(p-2)/p}
  \phi_k^{2/p}.
 \]
 By \eqref{eq:small-amplitude-comparison}, the left-hand side controls
 \[
  \int_{\{\rho>R^2\}\cap\Hhalf}
  \rho^{\beta-1}\phi_k^2\,dx
  +
  \int_{\{\rho>R^2\}\cap\bdH}
  \rho^\beta\phi_k^2\,dy.
 \]
 Using \eqref{eq:small-amplitude-gradient-final}, the product rule, and
 choosing \(c_*\) small enough to absorb the resulting interior term,
 we obtain
 \begin{align*}
  &\int_{\{\rho>R^2\}\cap\Hhalf}
  \rho^{\beta-1}\phi_k^2\,dx
  +
  \int_{\{\rho>R^2\}\cap\bdH}
  \rho^\beta\phi_k^2\,dy\\
  &\quad\le
  C\int_{\{\rho>R\}\cap\Hhalf}
  \rho^{-1}
  \bigl(|\nabla U|+\varepsilon_k|\nabla\phi_k|\bigr)^{p-2}
  |\nabla\phi_k|^2\,dx\\
  &\qquad+
  \frac{C}{(\log R)^p}
  \int_{\{R<\rho<R^2\}\cap\Hhalf}
  \rho^{\beta-1}\phi_k^2\,dx.
 \end{align*}
 The first term is at most \(C/R\) by
 \eqref{eq:small-amplitude-regime}, while the second is at most
 \(C/(\log R)^p\) by \eqref{eq:small-regime-weighted-bound}.
 Therefore
 \[
  \sup_k\left[
  \int_{\{\rho>R^2\}\cap\Hhalf}\rho^{\beta-1}\phi_k^2\,dx
  +
  \int_{\{\rho>R^2\}\cap\bdH}\rho^\beta\phi_k^2\,dy
  \right]
  \le
  \frac{C}{R}+\frac{C}{(\log R)^p},
 \]
 which proves \eqref{eq:small-regime-tail-vanishing}.

 On bounded boundary sets,
 \(
 ((1+\varepsilon_k\phi_k/U)^{(p-2)/p}\phi_k^{2/p})
 \)
 is bounded in \(W^{1,p}\), so its traces are bounded in
 \(L^{p_*}\).  Since \(p_*/p>1\) and
 \eqref{eq:small-amplitude-comparison} holds, the sequence
 \((\phi_k^2)\) is locally
 uniformly integrable on \(\bdH\).  Together with
 \(\varepsilon_k\to0\) and the almost-everywhere convergence, this
 gives local \(L^1\)-convergence of
 \[
  \frac{(U+C_0\varepsilon_k\phi_k)^{p_*}}
       {U^2+\varepsilon_k^2\phi_k^2}\phi_k^2
  \quad\text{to}\quad
  U^{p_*-2}\phi^2.
 \]
 Combining this local convergence with
 \eqref{eq:small-regime-tail-vanishing}, we conclude that
 \begin{equation}
  \label{eq:small-regime-limit}
  \lim_{k\to\infty}
  \int_{\bdH}
  \frac{(U+C_0\varepsilon_k\phi_k)^{p_*}}
       {U^2+\varepsilon_k^2\phi_k^2}
  \phi_k^2\,dy
  =
  \int_{\bdH}U^{p_*-2}\phi^2\,dy.
 \end{equation}

 \medskip
 \noindent\emph{Step 2: removal of the truncation.}
 Let \(\zeta>0\) be fixed as in Step 1 and
 define
 \[
  \phi_{k,1}:=
  \min\left\{\phi_k,\frac{\zeta U}{\varepsilon_k}\right\},
  \qquad
  \phi_{k,2}:=\phi_k-\phi_{k,1}.
 \]
 For \(t\ge0\), set
 \[
  F_t(z):=(t+|z|)^{p-2}|z|^2.
 \]
 We use the standard convexity splitting from Figalli--Zhang:
 \begin{equation}
  \label{eq:truncated-energy-splitting}
  \begin{aligned}
  &\int_{\Hhalf}
  \bigl(|\nabla U|+\varepsilon_k|\nabla\phi_{k,1}|\bigr)^{p-2}
  |\nabla\phi_{k,1}|^2\,dx\\
  &\quad+
  \int_{\Hhalf}
  \bigl(|\nabla U|+\varepsilon_k|\nabla\phi_{k,2}|\bigr)^{p-2}
  |\nabla\phi_{k,2}|^2\,dx\\
  &\le
  C\int_{\Hhalf}
  \bigl(|\nabla U|+\varepsilon_k|\nabla\phi_k|\bigr)^{p-2}
  |\nabla\phi_k|^2\,dx
  \le C.
  \end{aligned}
 \end{equation}
 Indeed, \(F_t\) is convex for every \(t\ge0\), and
 testing the \(p\)-harmonic equation for \(U\) with
 \((\varepsilon_k\phi_k-\zeta U)_+\) gives
 \[
  \int_{\{\varepsilon_k\phi_k>\zeta U\}}
  DF_{|\nabla U|}(\zeta\nabla U)
  \cdot\nabla(\varepsilon_k\phi_k-\zeta U)\,dx
  \ge0.
 \]
 The corresponding convexity inequality, integrated over
 \(\{\varepsilon_k\phi_k>\zeta U\}\), then yields
 \begin{equation}
  \label{eq:U-energy-on-large-set}
  \frac{c}{\varepsilon_k^2}
  \int_{\{\varepsilon_k\phi_k>\zeta U\}}|\nabla U|^p\,dx
  \le
  \int_{\{\varepsilon_k\phi_k>\zeta U\}}
  \bigl(|\nabla U|+\varepsilon_k|\nabla\phi_k|\bigr)^{p-2}
  |\nabla\phi_k|^2\,dx
  \le C,
 \end{equation}
 and the doubling estimate for \(F_t\) gives
 \eqref{eq:truncated-energy-splitting}.

 We now prove that the large-amplitude part is negligible:
 \begin{equation}
  \label{eq:large-part-W1p-small}
  \int_{\Hhalf}|\nabla\phi_{k,2}|^p\,dx
  \le
  C\varepsilon_k^{2-p}.
 \end{equation}
 Split
 \[
  E_k:=\{\varepsilon_k|\nabla\phi_{k,2}|\ge|\nabla U|\},
  \qquad
  F_k:=\{\varepsilon_k|\nabla\phi_{k,2}|<|\nabla U|\}.
 \]
 On \(F_k\), H\"older's inequality, \eqref{eq:U-energy-on-large-set}, and
 \eqref{eq:truncated-energy-splitting} give
 \[
  \int_{F_k}|\nabla\phi_{k,2}|^p\,dx
  \le
  C
  \left(
  \int_{F_k}|\nabla U|^{p-2}|\nabla\phi_{k,2}|^2\,dx
  \right)^{p/2}
  \left(
  \int_{\{\varepsilon_k\phi_k>\zeta U\}}|\nabla U|^p\,dx
  \right)^{1-p/2}
  \le
  C\varepsilon_k^{2-p}.
 \]
 On \(E_k\),
 \[
  |\nabla U|+\varepsilon_k|\nabla\phi_{k,2}|
  \asymp
  \varepsilon_k|\nabla\phi_{k,2}|,
 \]
 and hence
 \[
  \int_{E_k}|\nabla\phi_{k,2}|^p\,dx
  \le
  C\varepsilon_k^{2-p}
  \int_{E_k}
  \bigl(|\nabla U|+\varepsilon_k|\nabla\phi_{k,2}|\bigr)^{p-2}
  |\nabla\phi_{k,2}|^2\,dx
  \le
  C\varepsilon_k^{2-p}.
 \]
 This proves \eqref{eq:large-part-W1p-small}.  The Sobolev trace
 inequality therefore yields
 \begin{align}
  \varepsilon_k^{p_*-2}
  \int_{\bdH}|\phi_{k,2}|^{p_*}\,dy
  &\le
  C\varepsilon_k^{p_*-2}
  \left(
  \int_{\Hhalf}|\nabla\phi_{k,2}|^p\,dx
  \right)^{p_*/p}
  \notag\\
  &\le
  C\varepsilon_k^{\,2(p_*-p)/p}
  \longrightarrow0.
  \label{eq:large-part-boundary-small}
 \end{align}
 In particular, \(\phi_{k,2}\to0\) strongly in
 \(\dot{W}^{1,p}(\Hhalf)\) and in \(L^{p_*}(\bdH)\), so
 \(\phi_{k,1}\rightharpoonup\phi\) in \(\dot{W}^{1,p}\) and
  \(\phi_{k,1}\to\phi\) almost everywhere on \(\bdH\).  Since
 \(\varepsilon_k\phi_{k,1}\le\zeta U\) and the first integral on the
 left-hand side of \eqref{eq:truncated-energy-splitting} is bounded,
 Step 1 gives
 \begin{equation}
  \label{eq:truncated-boundary-limit}
  \lim_{k\to\infty}
  \int_{\bdH}
  \frac{(U+C_0\varepsilon_k\phi_{k,1})^{p_*}}
       {U^2+\varepsilon_k^2\phi_{k,1}^2}
  \phi_{k,1}^2\,dy
  =
  \int_{\bdH}U^{p_*-2}|\phi|^2\,dy.
 \end{equation}

 It remains only to compare \(\phi_k\) with \(\phi_{k,1}\).  On
 \(A_k:=\{\varepsilon_k\phi_k>\zeta U\}\cap\bdH\), one has
 \(\phi_{k,1}=\zeta U/\varepsilon_k\).  A one-dimensional
 differentiation in the variable \(s\) gives
 \begin{align*}
 &\left|
  \frac{(U+C_0\varepsilon_k\phi_k)^{p_*}}
       {U^2+\varepsilon_k^2\phi_k^2}\phi_k^2
  -
  \frac{(U+C_0\varepsilon_k\phi_{k,1})^{p_*}}
       {U^2+\varepsilon_k^2\phi_{k,1}^2}\phi_{k,1}^2
 \right|
 \\
 &\qquad\le
  C\varepsilon_k^{p_*-2}
  \left(
  \phi_{k,1}^{p_*-1}\phi_{k,2}
  +
  \phi_{k,2}^{p_*}
 \right).
 \end{align*}
 On the boundary set
 \(A_k\),
we have
 \[
  \varepsilon_k^{p_*-2}
  \int_{A_k}\phi_{k,1}^{p_*}\,dy
  \le C\int_{A_k}
  \frac{(U+C_0\varepsilon_k\phi_{k,1})^{p_*}}
       {U^2+\varepsilon_k^2\phi_{k,1}^2}
  \phi_{k,1}^2\,dy\le C.
 \]
 Since \(\phi_{k,2}=0\) outside \(A_k\), H\"older's inequality and
 \eqref{eq:large-part-boundary-small} imply
 \[
  \varepsilon_k^{p_*-2}
  \int_{A_k}\phi_{k,1}^{p_*-1}\phi_{k,2}\,dy
  \longrightarrow0,
  \qquad
  \varepsilon_k^{p_*-2}
  \int_{\bdH}\phi_{k,2}^{p_*}\,dy
  \longrightarrow0.
 \]
 Combining this with \eqref{eq:truncated-boundary-limit}, and using
 \(h_k=\varepsilon_k\phi_k\), gives
 \[
  \lim_{k\to\infty}
  \frac{1}{\varepsilon_k^2}
  \int_{\bdH}
  \frac{(U+C_0|h_k|)^{p_*}}{U^2+|h_k|^2}
  |h_k|^2\,dy
  =
  \int_{\bdH}U^{p_*-2}|\phi|^2\,dy.
 \]
 This is \eqref{eq:nonlinear-boundary-limsup}.
\end{proof}

\subsection{Disturbed spectral gap}
As in \cite[Proposition 3.8]{FigalliZhang2022}, we combine the blow-up statement and the spectral gap estimate to prove the following disturbed spectral gap.
\begin{proposition}[Disturbed spectral gap]
 \label{prop:nonlinear-spectral-gap}
 Let \(\lambda_{\mathrm T}>0\) be the constant from
 Corollary~\ref{cor:linear-spectral-gap}.

 If \(1<p<2\) and \(p_*\le2\), then, for every fixed
 \(C_0\ge1\), there exists \(\delta_0=\delta_0(n,p,C_0)>0\)
 such that every \(h\in \dot{W}^{1,p}(\Hhalf)\) satisfying
 \begin{align}
  \|\nabla h\|_{L^p(\Hhalf)}
  &\le\delta_0,
  \label{eq:nonlinear-gap-smallness-q-less-two}\\
  \int_{\bdH}
  U^{p_*-2}h\,Z\,dy
  &=0
  \qquad
  \forall Z\in T_U\MT
  \label{eq:nonlinear-gap-orthogonality-q-less-two}
 \end{align}
 obeys
 \begin{align}
  \mathcal G_U[h]
  &\ge
  \left[
  (p_*-1)\ST^p+\lambda_{\mathrm T}
  \right]
  \|U\|_{L^{p_*}(\bdH)}^{\,p-p_*}
  \int_{\bdH}
  \frac{(U+C_0|h|)^{p_*}}{U^2+|h|^2}|h|^2\,dy.
  \label{eq:nonlinear-spectral-gap-q-less-two}
 \end{align}

 If \(p_*>2\), then there exists
 \(\delta_1=\delta_1(n,p)>0\) such that every
 \(h\in \dot{W}^{1,p}(\Hhalf)\) satisfying
 \begin{align}
  \|\nabla h\|_{L^p(\Hhalf)}
  &\le\delta_1,
  \label{eq:q-greater-two-smallness}\\
  \int_{\bdH}
  U^{p_*-2}h\,Z\,dy
  &=0
  \qquad
  \forall Z\in T_U\MT
  \label{eq:q-greater-two-orthogonality}
 \end{align}
 obeys
 \begin{align}
  \mathcal G_U[h]
  &\ge
  \left[
  (p_*-1)\ST^p+\lambda_{\mathrm T}
  \right]
  \|U\|_{L^{p_*}(\bdH)}^{\,p-p_*}
  \int_{\bdH}
  U^{p_*-2}|h|^2\,dy.
  \label{eq:nonlinear-gap-q-greater-two}
 \end{align}
\end{proposition}

\begin{proof}
 We prove the two cases separately, but the compactness argument
 is the same.

\noindent\textbf{(1). Assume first that \(1<p<2\) and \(p_*\le2\).}
 
 Assume for contradiction that the first assertion is false.  Then there exists
 a sequence \(h_k\) satisfying
 \[
  \|\nabla h_k\|_{L^p(\Hhalf)}\longrightarrow0
 \]
 and the orthogonality conditions
 \eqref{eq:nonlinear-gap-orthogonality-q-less-two}, but such that
 \begin{align}
  \mathcal G_U[h_k]
  &<
  \left[
  (p_*-1)\ST^p+\lambda_{\mathrm T}
  \right]
  \|U\|_{L^{p_*}(\bdH)}^{\,p-p_*}
  \int_{\bdH}
  \frac{(U+C_0|h_k|)^{p_*}}{U^2+|h_k|^2}|h_k|^2\,dy.
  \label{eq:contradiction-nonlinear-gap}
 \end{align}

 Define
 \[
  \varepsilon_k^2
  :=
  \int_{\Hhalf}
  \bigl(|\nabla U|+|\nabla h_k|\bigr)^{p-2}
  |\nabla h_k|^2\,dx,
  \qquad
  \phi_k:=\frac{h_k}{\varepsilon_k}.
 \]
 Lemma~\ref{lem:gradient-remainder-coercivity} shows that
 \[
  \mathcal G_U[h_k]\simeq\varepsilon_k^2.
 \]
 Lemma~\ref{lem:normalized-perturbation-compactness} yields $\varepsilon_k\rightarrow0$ and, after passing
 to a subsequence,
 \[
  \phi_k\rightharpoonup\phi
  \quad\text{in }\dot{W}^{1,p}(\Hhalf),
 \]
 with
 \[
  \phi\perp T_U\MT.
 \]

 By lower semicontinuity of the nonlinear gradient remainder,
 \begin{equation}
  \label{eq:gradient-liminf-nonlinear}
  \liminf_{k\to\infty}
  \frac{\mathcal G_U[h_k]}{\varepsilon_k^2}
  \ge
  \int_{\Hhalf}
  \mathcal A_U\nabla\phi\cdot\nabla\phi\,dx.
 \end{equation}
 Indeed, writing \(h_k=\varepsilon_k\phi_k\), define
 \[
  f_k(x,Y)
  :=
  \frac{2}{p\varepsilon_k^2}
  \mathscr G_p(\nabla U(x),\varepsilon_kY).
 \]
 By \eqref{eq:nonlinear-gradient-form},
 \[
  \frac{\mathcal G_U[h_k]}{\varepsilon_k^2}
  =
  \int_{\Hhalf}f_k(x,\nabla\phi_k)\,dx.
 \]
 For every \(x\), the map \(Y\mapsto f_k(x,Y)\) is nonnegative.  Fix \(R>0\) and set
 \(B_R^+:=B_R\cap\Hhalf\).  Since \(\nabla U\) is smooth and bounded
 away from zero on \(\overline{B_R^+}\), Taylor's formula gives,
 for every \(R, M<\infty\),
 \[
  \sup_{\substack{x\in\overline{B_R^+}\\ \varepsilon_k^{1/4}|Y|\le M|\nabla U(x)|}}
  \left|
  f_k(x,Y)-\mathcal A_U(x)Y\cdot Y
  \right|
  \longrightarrow0.
 \]
 For fixed \(M>0\), set
 \[
  E_k:=
  \{\varepsilon_k^{1/4}|\nabla\phi_k|\le M|\nabla U|\}.
 \]
 Then
\begin{align}
  \liminf_{k\to\infty}
  \int_{B_R^+}f_k(x,\nabla\phi_k)\,dx
  \ge&
  \liminf_{k\to\infty}\int_{B_R^+}
   f_k(x,\chi_{E_k}\nabla\phi_k)\,dx\nonumber\\
   =&\liminf_{k\to\infty}\int_{B_R^+\cap E_k}
   \mathcal A_U(x)\nabla\phi_k\cdot \nabla\phi_k\,dx.
\end{align}
 Since
 \[
  |B_R^+\setminus E_k|
  \le C\int_{B_R^+\setminus E_k}|\nabla U|^p\,dx
  \le C\varepsilon_k^{p/4}
  \int_{B_R^+\setminus E_k}|\nabla\phi_k|^p\,dx
  \le C\varepsilon_k^{p/4},
 \]
 we have \(|B_R^+\setminus E_k|\to0\) and
 \(\chi_{E_k}\nabla\phi_k\rightharpoonup\nabla\phi\) weakly in
 \(L^p(B_R^+)\).  The lower semicontinuity theorem for convex integral
 functionals \cite[Theorem~3]{Ioffe1977}, applied on \(B_R^+\), therefore yields
 \[
  \liminf_{k\to\infty}
  \int_{B_R^+}f_k(x,\nabla\phi_k)\,dx
  \ge
  \int_{B_R^+}
  \mathcal A_U\nabla\phi\cdot\nabla\phi\,dx.
 \]
 Since \(f_k\ge0\), for every \(R>0\)
 \[
  \liminf_{k\to\infty}
  \frac{\mathcal G_U[h_k]}{\varepsilon_k^2}
  \ge
  \int_{B_R^+}
  \mathcal A_U\nabla\phi\cdot\nabla\phi\,dx.
 \]
 Finally, the integrand on the right is nonnegative and
 \(B_R^+\uparrow\Hhalf\).  Letting \(R\to\infty\) and applying the
 monotone convergence theorem proves
 \eqref{eq:gradient-liminf-nonlinear}.  Since the left-hand side is
 bounded along the contradiction sequence, the right-hand side is
 finite.  The equivalence \eqref{eq:energy-equivalence} then implies
 \(\phi\in\dot{W}^{1,2}(\Hhalf;|\nabla U|^{p-2})\).
 On the other hand, \eqref{eq:nonlinear-boundary-limsup} in
 Lemma~\ref{lem:normalized-perturbation-compactness} gives
 \begin{equation}
  \label{eq:boundary-limsup-nonlinear}
  \limsup_{k\to\infty}
  \frac{1}{\varepsilon_k^2}
  \int_{\bdH}
  \frac{(U+C_0|h_k|)^{p_*}}{U^2+|h_k|^2}|h_k|^2\,dy
  \le
  \int_{\bdH}
  U^{p_*-2}|\phi|^2\,dy.
 \end{equation}
 Hence \eqref{eq:contradiction-nonlinear-gap} implies
 \begin{align}
  \int_{\Hhalf}
  \mathcal A_U\nabla\phi\cdot\nabla\phi\,dx
  \le
  \left[
  (p_*-1)\ST^p+\lambda_{\mathrm T}
  \right]
  \|U\|_{L^{p_*}(\bdH)}^{\,p-p_*}
  \int_{\bdH}
  U^{p_*-2}|\phi|^2\,dy.
  \label{eq:limit-violates-gap}
 \end{align}

 The limit \(\phi\) is nontrivial.  Indeed,
 Lemma~\ref{lem:gradient-remainder-coercivity} gives
 \(\mathcal G_U[h_k]\ge c\varepsilon_k^2\); if \(\phi=0\), then
 \eqref{eq:boundary-limsup-nonlinear} and
 \eqref{eq:contradiction-nonlinear-gap} would force
 \(\limsup_k\mathcal G_U[h_k]/\varepsilon_k^2\le0\), a contradiction.
 Therefore
 \eqref{eq:limit-violates-gap} contradicts the linear spectral gap
 \eqref{eq:linear-spectral-gap-standard}.

\noindent\textbf{(2). Assume now that \(p_*>2\).}

 If the second assertion is false, there exists \(h_k\to0\) in
 \(\dot{W}^{1,p}(\Hhalf)\), satisfying the orthogonality conditions,
 such that
 \begin{align}
  \mathcal G_U[h_k]
  &<
  \left[
  (p_*-1)\ST^p+\lambda_{\mathrm T}
  \right]
  \|U\|_{L^{p_*}(\bdH)}^{\,p-p_*}
  \int_{\bdH}
  U^{p_*-2}|h_k|^2\,dy.
  \label{eq:q-greater-two-contradiction}
 \end{align}
 Normalize \(h_k\) by the nonlinear energy scale
 \(\varepsilon_k\) as in \eqref{eq:epsilon-k-definition}, and set
 \(\phi_k:=h_k/\varepsilon_k\).
 
 If \(1<p<2\),
 Lemma~\ref{lem:normalized-perturbation-compactness} gives, after passing to
 a subsequence,
 \[
  \phi_k\rightharpoonup\phi
  \quad\text{in }\dot{W}^{1,p}(\Hhalf)
 \]
 and
 \[
  \phi_k\to\phi
  \quad\text{in }L^2(\bdH,U^{p_*-2}dy).
 \]
 
 If \(p\ge2\), we have
 \begin{align*}
  \varepsilon_k^2
  &\le
  C\int_{\Hhalf}
  \left(
  |\nabla U|^{p-2}|\nabla h_k|^2+|\nabla h_k|^p
  \right)\,dx\\
  &\le
  C\|\nabla U\|_{L^p(\Hhalf)}^{p-2}
  \|\nabla h_k\|_{L^p(\Hhalf)}^2
  +C\|\nabla h_k\|_{L^p(\Hhalf)}^p
  \longrightarrow0.
 \end{align*}
 Moreover, by the definition of $\varepsilon_k$,
 \[
 \int_{\Hhalf}
  |\nabla U|^{p-2}|\nabla \phi_k|^2
  \,dx\le 1.
 \]
 Therefore, after passing to a
 subsequence,
 \[
  \phi_k\rightharpoonup\phi
  \quad\text{in }
  \dot{W}^{1,2}(\Hhalf;|\nabla U|^{p-2}),
  \qquad
  \phi_k\to\phi
  \quad\text{in }L^2(\bdH,U^{p_*-2}dy)
 \]
 by Theorem~\ref{thm:weighted-trace-compact}.
 The limit remains orthogonal to \(T_U\MT\).

 When \(1<p<2\), the gradient lower bound is precisely
 \eqref{eq:gradient-liminf-nonlinear}.  We now establish the same bound
 when \(p\ge2\).  
 Define
 \[
  f_k(x,Y):=
  \frac{2}{p\varepsilon_k^2}
  \mathscr G_p(\nabla U(x),\varepsilon_kY).
 \]
 Fix \(R>0\), write \(B_R^+:=B_R\cap\Hhalf\), and set
 \[
  E_{k,R}:=
  \{x\in B_R^+:
  \varepsilon_k^{1/4}|\nabla\phi_k(x)|\le|\nabla U(x)|\}.
 \]
 Since \(|\nabla U|\) is bounded above and away from zero on
 \(\overline{B_R^+}\), the local \(L^2\)-norms of
 \(\nabla\phi_k\) are uniformly bounded.  Consequently,
 \[
  |B_R^+\setminus E_{k,R}|
  \le C_R\varepsilon_k^{1/2}
  \int_{B_R^+}|\nabla\phi_k|^2\,dx
  \longrightarrow0.
 \]
 It follows that
 \[
  \chi_{E_{k,R}}\nabla\phi_k
  \rightharpoonup\nabla\phi
  \quad\text{in }L^2(B_R^+).
 \]
  Moreover, Taylor's formula,
 uniformly on \(\overline{B_R^+}\), gives
 \[
  \sup_{\substack{x\in\overline{B_R^+}\\
  \varepsilon_k^{1/4}|Y|\le|\nabla U(x)|}}
  \left|f_k(x,Y)-\mathcal A_U(x)Y\cdot Y\right|
  \longrightarrow0.
 \]
 Since \(f_k\ge0\) and \(f_k(x,0)=0\), weak lower
 semicontinuity of the quadratic integral on \(B_R^+\) yields
 \begin{align*}
  \liminf_{k\to\infty}
  \frac{\mathcal G_U[h_k]}{\varepsilon_k^2}
  &\ge
  \liminf_{k\to\infty}
  \int_{E_{k,R}}f_k(x,\nabla\phi_k)\,dx\\
  &=
  \liminf_{k\to\infty}
  \int_{E_{k,R}}
  \mathcal A_U\nabla\phi_k\cdot\nabla\phi_k\,dx\\
  &\ge
  \int_{B_R^+}
  \mathcal A_U\nabla\phi\cdot\nabla\phi\,dx.
 \end{align*}
 Letting \(R\to\infty\) proves the desired bound.  Thus, in both
 ranges of \(p\),
 \begin{equation}
  \label{eq:gradient-liminf-q-greater-two}
  \liminf_{k\to\infty}
  \frac{\mathcal G_U[h_k]}{\varepsilon_k^2}
  \ge
  \int_{\Hhalf}
  \mathcal A_U\nabla\phi\cdot\nabla\phi\,dx.
 \end{equation}

 Dividing \eqref{eq:q-greater-two-contradiction} by
 \(\varepsilon_k^2\), using
 \eqref{eq:gradient-liminf-q-greater-two} and
 the strong weighted trace convergence, gives
 \begin{align}
  \int_{\Hhalf}
  \mathcal A_U\nabla\phi\cdot\nabla\phi\,dx
  \le
  \left[
  (p_*-1)\ST^p+\lambda_{\mathrm T}
  \right]
  \|U\|_{L^{p_*}(\bdH)}^{\,p-p_*}
  \int_{\bdH}
  U^{p_*-2}|\phi|^2\,dy.
 \end{align}
 The limit \(\phi\) is nontrivial.  If \(\phi=0\), then the strong convergence
 in \(L^2(\bdH,U^{p_*-2}dy)\) makes the right-hand side of
 \eqref{eq:q-greater-two-contradiction}, after division by
 \(\varepsilon_k^2\), tend to zero, while
 \(\mathcal G_U[h_k]\ge c\varepsilon_k^2\).  This is impossible.
 Hence the last displayed inequality contradicts
 Corollary~\ref{cor:linear-spectral-gap}.
\end{proof}

\begin{corollary}[Uniform nonlinear spectral gap]
 \label{cor:uniform-nonlinear-gap}
 The preceding estimates hold uniformly for \(v\in\MT^+\)
 with
 \[
  \frac12\le\|v\|_{L^{p_*}(\bdH)}\le\frac32.
 \]
 More precisely, after decreasing the smallness threshold if necessary,
 the estimate \eqref{eq:nonlinear-spectral-gap-q-less-two} holds with
 \(U\) replaced by \(v\) when \(p_*\le2\), and the estimate
 \eqref{eq:nonlinear-gap-q-greater-two} holds with \(U\) replaced by
 \(v\) when \(p_*>2\).  In both cases, the coefficient is
 \[
  [(p_*-1)\ST^p+ \lambda_{\mathrm T}]\|v\|_{L^{p_*}(\bdH)}^{p-p_*}.
 \]
 For later use, we write
 \begin{equation}
  \label{eq:scaled-spectral-gap-constant}
  \lambda_{\mathrm T}(v)
  :=\lambda_{\mathrm T}
  \|v\|_{L^{p_*}(\bdH)}^{p-p_*}.
 \end{equation}
\end{corollary}

\begin{proof}
 Translate and critically dilate \(v\) to \(U\), and then multiply by
 its boundary \(L^{p_*}\)-scale.  The quotient, the tangent
 orthogonality, and the nonlinear forms transform homogeneously under
 these operations.  Since the boundary norm of \(v\) is restricted to
 a compact interval, the smallness thresholds and constants remain
 uniform.
\end{proof}

\section{Proof of Theorem~\ref{thm:main}}
\label{sec:proof-main}

We begin with concentration--compactness.  The following statement is a
consequence of Lions' limit-case principle for Sobolev and trace
inequalities \cite{Lions1985,Lions1985Part2}; see, in particular,
\cite[Theorem~2.3]{Lions1985Part2}.

\begin{theorem}[Lions' concentration--compactness theorem]
 \label{thm:almost-extremizer-compactness}
 Let \((u_k)\subset \dot{W}^{1,p}(\Hhalf)\) satisfy
 \begin{equation}
  \label{eq:almost-extremizer-normalization}
  \|u_k\|_{L^{p_*}(\bdH)}=1
 \end{equation}
 and
 \begin{equation}
  \label{eq:almost-extremizer-energy}
  \|\nabla u_k\|_{L^p(\Hhalf)}^p
  \longrightarrow
  \ST^p.
 \end{equation}
 Then there exist bubbles \(v_k\in\MT\) such that
 \begin{equation}
  \label{eq:distance-to-manifold-goes-zero}
  \|\nabla(u_k-v_k)\|_{L^p(\Hhalf)}
  \longrightarrow0.
 \end{equation}
\end{theorem}

\begin{remark}[Qualitative stability]
 \label{rem:qualitative-stability}
 For every \(\varepsilon>0\), there exists
 \(\delta=\delta(n,p,\varepsilon)>0\) such that, if
 \(u\in\dot{W}^{1,p}(\Hhalf)\), \(u\not\equiv0\) on \(\bdH\), and
 \begin{equation}
  \label{eq:small-deficit-qualitative}
  \delta_{\mathrm T}(u)\le\delta,
 \end{equation}
 then
 \begin{equation}
  \label{eq:qualitative-distance-small}
  d_{\mathrm T}(u,\MT)\le\varepsilon.
 \end{equation}
 Equivalently,
 \begin{equation}
  \label{eq:deficit-controls-closeness-qualitatively}
  \delta_{\mathrm T}(u_k)\longrightarrow0
  \quad\Longrightarrow\quad
  d_{\mathrm T}(u_k,\MT)\longrightarrow0.
 \end{equation}
 Indeed, normalize \(\|u_k\|_{L^{p_*}(\bdH)}=1\).  Then
 \(\delta_{\mathrm T}(u_k)\to0\) implies
 \[
  \|\nabla u_k\|_{L^p(\Hhalf)}^p\longrightarrow\ST^p.
 \]
 Theorem~\ref{thm:almost-extremizer-compactness} now gives convergence
 to \(\MT\).
\end{remark}

As in \cite[Lemma 4.1]{FigalliZhang2022}, we first establish the modulation result that supplies the exact
tangent orthogonality needed in the proof of the main theorem.  For
\(v\in\MT^+\), define
\begin{equation}
 \label{eq:auxiliary-modulation-functional}
 \mathcal F_u(v)
 :=
 \frac1{p_*}
 \int_{\bdH}v^{p_*}\,dy
 -
 \frac1{p_*-1}
 \int_{\bdH}v^{p_*-1}u\,dy.
\end{equation}

\begin{proposition}[Orthogonal modulation near a bubble]
 \label{prop:uniform-modulation}
 Assume that
 \[
  \|u\|_{L^{p_*}(\bdH)}=1,
  \qquad
  \|\nabla(u-\widehat v)\|_{L^p(\Hhalf)}\le\widehat\eta,
 \]
 for some \(\widehat v\in\MT^+\) satisfying
 \[
  \frac34
  \le
  \|\widehat v\|_{L^{p_*}(\bdH)}
  \le
  \frac43.
 \]
 There exist \(\eta'=\eta'(n,p)>0\) and a modulus of continuity
 \(\omega:\R_+\to\R_+\), with \(\omega(t)\to0\) as \(t\to0\), such
 that, if \(\widehat\eta\le\eta'\), then there exists \(v\in\MT^+\)
 for which
 \begin{equation}
  \label{eq:uniform-modulation-orthogonality}
  \int_{\bdH}v^{p_*-2}(u-v)Z\,dy=0
  \qquad
  \forall Z\in T_v\MT,
 \end{equation}
 and
 \begin{equation}
  \label{eq:modulation-modulus-bound}
  \|\nabla(u-v)\|_{L^p(\Hhalf)}
  \le
  \omega(\widehat\eta).
 \end{equation}
\end{proposition}

\begin{proof}
 We minimize \(\mathcal F_u\) over \(\MT^+\).  First suppose that
 \(u=\widehat v\).  For \(w\in\MT^+\), H\"older's inequality gives
 \begin{align}
  \mathcal F_{\widehat v}(w)
  &\ge
  \frac1{p_*}\|w\|_{L^{p_*}(\bdH)}^{p_*}
  -
  \frac{\|\widehat v\|_{L^{p_*}(\bdH)}}{p_*-1}
  \|w\|_{L^{p_*}(\bdH)}^{p_*-1}
  \notag\\
  &\ge
  -\frac{\|\widehat v\|_{L^{p_*}(\bdH)}^{p_*}}
  {p_*(p_*-1)}
  =
  \mathcal F_{\widehat v}(\widehat v).
  \label{eq:modulation-unique-minimizer}
 \end{align}
 The scalar function
 \[
  s\longmapsto
  \frac{s^{p_*}}{p_*}
  -
  \frac{\|\widehat v\|_{L^{p_*}(\bdH)}}{p_*-1}s^{p_*-1}
 \]
 is uniquely minimized at
 \(s=\|\widehat v\|_{L^{p_*}(\bdH)}\).  Equality in H\"older's
 inequality then shows that equality in
 \eqref{eq:modulation-unique-minimizer} holds only for
 \(w=\widehat v\).  Thus \(\widehat v\) is the unique minimizer of
 \(\mathcal F_{\widehat v}\) on \(\MT^+\).

 We next exclude loss of compactness of the minimizing parameters.
 Write \(w=V_{b,\lambda,\xi}\in\MT^+\).  By H\"older's inequality,
 \[
  \left|\int_{\bdH}w^{p_*-1}u\,dy\right|
  \le
  \|w\|_{L^{p_*}(\bdH)}^{p_*-1}
  \|u\|_{L^{p_*}(\bdH)}.
 \]
 Since the boundary norm is proportional to \(b\), it follows that
 \(\mathcal F_u(w)\to0\) when \(b\to0\), whereas
 \(\mathcal F_u(w)\to+\infty\) when \(b\to+\infty\).  If \(b\) stays
 in a compact subset of \((0,\infty)\), while
 \(\lambda\to0\), \(\lambda\to+\infty\), or the center escapes to
 infinity, then \(w^{p_*-1}\rightharpoonup0\) in
 \(L^{p_*'}(\bdH)\).  Consequently,
 \[
  \liminf \mathcal F_u(w)
  =
  \liminf\frac1{p_*}\|w\|_{L^{p_*}(\bdH)}^{p_*}
  \ge0
 \]
 in all these noncompact parameter regimes.

 On the other hand,
 \(\mathcal F_{\widehat v}(\widehat v)
 =-\|\widehat v\|_{L^{p_*}}^{p_*}/[p_*(p_*-1)]\), which is bounded
 away from zero uniformly for
 \(3/4\le\|\widehat v\|_{L^{p_*}}\le4/3\).  The trace inequality
 shows that \(\mathcal F_u(\widehat v)<0\) uniformly when
 \(\|\nabla(u-\widehat v)\|_{L^p}\) is sufficiently small.  Thus a
 minimizing sequence cannot enter any of the preceding noncompact
 regimes, and the minimum is attained.  The uniqueness proved above,
 followed by a contradiction argument, shows that every minimizer
 converges to \(\widehat v\) in \(\dot{W}^{1,p}\) as
 \(\|\nabla(u-\widehat v)\|_{L^p}\to0\), uniformly under the stated
 bounds on \(\|\widehat v\|_{L^{p_*}}\).  Consequently, for
 \(\widehat\eta\le\eta'\), the minimum is attained at some
 \(v\in\MT^+\) and
 \[
  \|\nabla(u-v)\|_{L^p(\Hhalf)}
  \le
  \omega(\widehat\eta)
 \]
 for a modulus \(\omega\) depending only on \(n,p\).

 Finally, for \(Z\in T_v\MT\), choose a smooth curve
 \(v_\tau\subset\MT^+\) with \(v_0=v\) and \(\dot v_0=Z\).
 The minimality of \(v\) and differentiation of
 \eqref{eq:auxiliary-modulation-functional} yield
 \begin{align*}
  0
  &=
  \left.\frac{d}{d\tau}\right|_{\tau=0}\mathcal F_u(v_\tau)\\
  &=
  \int_{\bdH}v^{p_*-1}Z\,dy
  -
  \int_{\bdH}v^{p_*-2}Zu\,dy\\
  &=
  -\int_{\bdH}v^{p_*-2}(u-v)Z\,dy.
 \end{align*}
 This proves \eqref{eq:uniform-modulation-orthogonality} and completes
 the proof.
\end{proof}

\begin{claim}[Global-to-local reduction]
 \label{prop:global-to-local-reduction}
 Suppose that there exist constants
 \[
  \varepsilon_0>0,
  \qquad
  c_0>0,
 \]
 such that every \(u\in \dot{W}^{1,p}(\Hhalf)\) with
 \(u|_{\bdH}\not\equiv0\) and satisfying
 \begin{equation}
  \label{eq:local-regime-assumption}
  d_{\mathrm T}(u,\MT)\le\varepsilon_0
 \end{equation}
 obeys
 \begin{equation}
  \label{eq:local-stability-assumption}
  \delta_{\mathrm T}(u)
  \ge
  c_0
  d_{\mathrm T}(u,\MT)^{\max\{2,p\}}.
 \end{equation}
 Then there exists \(c=c(n,p,c_0,\varepsilon_0)>0\) such that
 \begin{equation}
  \label{eq:global-stability-from-local}
  \delta_{\mathrm T}(u)
  \ge
  c
  d_{\mathrm T}(u,\MT)^{\max\{2,p\}}
 \end{equation}
 for every \(u\in \dot{W}^{1,p}(\Hhalf)\) with
 \(u|_{\bdH}\not\equiv0\).
\end{claim}

\begin{proof}
 By the triangle inequality, choosing
 \(v\in\MT\) with
 \[
  \|\nabla v\|_{L^p(\Hhalf)}
  \le
  2\|\nabla u\|_{L^p(\Hhalf)}
 \]
 gives a universal upper bound
 \begin{equation}
  \label{eq:distance-universal-bound}
  d_{\mathrm T}(u,\MT)\le C(n,p).
 \end{equation}

 By Remark~\ref{rem:qualitative-stability}, there exists \(\delta_0>0\)
 such that
 \[
  \delta_{\mathrm T}(u)<\delta_0
  \quad\Longrightarrow\quad
  d_{\mathrm T}(u,\MT)<\varepsilon_0.
 \]
 In this case, \eqref{eq:local-stability-assumption} applies.  If
 \(\delta_{\mathrm T}(u)\ge\delta_0\), then
 \eqref{eq:distance-universal-bound} yields
 \[
  \delta_{\mathrm T}(u)
  \ge
  \frac{\delta_0}{C(n,p)^{\max\{2,p\}}}
  d_{\mathrm T}(u,\MT)^{\max\{2,p\}}.
 \]
 Taking the smaller of the two constants proves the claim.
\end{proof}

\begin{proof}[Proof of Theorem~\ref{thm:main}]
 By homogeneity, we may normalize
 \begin{equation}
  \label{eq:main-proof-normalization}
  \|u\|_{L^{p_*}(\bdH)}=1.
 \end{equation}
 By Claim~\ref{prop:global-to-local-reduction}, it is enough to work
 when \(\delta_{\mathrm T}(u)\) is sufficiently small.  Fix
 \(\widehat\eta>0\), to be chosen below.  Remark~\ref{rem:qualitative-stability}
 provides \(\widehat v\in\MT\) such that
 \[
  \|\nabla(u-\widehat v)\|_{L^p(\Hhalf)}
  \le\widehat\eta.
 \]
 If \(\widehat v<0\), we replace both \(u\) and \(\widehat v\) by
 their negatives; this leaves the deficit and the distance unchanged.
 Thus we may assume \(\widehat v\in\MT^+\).  The sharp trace inequality
 and the reverse triangle inequality give, after decreasing
 \(\widehat\eta\),
 \[
  \frac34
  \le
  \|\widehat v\|_{L^{p_*}(\bdH)}
  \le
  \frac43.
 \]
 Up to a tangential translation and a critical dilation, we may write
 \(\widehat v=V_{a,1,0}\).  Proposition~\ref{prop:uniform-modulation}
 then gives \(v\in\MT^+\) such that, with \(h:=u-v\),
 \begin{equation}
  \label{eq:main-proof-orthogonality}
  \int_{\bdH}v^{p_*-2}hZ\,dy=0
  \qquad\forall Z\in T_v\MT,
 \end{equation}
 and
 \[
  \|\nabla h\|_{L^p(\Hhalf)}
  \le\omega(\widehat\eta).
 \]
 If \(h=0\), the conclusion is immediate.  Otherwise, set
 \[
  \varepsilon:=\|\nabla h\|_{L^p(\Hhalf)},
  \qquad
  \varphi:=\frac{h}{\varepsilon},
  \qquad
  \|\nabla\varphi\|_{L^p(\Hhalf)}=1.
 \]
 A further decrease of \(\widehat\eta\), together with the trace
 inequality and \eqref{eq:main-proof-normalization}, gives
 \begin{equation}
  \label{eq:v-trace-norm-local-bounds}
  \frac12
  \le
  \|v\|_{L^{p_*}(\bdH)}
  \le
  \frac32.
 \end{equation}
 Hence the uniform nonlinear spectral gap applies to \(v\).

Since
\[
 \delta_{\mathrm T}(u)
 =
 \|\nabla u\|_{L^p(\Hhalf)}-\ST
\]
under the normalization \eqref{eq:main-proof-normalization}, for
\(\delta_{\mathrm T}(u)\) small,
\begin{equation}
 \label{eq:norm-energy-deficit-comparison-general}
 c\,\delta_{\mathrm T}(u)
 \le
 \int_{\Hhalf}|\nabla u|^p\,dx-\ST^p
 \le
 C\,\delta_{\mathrm T}(u).
\end{equation}
Thus, in the local regime, it is enough to prove
\begin{equation}
 \label{eq:target-energy-local}
 \int_{\Hhalf}|\nabla u|^p\,dx-\ST^p
 \ge
 c\,
 d_{\mathrm T}(u,\MT)^{\max\{2,p\}}.
\end{equation}

Since \(v\) is an extremal,
\begin{equation}
 \label{eq:v-energy-exact}
 \int_{\Hhalf}|\nabla v|^p\,dx
 =
 \ST^p\|v\|_{L^{p_*}(\bdH)}^p.
\end{equation}
The bulk expansion gives
\begin{align}
 \int_{\Hhalf}|\nabla u|^p\,dx
 &=
 \int_{\Hhalf}|\nabla v|^p\,dx
 +
 p\int_{\Hhalf}
 |\nabla v|^{p-2}\nabla v\cdot\nabla h\,dx
 \notag\\
 &\quad
 +
 \frac p2\mathcal G_v[h],
 \label{eq:bulk-energy-exact-expansion}
\end{align}
where \(\mathcal G_v[h]\) is defined in
\eqref{eq:general-bubble-gradient-form}.

Since \(p/p_*<1\), the map
\[
 s\longmapsto s^{p/p_*}
\]
is concave.  Hence, with
\(s=\int_{\bdH}v^{p_*}\,dy\),
\begin{equation}
 \label{eq:concavity-at-A-v}
 1-\|v\|_{L^{p_*}(\bdH)}^p
 \le
 \frac p{p_*}
 \|v\|_{L^{p_*}(\bdH)}^{p-p_*}
 \left(
 1-\int_{\bdH}v^{p_*}\,dy
 \right).
\end{equation}
Using
\[
 \int_{\bdH}|u|^{p_*}\,dy=1,
\]
we have
\begin{align}
 1-\int_{\bdH}v^{p_*}\,dy
 &=
 p_*\int_{\bdH}v^{p_*-1}h\,dy
 +
 \int_{\bdH}\mathscr B_{p_*}(v,h)\,dy,
 \label{eq:boundary-remainder-definition}
\end{align}

Combining \eqref{eq:v-energy-exact},
\eqref{eq:bulk-energy-exact-expansion}, and
\eqref{eq:concavity-at-A-v}, we obtain
\begin{align}
 \int_{\Hhalf}|\nabla u|^p\,dx-\ST^p
 &\ge
 p\int_{\Hhalf}
 |\nabla v|^{p-2}\nabla v\cdot\nabla h\,dx
 +
 \frac p2\mathcal G_v[h]
 \notag\\
 &\quad
 -
 \frac p{p_*}
 \ST^p\|v\|_{L^{p_*}(\bdH)}^{p-p_*}
 \left(
 p_*\int_{\bdH}v^{p_*-1}h\,dy
 +
 \int_{\bdH}\mathscr B_{p_*}(v,h)\,dy
 \right).
 \label{eq:deficit-before-cancellation}
\end{align}

The Euler--Lagrange equation for \(v\) gives
\begin{equation}
 \label{eq:EL-cancellation-main}
 \int_{\Hhalf}
 |\nabla v|^{p-2}\nabla v\cdot\nabla h\,dx
 =
 \ST^p\|v\|_{L^{p_*}(\bdH)}^{p-p_*}
 \int_{\bdH}v^{p_*-1}h\,dy.
\end{equation}
Therefore the first-order terms in
\eqref{eq:deficit-before-cancellation} cancel exactly, and
\begin{equation}
 \label{eq:master-deficit-lower-bound}
 \int_{\Hhalf}|\nabla u|^p\,dx-\ST^p
 \ge
 \frac p2\mathcal G_v[h]
 -
 \frac p{p_*}
 \ST^p\|v\|_{L^{p_*}(\bdH)}^{p-p_*}
 \int_{\bdH}\mathscr B_{p_*}(v,h)\,dy.
\end{equation}

It remains to bound the right-hand side of
\eqref{eq:master-deficit-lower-bound} from below.

\noindent\textbf{Case 1: \(p_*\le2\).}

Fix \(\kappa>0\).  The scalar inequality
\eqref{eq:boundary-remainder-q-less-two} gives
\begin{equation}
 \label{eq:R-q-upper-singular}
 \int_{\bdH}\mathscr B_{p_*}(v,h)\,dy
 \le
 \left(
 \frac{p_*(p_*-1)}2+\kappa
 \right)
 \int_{\bdH}
 \frac{(v+C_\kappa|h|)^{p_*}}{v^2+|h|^2}|h|^2\,dy.
\end{equation}
Substituting \eqref{eq:R-q-upper-singular} into
\eqref{eq:master-deficit-lower-bound}, we obtain
\begin{align}
 \int_{\Hhalf}|\nabla u|^p\,dx-\ST^p
 &\ge
 \frac p2\mathcal G_v[h]
 \notag\\
 &\quad
 -
 \frac p2
 \left(
 (p_*-1)\ST^p
 \|v\|_{L^{p_*}(\bdH)}^{p-p_*}
 +
  \frac{2\kappa}{p_*}
  \ST^p\|v\|_{L^{p_*}(\bdH)}^{p-p_*}
  \right)
 \notag\\
 &\qquad\times
 \int_{\bdH}
 \frac{(v+C_\kappa|h|)^{p_*}}{v^2+|h|^2}|h|^2\,dy.
 \label{eq:singular-before-gap}
\end{align}

By Corollary~\ref{cor:uniform-nonlinear-gap}, for sufficiently small
\(\varepsilon\),
\begin{align}
 \mathcal G_v[h]
 &\ge
 \left(
 (p_*-1)\ST^p
 \|v\|_{L^{p_*}(\bdH)}^{p-p_*}
 +
  \lambda_{\mathrm T}(v)
 \right)
 \int_{\bdH}
 \frac{(v+C_\kappa|h|)^{p_*}}{v^2+|h|^2}|h|^2\,dy.
 \label{eq:singular-nonlinear-gap-used}
\end{align}
Choose \(\kappa>0\), depending only on \(n,p\), sufficiently small
that, for every \(v\) satisfying
\eqref{eq:v-trace-norm-local-bounds},
\begin{equation}
 \label{eq:kappa-choice-singular}
 \frac{2\kappa}{p_*}\ST^p
 \|v\|_{L^{p_*}(\bdH)}^{p-p_*}
 \le
 \frac{\lambda_{\mathrm T}(v)}4.
\end{equation}
Using the uniform bounds on \(\|v\|_{L^{p_*}(\bdH)}\), the last two
displays imply
\begin{equation}
 \label{eq:singular-after-gap-preliminary}
 \int_{\Hhalf}|\nabla u|^p\,dx-\ST^p
 \ge
 c\,\mathcal G_v[h].
\end{equation}

It remains to convert the mixed gradient remainder into a quadratic
power of \(\varepsilon\).

\begin{claim}
\label{claim:mixed-remainder-lower-epsilon-two}
Let \(1<p<2\), let \(v\in\MT^+\) satisfy
\eqref{eq:v-trace-norm-local-bounds}, and let
\[
h=\varepsilon\varphi,
\qquad
\|\nabla\varphi\|_{L^p(\Hhalf)}=1.
\]
Then, for \(0<\varepsilon\le1\),
\begin{equation}
 \label{eq:mixed-remainder-lower-epsilon-two}
 \int_{\Hhalf}
 \min
 \left\{
 |\nabla h|^p,\,
 |\nabla v|^{p-2}|\nabla h|^2
 \right\}\,dx
 \ge
 c(n,p)\varepsilon^2.
\end{equation}
\end{claim}

\begin{proof}
 Define
 \begin{align}
  E
  &:=
  \left\{
  \varepsilon|\nabla\varphi|
  \ge
  |\nabla v|
  \right\},
  \label{eq:E-gradient-large}\\
  F
  &:=
  \left\{
  \varepsilon|\nabla\varphi|
  <
  |\nabla v|
  \right\}.
  \label{eq:F-gradient-small}
 \end{align}
 Then the integral in \eqref{eq:mixed-remainder-lower-epsilon-two} is
 \begin{align}
  &\varepsilon^p
  \int_E|\nabla\varphi|^p\,dx
  +
  \varepsilon^2
  \int_F
  |\nabla v|^{p-2}
  |\nabla\varphi|^2\,dx.
  \label{eq:R-split-E-F}
 \end{align}

 If
 \[
  \int_E|\nabla\varphi|^p\,dx\ge\frac12,
 \]
 then \eqref{eq:R-split-E-F} is bounded below by
 \(\frac12\varepsilon^p\ge\frac12\varepsilon^2\), since \(p<2\) and
 \(\varepsilon\le1\).  Otherwise,
 \[
  \int_F|\nabla\varphi|^p\,dx\ge\frac12.
 \]
 By H\"older's inequality,
 \begin{align}
  \int_F|\nabla\varphi|^p\,dx
  &=
  \int_F
  \left(
  |\nabla v|^{\frac{p-2}{2}}
  |\nabla\varphi|
  \right)^p
  |\nabla v|^{\frac{p(2-p)}2}
  \,dx
  \notag\\
  &\le
  \left(
  \int_F
  |\nabla v|^{p-2}
  |\nabla\varphi|^2\,dx
  \right)^{p/2}
  \left(
  \int_F|\nabla v|^p\,dx
  \right)^{(2-p)/2}.
 \label{eq:holder-mixed-remainder}
\end{align}
Since \(v\in\MT\) and its trace norm is uniformly bounded above and
below,
\[
 \|\nabla v\|_{L^p(\Hhalf)}
 =
 \ST\|v\|_{L^{p_*}(\bdH)}
 \simeq_{n,p}1.
\]
It follows that
\[
 \int_F
 |\nabla v|^{p-2}
 |\nabla\varphi|^2\,dx
 \ge
 c(n,p).
\]
This proves \eqref{eq:mixed-remainder-lower-epsilon-two}.
\end{proof}

Combining \eqref{eq:singular-after-gap-preliminary},
Lemma~\ref{lem:gradient-remainder-coercivity}, and
\eqref{eq:mixed-remainder-lower-epsilon-two}, we conclude that
\begin{equation}
 \label{eq:first-regime-final}
 \int_{\Hhalf}|\nabla u|^p\,dx-\ST^p
 \ge
 c\varepsilon^2
 \qquad
 \text{if }
 1<p\le\frac{2n}{n+1}.
\end{equation}

\noindent\textbf{Case 2: \(p_*>2\) and \(p<2\).}

It follows from \eqref{eq:boundary-remainder-q-greater-two} that, for every \(\kappa>0\),
\begin{align}
 \int_{\bdH}\mathscr B_{p_*}(v,h)\,dy
 &\le
 \left(
 \frac{p_*(p_*-1)}2+\kappa
 \right)
 \int_{\bdH}v^{p_*-2}|h|^2\,dy
 \notag\\
 &\quad
 +
 C_\kappa
 \int_{\bdH}|h|^{p_*}\,dy.
 \label{eq:R-q-upper-q-greater-two}
\end{align}
Using this in \eqref{eq:master-deficit-lower-bound}, we find
\begin{align}
 &\int_{\Hhalf}|\nabla u|^p\,dx-\ST^p
 \ge
 \frac p2\mathcal G_v[h]-
 C\|h\|_{L^{p_*}(\bdH)}^{p_*}
 \notag\\
 &
 -
 \frac p2
 \left(
 (p_*-1)\ST^p
 \|v\|_{L^{p_*}(\bdH)}^{p-p_*}
 +
 \frac{2\kappa}{p_*}
 \ST^p\|v\|_{L^{p_*}(\bdH)}^{p-p_*}
 \right)
 \int_{\bdH}v^{p_*-2}|h|^2\,dy.
 \label{eq:second-regime-before-gap}
\end{align}

By Corollary~\ref{cor:uniform-nonlinear-gap},
after choosing \(\kappa\) sufficiently small,
\begin{equation}
 \label{eq:second-regime-after-gap}
 \int_{\Hhalf}|\nabla u|^p\,dx-\ST^p
 \ge
 c\mathcal G_v[h]
 -
 C\|h\|_{L^{p_*}(\bdH)}^{p_*}.
\end{equation}
Since \(p<2\), Lemma~\ref{lem:gradient-remainder-coercivity} and
Claim~\ref{claim:mixed-remainder-lower-epsilon-two} imply
\begin{equation}
 \label{eq:G-quadratic-second-regime}
 \mathcal G_v[h]
 \ge
 c\varepsilon^2.
\end{equation}
The sharp trace inequality gives
\begin{equation}
 \label{eq:h-trace-q-controlled}
 \|h\|_{L^{p_*}(\bdH)}^{p_*}
 \le
 \ST^{-p_*}
 \|\nabla h\|_{L^p(\Hhalf)}^{p_*}
 =
 \ST^{-p_*}\varepsilon^{p_*}.
\end{equation}
Because \(p_*>2\),
\[
 \varepsilon^{p_*}=o(\varepsilon^2)
 \qquad\text{as }\varepsilon\to0.
\]
Thus, for sufficiently small \(\varepsilon\),
\begin{equation}
 \label{eq:second-regime-final}
 \int_{\Hhalf}|\nabla u|^p\,dx-\ST^p
 \ge
 c\varepsilon^2.
\end{equation}

\noindent\textbf{Case 3: \(p\ge2\).}

As in the preceding case, \eqref{eq:boundary-remainder-q-greater-two} and Corollary~\ref{cor:uniform-nonlinear-gap} imply
\begin{equation}
 \label{eq:third-regime-after-gap}
 \int_{\Hhalf}|\nabla u|^p\,dx-\ST^p
 \ge
 c\mathcal G_v[h]
 -
 C\|h\|_{L^{p_*}(\bdH)}^{p_*}.
\end{equation}
For \(p\ge2\), Lemma~\ref{lem:gradient-remainder-coercivity} yields
\begin{align}
 \mathcal G_v[h]
 &\ge
 c
 \int_{\Hhalf}
 \left(
 |\nabla v|^{p-2}|\nabla h|^2
 +
 |\nabla h|^p
 \right)\,dx
 \notag\\
 &\ge
 c\|\nabla h\|_{L^p(\Hhalf)}^p
 =
 c\varepsilon^p.
 \label{eq:G-controls-epsilon-p}
\end{align}
Moreover, by the Sobolev trace inequality,
\begin{equation}
 \label{eq:third-regime-boundary-error}
 \|h\|_{L^{p_*}(\bdH)}^{p_*}
 \le
 C\varepsilon^{p_*}.
\end{equation}
Since \(p_*>p\),
\[
 \varepsilon^{p_*}=o(\varepsilon^p)
 \qquad\text{as }\varepsilon\to0.
\]
Consequently,
\begin{equation}
 \label{eq:third-regime-final}
 \int_{\Hhalf}|\nabla u|^p\,dx-\ST^p
 \ge
 c\varepsilon^p.
\end{equation}

Combining
\eqref{eq:first-regime-final},
\eqref{eq:second-regime-final}, and
\eqref{eq:third-regime-final}, we have proved
\begin{equation}
 \label{eq:all-regimes-epsilon}
 \int_{\Hhalf}|\nabla u|^p\,dx-\ST^p
 \ge
 c
 \varepsilon^{\max\{2,p\}},
\end{equation}
when \(\varepsilon\) is sufficiently small.  By the definition of the
distance and the normalization \eqref{eq:main-proof-normalization},
\[
 d_{\mathrm T}(u,\MT)
 \le
 \frac{\varepsilon}{\|\nabla u\|_{L^p(\Hhalf)}}
 \le
 \frac{\varepsilon}{\ST}.
\]
Together with \eqref{eq:norm-energy-deficit-comparison-general}, this
proves the local stability estimate.  Claim~\ref{prop:global-to-local-reduction}
then yields
 \begin{equation}
  \label{eq:global-main-final}
  \delta_{\mathrm T}(u)
  \ge
  c(n,p)
  d_{\mathrm T}(u,\MT)^{\max\{2,p\}}
 \end{equation}
 for every \(u\in \dot{W}^{1,p}(\Hhalf)\) with
 \(u|_{\bdH}\not\equiv0\), as required.
\end{proof}

\section{Declaration of generative AI and AI-assisted technologies in the manuscript preparation process}

\noindent\textbf{Statement}: During the preparation of this work, the authors used ChatGPT (OpenAI) to assist in converting their proof manuscripts into a preliminary LaTeX draft and to identify potential grammatical and logical errors in the manuscript. After using this tool, the authors reviewed and edited the content as needed and take full responsibility for the content of the published article.

\section{Acknowledgments}

The first and third authors were supported by the National Natural Science Foundation of China [grant number: 2025YFA1017601].

\raggedbottom

\end{document}